\documentclass[12pt,amsfonts]{amsart} 
\usepackage[utf8x]{inputenc}
\usepackage{tikz} 
\usepackage[all]{xy}
\usepackage[babel=true]{csquotes}
\usepackage{amsmath,amsthm, amssymb,amsfonts}
\usepackage[hmargin=1in,vmargin=1in]{geometry}
\numberwithin{equation}{subsection} 
\xyoption{arc}
\usepackage{eucal}
\usepackage{indentfirst}
\usepackage{mathrsfs}
\usepackage{verbatim}
\usepackage{graphicx}
\usepackage{hyperref}
\usepackage{enumitem} 
\usepackage{extpfeil} 
\usepackage[T1]{fontenc}
\newtheorem{thm}{Theorem}[section]
\newtheorem{prop}[thm]{Proposition}
\newtheorem{lem}[thm]{Lemma}

\newtheorem{cor}[thm]{Corollary}

\newtheorem{df}[thm]{Definition}

\theoremstyle{definition}
\newtheorem{rmk}[thm]{Remark}

\newtheorem*{note}{Note}

\newtheorem{nota}[thm]{Notation}

\newtheoremstyle{rue}
{3pt}
{3pt}
{\itshape}
{}
{\bfseries}
{.}
{.5em}
{}
\theoremstyle{rue}
\newtheorem*{thmintro}{Theorem}

\newcommand{\C}{\mathcal{C}}
\newcommand{\Ub}{\mathcal{U}}

\newcommand{\F}{\mathcal{F}}

\renewcommand{\O}{\mathcal{O}}
\newcommand{\Ar}{\text{Arr}}

\newcommand{\D}{\mathcal{D}}
\newcommand{\Ba}{\mathcal{B}}

\newcommand{\Xa}{\mathcal{X}}

\newcommand{\A}{\mathcal{A}}
\newcommand{\Aa}{\mathcal{A}}
\newcommand{\B}{\mathscr{B}}

\newcommand{\M}{\mathscr{M}}

\newcommand{\Ea}{\mathcal{E}} 
\newcommand{\Qa}{\mathcal{Q}} 

\newcommand{\Un}{\mathbb{I}} 

\newcommand{\G}{\mathcal{G}}

\newcommand{\Z}{\mathbb{Z}}

\newcommand{\Rci}{\mathcal{R}\mathcal{I}}

\newcommand{\Lci}{\mathcal{L}\mathcal{I}}

\renewcommand{\H}{\mathcal{H}}

\newcommand{\Fb}{\mathbf{F}}
\newcommand{\Ya}{\mathcal{Y}}

\newcommand{\Pa}{\mathcal{P}}

\renewcommand{\to}{\longrightarrow}
\newcommand{\ol}[1]{\overline{#1}}
\newcommand{\ul}[1]{\underline{#1}}

\newcommand{\hrw}{\hookrightarrow}
\newcommand{\xhrw}{\xhookrightarrow}
 
\renewcommand{\bf}[1]{\mathbf{#1}}

\newcommand{\Ob}{\text{Ob}}
\newcommand{\0}{\textbf{0}} 

\newcommand{\tx}[1]{\text{#1}}
\newcommand{\tld}[1]{\widetilde{#1}}

\renewcommand{\to}{\longrightarrow}
\DeclareMathOperator\Id{Id}
\DeclareMathOperator\Hom{Hom}
\DeclareMathOperator\HOM{\ul{Hom}}

\DeclareMathOperator\Set{\textbf{Set}} 
\DeclareMathOperator\Cat{\mathbf{Cat}}
\DeclareMathOperator\colim{\tx{$colim$}}
\DeclareMathOperator\cof{\text{cof}}
\DeclareMathOperator\fib{\text{fib}} 
\DeclareMathOperator\op{\tx{op}}

\DeclareMathOperator\Ho{\mathbf{ho}} 
\DeclareMathOperator\oalg{\O\tx{-$Alg$}}

\DeclareMathOperator\Arr{\text{Arr}} %
\DeclareMathOperator\ag{\mathscr{A}}
\DeclareMathOperator\mua{\M_\Ub[\ag]}
\newcommand{\mupa}{\tx{$\M^{'}_{\Ub^{'}}[\ag^{'}]$}}
\DeclareMathOperator\mdua{(\M \downarrow \Ub)}

\DeclareMathOperator\pig{\pi_{\G}}
\DeclareMathOperator\pif{\pi_{\F}}
\DeclareMathOperator\Piar{\Pi^{\Ar}}
\DeclareMathOperator\Pio{\Pi^{0}}
\DeclareMathOperator\Piun{\Pi^{1}}
\DeclareMathOperator\ogcr{[\G^0,\G^1,\pi_{\G}]}
\DeclareMathOperator\mdu{\M_{\Ub}[\ag]} %
\DeclareMathOperator\mualg{\M_{\Ub}[\oalg(\M)]} %

\DeclareMathOperator\Fc{[\F]}
\DeclareMathOperator\Gc{[\G]}
\DeclareMathOperator\Pc{[\Pa]}
\DeclareMathOperator\Qc{[\Qa]}
\DeclareMathOperator\Ec{[\Ea]}
\DeclareMathOperator\Xc{[\Xa]}
\DeclareMathOperator\Xci{[\Xa_\text{$i$}]}
\DeclareMathOperator\Bc{[\Ba]}

\DeclareMathOperator\arm{\Ar(\M)}
\DeclareMathOperator\armij{\Ar(\M)_{inj}}
\DeclareMathOperator\armpj{\Ar(\M)_{proj}}
\DeclareMathOperator\sg{\sigma}
\DeclareMathOperator\sgo{\sigma^{0}}
\DeclareMathOperator\sgi{\sigma^{1}}
\DeclareMathOperator\osg{[\sigma^{0},\sigma^{1}]}

\DeclareMathOperator\lm{\mathscr{L}_{\M}}
\DeclareMathOperator\mr{\mathscr{R}_{\M}}

\DeclareMathOperator\piq{\pi_{\Qa}}
\DeclareMathOperator\muaij{\M_{\Ub}[\ag]_{inj}} %
\DeclareMathOperator\muapj{\M_{\Ub}[\ag]_{proj}} %
\DeclareMathOperator\Iam{\mathbf{I}_{\M}} %
\DeclareMathOperator\Jam{\mathbf{J}_{\M}} %
\DeclareMathOperator\sset{\mathbf{sSet}} 
\DeclareMathOperator\sseta{\sset_{\ast}} 

\DeclareMathOperator\lcof{\text{cof}_{L}}
\DeclareMathOperator\lfib{\text{fib}_{L}}
\DeclareMathOperator\lwe{\text{W}_{L}}
\DeclareMathOperator\we{\text{W}}
\DeclareMathOperator\fdua{(\Fb \downarrow \ag)}

\DeclareMathOperator\modcat{\text{Mod}}
\DeclareMathOperator\Mor{\text{Mor}}
\DeclareMathOperator\Mono{\text{Mono}}
\DeclareMathOperator\Epi{\text{Epi}}
\DeclareMathOperator\bcov{\text{B}_\text{cov}}
\DeclareMathOperator\ot{\otimes}

\title{Some functorial factorizations for Quillen functors} 

\author{Hugo Bacard}

\date{\today}
\begin{document}
	\maketitle
	\begin{abstract}
	\noindent	We prove that any right Quillen functor  between  arbitrary model categories admits non trivial functorial factorizations that are similar to those of a model structure. We also prove that these factorizations can be made for lax monoidal right Quillen functors.
    Given a monad, operad or a PROP(erad) $\O$,  if we apply one of the factorizations to the forgetful functor $\Ub : \oalg(\M) \to \M$, we extend the  theory of Quillen-Segal $\O$-algebras initiated in \cite{Bacard_QS} without the hypothesis of $\M$ being a combinatorial model category. 
	\end{abstract}

	\section{Introduction}
	\noindent 
	This paper is motivated by a remark in Hovey's book \cite[p.21]{Hov-model} within which he suggests that we should think of the category of all model categories as itself ``\emph{something like a model structure, with the weak equivalences being the Quillen equivalences}''. This remark appears after observing that the class of Quillen equivalences is closed under retracts and has the $3$-for-$2$ property. We note that it is clear from the beginning that the known categories of model categories  and Quillen functors lack of (finite) limits and colimits, thus cannot fulfill the axioms  of a closed model structure given by Quillen \cite{Quillen_HA}. 
	 Nonetheless, certain categorical operations on model categories remain possible such as products, homotopy fiber products, and more generally \emph{lax} homotopy limits (see \cite{Bergner_holim, Harpaz_lax,Hov-model,To_hall}). Moreover, Bergner \cite{Bergner_hocolim_mod} developed  a notion of homotopy colimit of a diagram of model categories, but she pointed out that the colimit-object is a category with weak equivalences that is not a model category in general.
	 More recently, Barton \cite{Barton_phd} has extensively studied  the question raised by Hovey with the more flexible notion of \emph{premodel category}. Barton proved that there is a \emph{model $2$-category structure} on the $2$-category of combinatorial premodel categories. In particular, a map of combinatorial premodel categories can be factored in two different ways using a (large) small object argument.
	 
	  Our goal in this paper is to discuss the existence of functorial factorizations for Quillen functors within the restrictive world where the objects are model categories. We pursue here some of the ideas in \cite{Bacard_QS} with a method different from that of Barton. 
	   Although it is standard to work with the category $\modcat_l$ of model categories and  left Quillen functors, we will consider instead the  category $\modcat_r$ of model categories and right Quillen functors between them. This choice doesn't affect the ultimate goal since a factorization of a right Quillen functor produces at the same time a factorization of its companion left adjoint. 
	 We show that any right Quillen functor admits two types of factorizations: one of the form ``trivial cofibration followed by a fibration'' while the other one is of the form ``cofibration followed by a trivial fibration''. 
	 One of these factorizations was established for right Quillen functors between combinatorial model categories in our previous work \cite{Bacard_QS}; and we extend this result here to all model categories. The other factorization generalizes the construction of the cylinder and path objects of a model category given by Renaudin \cite{Renaudin_mod}. Our main results are based on the following theorems (see Theorem \ref{fact-inj-thm}, Theorem \ref{fact-proj-thm}, Theorem \ref{right-fact-inj-thm} and Theorem \ref{right-fact-proj-thm}).
	\begin{thmintro}
     Let $\Ub : \ag \to \M$ be a right Quillen functor.\\
      Then there is a factorization $\ag \xhookrightarrow{i} \B \xtwoheadrightarrow[\sim]{p} \M$ of $\Ub$ such that:
     	\begin{enumerate}
     		\item $i$ is a right Quillen functor which is injective on objects.
     		\item $p$ is a right Quillen equivalence, an isofibration, and admits a section.
     	\end{enumerate}
	\end{thmintro}

	\begin{thmintro}
	Let $\Ub : \ag \to \M$ be a right Quillen functor.\\
	 Then there is a factorization $\ag \xhookrightarrow[\sim]{i'} \B^{'} \xtwoheadrightarrow{p'} \M$ of $\Ub$ such that:
		\begin{enumerate}
			\item $i'$ is a right Quillen equivalence which admits a retraction and is injective on objects.
			\item $p'$ is a right Quillen functor and an isofibration.
		\end{enumerate}
\end{thmintro}
\noindent We prove that  the couples $(i,p)$ and $(i',p')$ define each a \emph{functorial factorization} on $\modcat_r$ (Theorem \ref{main-thm-fact}). We remind the reader that  $\modcat_r$ (resp. $\modcat_l$) can be given the structure of a large $2$-category, where a $2$-morphism is a natural transformation between parallel right (resp. left) Quillen functors. Furthermore, we have an equivalence of underlying $1$-categories $\modcat_{r, \leq 1}^{op} \simeq \modcat_{l,\leq 1}$, due to the fact that an adjunction is made of functors going in opposite directions. There is also a notion of homotopy between right (resp. left) Quillen functors defined as special types of  $2$-morphism (Definition \ref{df-right-ho}). With this notion of homotopy, we show that for each factorization, the maps $i$ and $p$ (resp. $i'$ and $p'$) are \emph{homotopy orthogonal}, in that any lifting problem defined by such maps admits automatically a solution up-to-homotopy (Lemma \ref{lem-lifting} and Lemma \ref{lem-lift-2}). In view of these results, each factorization above does not give a \emph{weak factorization system} as one might wishes, but a `\emph{weaker}' factorization system. 
Before closing this introduction, let us outline briefly the idea of the proof of the above theorems. We define $\B$ as the comma category $\mdua$, through which the functor $\Ub$ factors functorially  as $(\ag \xhookrightarrow{\iota} \mdua \xrightarrow{\Pio} \M)$ (Proposition \ref{adjunction-prop}).  On the category $\mdua$ we have the well known injective and projective model structures. For each model structure, we have two right Quillen functors $\Pio: \mdua \to \M$ and $\Piun : \mdua \to \ag$. The strategy of the proof is to localize these model structures along $\Pio$ and along $\Piun$. In other words, we do not change the underlying factorization at a categorical level.

The existence of these factorizations has many consequences most of which will be treated in a different work. In \cite{Bacard_QS}, we have shown that the factorizations of an endofunctor $\Ub: \M \to \M$, such as the loop space functor $\Omega: \sseta \to \sseta $, along with \emph{linked $\Z$-sequences} were the building blocks in our approach to getting the \emph{stable homotopy category}. Another interesting aspect is the construction of Quillen equivalent models out of a morphism and the possibility to replace a map by a better behaved one. The last fact is reminiscent to the \emph{mapping cylinder and mapping path-space} constructions in Topology which strengthens Hovey's intuition that model categories - invented by Quillen by abstracting the homotopy theory of spaces - behave themselves like spaces. These factorizations enable us to extend a result on Quillen-Segal algebras in our previous work \cite{Bacard_QS}. Furthermore, we prove that they can be made for any lax monoidal functor which is a part of a monoidal Quillen adjunction. 
With the recent developments in Homotopy theory, and especially in the theory of higher categories, it is clear that the $2$-category $\modcat_r$ is only an approximation of the true object we should be looking at, which is an $\infty$-category. The study of the corresponding $\infty$-category goes beyond the scope of this paper and shall be done in a different work.
Moreover, we note that the comma category $\mdua$ can  be defined for Quillen functors between premodel categories and it would be interesting to find out the resulting structure in this context as well as the functorial factorizations that might emerge.
\subsection{Organization of the paper} The paper is structured as follows. %
\begin{enumerate}
	\item Section \ref{sec-prelim} contains some preliminary results that lead to the functorial factorization $(\ag \hookrightarrow \mdua \to \M)$ in the category of adjunctions (Proposition \ref{prop-left-ehk}).
	\item Section \ref{sec-fact} contains most of the results of the paper. The proofs of the main theorem is in Subsection \ref{proof-main-thm-fact}.
	\item We give a short discussion on the lifting properties in Section \ref{sec-lifting}.
	\item In Section \ref{sec-qs-alg}  we extend a previous result on Quillen-Segal algebras.
	\item In Section \ref{sec-fact-lax} we prove that:
	\begin{enumerate}
		\item  if $\Ub$ is lax monoidal functor, then there is point-wise product on $\mdua$,  and the previous factorization expresses $\Ub$ as the composite of a lax monoidal factorization followed by a strong monoidal functor.
		\item  if $\Ub$ is part of \emph{weak  monoidal Quillen adjunction}  in the sense of Schwede-Shipley \cite{Sch-Sh-equiv},  then $\mdua$ carries also a structure of a monoidal model category (Theorem \ref{muaij-mon-mod}, Theorem \ref{muaij-left-mon-mod}). 
		\item the lax functor $\Ub$ can be replaced it by a strong monoidal functor which is part of a monoidal Quillen adjunction (Corollary \ref{cor-fact-mon}).
	\end{enumerate}
    \item We've also provided other factorizations in Appendix \ref{sec-other-fact}.
    \item In Appendix \ref{sec-hom-alg-top} we prove that the aforementioned factorization holds for abelian categories and Grothendieck sites under reasonable hypotheses. This will be used in a future work following  Rezk's notion of  \emph{model topos} (see\cite{Rezk_mod_top}). 
	\item We've put the long proofs in an appendix. 
\end{enumerate}

\subsection{Notation and Hypotheses} 
\hfill 
\begin{itemize}
	\item $\Ub: \ag \to \M$ will be in general a right adjoint whose left adjoint is $\Fb$.
	\item ``left $\dashv$ right'' $=$ an adjunction, where ``left'' is the left adjoint and ``right'' is the right adjoint.
	\item $\mua := \mdua = $ the comma category whose objects are  triples $\Fc=[\F^0, \F^{1}, \pi_{\F}: \F^0 \to \Ub(\F^1)]\in \M \times \ag \times \Arr(\M)$. 
	\item $\sigma: [\F] \xrightarrow{\osg}[\G]$ represents a map in $\mua$.
	\item The categories $\ag$ and $\M$ are arbitrary model categories.
	\item We have three functors:
	\begin{itemize}
		\item  $\Pi^0: \mdua \to \M $, with $\Pi^0([\F])= \F^0$;
		\item  $\Pi^{1}: \mdua \to \ag $, with  $\Pi^{1}([\F])= \F^1$;
		\item  $\Pi_{\Ar}: \mdua \to \Ar(\M)$, with $\Pi_{\Ar}([\F])= \pi_{\F}$.
	\end{itemize}
	
	\item $\iota : \ag \hrw \mua$ is the embedding that maps $\Pa \mapsto \iota(\Pa)= [\Ub(\Pa), \Pa, \Id_{\Ub(\Pa)}]$.
	\item $\Mor(\C)=$ the class of all morphisms of a category $\C$.
	\item $\Mono(\C)= $ the class of monomorphims of $\C$.
	\item $\Epi(\C)= $ the class of epimorphisms of $\C$.
	\item $\emptyset, \emptyset_{\C}, \emptyset_{\ag}, \emptyset_{\M}, \cdots $ are initial objects.
	\item $\ast, \ast_{\C}, \ast_{\ag}, \ast_{\M}, \cdots $ are terminal objects.
	\item $\cof(-), \fib(-), \we(-)$, are respectively the classes of cofibrations, fibrations and weak equivalences of a model category `$-$'.
	\item $\Cat =$ the category of small categories. 
\end{itemize}
\begin{note}
Rather than demanding the existence of all limits and colimits in a model category, we will consider instead Quillen's original requirement of having finite limits and colimits. This assumption resolves the size issues that might occur when computing  limits and colimits in the comma category $\mdua$.
\end{note}
\section{Categorical preliminaries}\label{sec-prelim}
\noindent Let $\C$ be a $1$-category and let $f: A \to B$ be a morphism of $\C$.
	\begin{itemize}
		\item Say that $f$ admits a left inverse (or a retraction) if there is a morphism $r: B \to A$ such that $r \circ f = \Id_A$. 
		\item Say that $f$ admits a right inverse (or a section) if there is a morphism $s: B \to A$ such that $f \circ r = \Id_B$. 
	\end{itemize}
We will denote by $\Lci(\C) \subseteq \Mor(\C)$ the class of all morphisms that admit a left inverse, and by $\Rci(\C) \subseteq \Mor(\C)$ the class of all morphisms that admit a right inverse. It can be easily seen that we have some isomorphisms of classes $\Lci(\C) \cong \Rci(\C^{op})$ and $\Rci(\C) \cong \Lci(\C^{op})$, where  $\C^{op}$ is the opposite or dual category. With some basic category theory, one can prove the following proposition.
\begin{prop}\label{prop-closure-li-ri}
For any category $\C$ the following hold.
\begin{enumerate}
	\item The two classes $\Lci(\C)$ and $\Rci(\C)$ are closed under composition and retracts.
	\item We have $\Lci(\C) \subseteq \Mono(\C)$ and $\Rci(\C) \subseteq \Epi(\C)$.
	\item If $\ast_{\C}$ is a terminal object of $\C$, then the unique map $X \to \ast_{\C}$ admits the right lifting property (RLP) against every map in $\Lci(\C)$, i.e., every object $X$ is $\Lci(\C)$-injective.
	\item If $\emptyset_{\C}$ is an initial object of $\C$, then the unique map 
	$\emptyset_{\C} \to Y$ admits the left lifting property (LLP) against every map in $\Rci(\C)$, i.e., every object $Y$ is $\Rci(\C)$-projective. 
\end{enumerate}
\end{prop}
\begin{proof}
We will only prove Assertion $(1)$, the other ones are straightforward. Moreover, using the duality $\Rci(\C) \cong \Lci(\C^{op})$, it is enough to prove this assertion for the class $\Lci(\C)$. We note that $\Lci(\C)$ is clearly closed under composition, thus it remains to show that it is also closed under retracts. Let $f$ be a retract of an element $g \in \Lci(\C)$. By definition, we have a commutative diagram where the horizontal composites are identities:
\[
\xy
(0,15)*+{A}="W";
(0,0)*+{B}="X";
(50,0)*+{B}="Y";
(50,15)*+{A}="E";
{\ar@{->}^-{f}"W";"X"};
{\ar@{->}^-{f}"E";"Y"};
(25,15)*+{C}="U";
(25,0)*+{D}="V";
{\ar@{->}^-{\alpha^{1}}"X";"V"};
{\ar@{->}^-{\beta^{1}}"V";"Y"};
{\ar@{->}^-{g}"U";"V"};
{\ar@{->}^-{\beta^0}"U";"E"};
{\ar@{->}^-{\alpha^0}"W";"U"};
\endxy
\]
If $r : D \to C$ is a left inverse of $g$, then the map $r' = \beta^0 \circ r \circ \alpha^{1}$ is a left inverse  of $f$ since:
$$r' \circ f = \beta^0 \circ r \circ \underbrace{\alpha^{1}  \circ f}_{= g \circ \alpha^0}= \beta^0 \circ \underbrace{r \circ g}_{=\Id_C} \circ \alpha^0 = \beta^0 \circ  \alpha^0 = \Id_A.$$
\end{proof}
\noindent If the category $\M$ has an initial object $\emptyset_{\M}$, then we can define a functor $L^1: \ag \to \mua$ with  $L^{1}(\Pa)= [\emptyset_{\M}, \Pa, \emptyset_{\M} \to  \Ub(\Pa)]$. Dually, in the presence of coinitial objects such that $\Ub$ preserves them, we can define a functor $R^0: \M \to \mua$ given by $R^0(m)=[m, \ast_{\ag}, m \to \Ub(\ast_{\ag})=\ast_{\M}]$. Moreover, if $\Ub$ has a left adjoint $\Fb$, we can also define a functor $\Fb^+: \M \to \mua$, with $\Fb^+(m)= [m, \Fb(m), m \xrightarrow{\eta_m} \Ub\Fb(m)]$. We've established in \cite{Bacard_QS} the following:
\begin{prop}\label{adjunction-prop}
	Let $\Ub: \ag \to \M$ be a functor between categories with initial and coinitial objects such that $\Ub(\ast_{\ag})= \ast_{\M}$. Then the following hold. 
	
	\begin{enumerate}
		\item We have three adjunctions $(\Pi^{1} \dashv \iota)$, $(L^{1} \dashv \Pi^{1})$, $(\Pi^0 \dashv R^0)$.
		\item If $\Ub$ has a left adjoint $\Fb: \M \to \ag$, then
		there is an adjunction  $(\Fb^+ \dashv \Pi^0)$.
		\item We have some factorizations: 
		$$\Pio \circ \iota = \Ub, \quad \Piun \circ \iota = \Id_{\ag}, \quad \Pio \circ R^0 = \Id_{\M}, \quad \Piun \circ L^1 = \Id_{\ag}.$$
	\end{enumerate}
In particular we have $\iota \in \Lci(\Cat)$, $ \Pi^0 \in \Rci(\Cat)$ and $L^{1} \in \Lci(\Cat)$.
\end{prop}
\noindent The adjunction $(\Piun \dashv \iota)$ holds without the existence of (co)initial objects. Recall that there is a model structure on $\Cat$, sometimes called the ``folk model structure''  or  ``canonical model structure'' (see for example  \cite{Joy-Tier,Lack_model_2_cat}). In this model structure, the cofibrations are the functors that are injective on objects, while the fibrations are the \emph{isofibrations}. The weak equivalences are the equivalences of categories. We will denote this model category by $\Cat_{folk}$.
\begin{prop}\label{fact-ri-li}
	In the category $\Cat$ the following hold.
	\begin{enumerate}
		\item Any functor $\Ub \in \Lci(\Cat)$ is injective on objects, whence a cofibration in $\Cat_{folk}$.
		\item Any functor $\Ub :\ag \to \M$ admits a factorization $\Ub = G_2 \circ G_1 $, where $G_2 \in \fib(\Cat_{folk})$ and $G_1 \in \Lci(\Cat)$.
		\item Any functor $\Ub :\ag \to \M$ between categories with coinitial objects such that $\Ub(\ast_{\ag})= \ast_{\M}$,  admits a factorization $\Ub = G_2 \circ G_1 $, where $G_2 \in \Rci(\Cat) \cap \fib(\Cat_{folk})$ and $G_1 \in \Lci(\Cat)$.
	\end{enumerate}
\end{prop}
\begin{proof}
	Since we have a functor $\Ob: \Cat \to \Set$, it is clear that if $r$ is a retraction of  $\Ub$ then $\Ob(r)$ is a retraction of $\Ob(\Ub)$ which implies that $ \Ob(\Ub) \in \Mono(\Set)$ by  Proposition \ref{prop-closure-li-ri}. This gives the first assertion. Assertion $(2)$ and $(3)$ follow from Proposition \ref{adjunction-prop} if we set $G_1 = (\iota: \ag \to \mdua)$ and $G_2 = (\Pio: \mdua \to \M)$. We only need to check that $\Pio$ is indeed an isofibration and that the factorization is functorial. This is not hard but we include the proof for the reader's convenience. $\Pio$ will be an isofibration if we can show that for any $\Fc \in \mdua$  and for any isomorphism $u : \Pio(\Fc) \xrightarrow{\cong} m$  in $\M$, there is an isomorphism $\sg : \Fc \to \Gc_u$ in $\mdua$ such that $\Pio(\sg) =u$.
Set $\Gc_u = [m, \F^1, \pif \circ u^{-1} : m \to \Ub(\F^1)]$ and let $\sg  : \Fc \xrightarrow{[\sg^0, \sg^{1}]}\Gc_u$ be the morphism given by $\sg^0 = u$ and $\sg^{1}= \Id_{\F^1}$. Clearly, $\sg$ is an isomorphism and one has $\Pio(\sg) =u$.
The factorization is clearly functorial, in that given any commutative square of solid arrows, there is a dotted functor such that the inner squares are also commutative:
\begin{align}\label{diag-fact-comma}
\xy
(0,15)*+{\ag}="W";
(0,0)*+{\ag'}="X";
{\ar@{->}^-{H}"W";"X"};
(25,15)*+{\M}="U";
(25,0)*+{\M'}="V";
{\ar@{->}^-{\Ub'}"X";"V"};
{\ar@{->}_-{K}"U";"V"};
{\ar@{->}^-{\Ub}"W";"U"};
\endxy
\xy
(-5,7)*+{\Rightarrow}="L";
(0,15)*+{\ag}="W";
(0,0)*+{\ag'}="X";
(50,0)*+{\M'}="Y";
(50,15)*+{\M}="E";
{\ar@{->}^-{H}"W";"X"};
{\ar@{->}_-{K}"E";"Y"};
(25,15)*+{\mdua}="U";
(25,0)*+{(\M'\downarrow \Ub')}="V";
{\ar@{->}^-{\iota}"X";"V"};
{\ar@{->}^-{\Pio}"V";"Y"};
{\ar@{.>}^-{E(H,K)}"U";"V"};
{\ar@{->}^-{\Pio}"U";"E"};
{\ar@{->}^-{\iota}"W";"U"};
\endxy
\end{align}
The functor $E(H,K)$ maps $[\F^0,\F^1,\pif] \mapsto [K(\F^0), H(\F^1), K(\pif) ]$ and takes $\sg =[\sg^0,\sg^{1}] \mapsto [K(\sg^0), H(\sg^{1})]$.
\end{proof}

\begin{prop}\label{prop-left-ehk}
Assume that the functors $\Ub$, $\Ub'$, $H$ and $K$ in (\ref{diag-fact-comma}) are right adjoint functors with respective left adjoints $\Fb,\Fb', H_\ast$ and $K_\ast$. Then the functor $E(H,K) : \mdua \to (\M'\downarrow \Ub')$ is also a right adjoint.
\end{prop}

\begin{proof}
	See Appendix \ref{proof-prop-left-ehk}
\end{proof}

\section{Factorizations in the category of model categories}\label{sec-fact}
\noindent There are various choices for the morphisms of a ``category of model categories''. The standard choice is to take as morphisms the left Quillen functors as in Hovey \cite{Hov-model}. However, as mentioned earlier, we will work with right Quillen functors between model categories. The example that motivates this choice is a forgetful functor $\Ub: \oalg(\M) \to \M$, which is  a right Quillen functor in many cases. Shulman \cite{Shulman_compose_der_func} showed that there is a double category  for model categories and it would be interesting to find out how the present work fits in his settings.
\begin{df}
	Let $\modcat_r$ be the large category of model categories defined as follows. 
	\begin{itemize}
		\item The objects are model categories.
		\item A morphism $G:\C \to \D$ is a triple $G=(G^l,G^r,\varphi)$ where $G_l\dashv G_r$ is a Quillen adjunction with $G^r : \C \to \D$ the right Quillen functor and $\varphi: \D(G^l X, Y) \xrightarrow{\cong} \C(X, G^r Y)$.
	\end{itemize}
\end{df}
\noindent If $(\Fb, \Ub, \varphi)$ is a Quillen adjunction where $\Ub$ is right Quillen, the associated morphism of $\modcat_r$ will be denoted by  $\Ub_{(\Fb,\varphi)} = (\Ub_{(\Fb,\varphi)}^l,\Ub_{(\Fb,\varphi)}^r,\varphi)$ with $\Ub_{(\Fb,\varphi)}^l= \Fb$ and $\Ub_{(\Fb,\varphi)}^r= \Ub$. 
We note that category $\modcat_r$ is in fact a $2$-category as explained in \cite{Hov-model}. Given  $G_i=(G_i^l,G_i^r,\varphi): \C \to \D$, $i \in \{0;1\}$, a $2$-morphism $\tau: G_0 \to G_{1}$ is just a natural transformation $\tau:G_0^r \to G_{1}^r$ between the right Quillen functors. Following Renaudin \cite{Renaudin_mod}, we consider:
\begin{df}\label{df-right-ho}
\ 	\medskip
	\begin{enumerate}
		\item 	Let $G_i=(G_i^l,G_i^r,\varphi): \C \to \D$ be parallel morphisms in $\modcat_r$, $i \in \{0,1\}$.\\
		Say that a $2$-morphism $\tau: G_0 \to G_{1}$ is a  \emph{right homotopy} if the map $\tau_C:G_0^r(C) \to G_{1}^r(C)$ is a weak equivalence in $\D$ for any fibrant object $C \in \C$.
		\item Call a morphism $G: \C \to \D $  a \emph{Quillen homotopy equivalence} if there is a morphism $H : \D \to \C$ with  a zig-zag of homotopies between $\Id_{\C}$ and $ H \circ G$ and zig-zag of homotopies between $\Id_{\D}$ and $G\circ H$.
	\end{enumerate}
\end{df}
\noindent 
Throughout this paper, unless otherwise specified, whenever we say `homotopy between right Quillen functors' we mean right homotopy. The following result can be found in \cite{Hov-model, Renaudin_mod}. 
\begin{lem}
Let $RG_i^r: \Ho(\C) \to \Ho(\D)$ be the respective right derived functors, $i \in \{0,1\}$. 
\begin{enumerate}
	\item If $\tau: G_0 \to G_{1}$ is a right homotopy, then  $R\tau: RG_0^r \to RG_0^r$ is a natural isomorphism.
	\item If $\tau: G_0 \to G_{1}$ is a right homotopy, then if one of  $G_1, G_2$ is a Quillen equivalence, then so is the other.
	\item Any Quillen homotopy equivalence is a Quillen equivalence.
\end{enumerate}
\end{lem}
\begin{lem}
If $K: \D \xrightarrow{(K^l,K^r,\varphi)} \D'$  is a morphism in $\modcat_r$ and $\tau : G_1 \to G_2$ is a right homotopy, then the composite $K \circ \tau : K \circ G_0 \to K \circ G_1$ is also a right homotopy. More generally, right homotopies can be horizontally composed, vertically composed and possess the vertical $3$-for-$2$ property.
\end{lem}
\begin{proof}
Any right Quillen functor preserves weak equivalences between fibrant objects by Ken Brown's lemma, thus $K^r(\tau_C)$ is a weak equivalence between fibrant objects if $C$ is fibrant. See \cite{Hov-model} for details.
\end{proof}
\noindent If $\tau : G_1 \to G_2$ is a homotopy and $H$ and $K$ are maps in $\modcat_r$, we will say that:
\begin{itemize}
	\item $\tau$ is a homotopy relative to $H$ if $\tau \circ H$ is the identity homotopy (if the composite makes sense)
	\item $\tau$ is a homotopy co-relative to $K$ if $K \circ \tau$ is the identity homotopy.
\end{itemize}
\begin{df}
Let $G:\C \to \D$ be a morphism in  $\modcat_r$.  
\begin{enumerate}
\item \begin{itemize}
	\item Say that $G$ admits a weakly invertible retraction $H : \D \to \C$, if $H$ is a retraction of $G$ and if there is a homotopy $\tau : \Id_{\D}  \to G \circ H$. 
	\item If in addition $\tau$ is homotopy relative to $G$ we will say that it is a deformation retraction
\end{itemize}

\item \begin{itemize}
	\item Similarly, say that $G$ admits a weakly invertible section $H : \D \to \C$, if $H$ is a section of $G$ and if there is a homotopy $\tau : \Id_{\C}  \to H \circ G$.
	\item If in addition $\tau$ is a homotopy co-relative to $G$, we will say that it is a deformation section.
\end{itemize}

\end{enumerate}
\end{df}
\begin{nota}
\hfill
\begin{itemize}
	\item $\Lci_w(\modcat_r) = $ the class of maps admitting a weakly invertible retraction.
  \item $\Lci_{wrel}(\modcat_r) = $ the class of maps admitting a weakly invertible retraction equipped with a deformation retraction 
  	\item $\Rci_w(\modcat_r) = $ the class of maps admitting a weakly invertible section.
  \item $\Rci_{wcor}(\modcat_r) = $ the class of maps admitting a weakly invertible section equipped with a deformation section.
  \item $Mor_{inj_{ob}}(\modcat_r) = $ the class of maps $G=(G^l,G^r,\varphi)$ such that $G^r$ is injective on objects.
  \item $Mor_{ifib} (\modcat_r) = $ the class of maps $G=(G^l,G^r,\varphi)$ such that $G^r$ is an isofibration between the underlying categories.
\end{itemize}
\end{nota}
\noindent It can be easily checked that the classes $Mor_{inj_{ob}}(\modcat_r)$ and $Mor_{ifib} (\modcat_r)$ are closed under composition and retracts.
\begin{prop} \label{prop-retract-Quillen}
With the previous definition, the following statements are true.
\begin{enumerate}
	\item We have some inclusion of classes:
	$$\Lci_{wrel}(\modcat_r) \subset \Lci_{w}(\modcat_r) \subset Mor_{inj_{ob}}(\modcat_r), \quad \quad  \Rci_{wcor}(\modcat_r) \subset \Rci_w(\modcat_r).$$
	\item Any element of $\Lci_w(\modcat_r)$, $\Lci_{wrel}(\modcat_r)$, $\Rci_w(\modcat_r)$ and $\Rci_{wcor}(\modcat_r)$  is a Quillen equivalence.
	\item The classes $\Lci_w(\modcat_r)$, $\Lci_{wrel}(\modcat_r)$, $\Rci_w(\modcat_r)$ and $\Rci_{wcor}(\modcat_r)$   are also closed under composition and retracts.
\end{enumerate}
\end{prop}

\begin{proof}
See Appendix \ref{proof-prop-retract}.
\end{proof}
\noindent The main result of this paper is the following theorem.
\begin{thm}\label{main-thm-fact}
	Let $(\Fb, \Ub, \varphi)$ be a Quillen adjunction where $\Ub: \ag \to \M$ is right Quillen, and let  $\Ub_{(\Fb,\varphi)}: \ag \to \M$ be the corresponding morphism in $\modcat_r$. Then the following hold.
	\begin{enumerate}
		\item There is a functorial factorization  $\Ub_{(\Fb,\varphi)}= G_2 \circ G_{1}$ where:
		\begin{enumerate}
			\item $G_1= (G_1^l, G_1^r,\varphi) \in \Lci_w(\modcat_r) \cap Mor_{inj_{ob}}(\modcat_r)$,
			\item $G_2= (G_2^l, G_2^r,\varphi) \in Mor_{ifib} (\modcat_r) $.
		\end{enumerate} 
		\item There is a functorial factorization  $\Ub_{(\Fb,\varphi)}= G_2 \circ G_{1}$ where:
		\begin{enumerate}
			\item $G_{1} \in Mor_{inj_{ob}}(\modcat_r)$,
			\item  $G_2= (G_2^l, G_2^r,\varphi)  \in \Rci_w(\modcat_r) \cap Mor_{ifib} (\modcat_r)$.
		\end{enumerate}
	\end{enumerate}
\end{thm}
\noindent We give the proof of this theorem at the end of this section (in Subsection \ref{proof-main-thm-fact}). The proof of the theorem relies on the existence of various model structures one can put on the comma category $\mdua$. 
We start with some backgrounds on the homotopy theory on the comma category $\mdua$. Recall that a morphism $\sigma: \Fc \xrightarrow{[\sg^0,\sg^{1}]} \Gc$ in $\mdua$ may be displayed as a either one of the adjoint equivalent commutative squares:
\[
\xy
(0,18)*+{\F^0}="W";
(0,0)*+{\Ub(\F^1)}="X";
(30,0)*+{\Ub(\G^1)}="Y";
(30,18)*+{\G^0}="E";
{\ar@{->}^-{\Ub(\sigma^{1})}"X";"Y"};
{\ar@{->}^-{\pi_{\F}}"W";"X"};
{\ar@{->}^-{\sigma^0}"W";"E"};
{\ar@{->}^-{\pi_{\G}}"E";"Y"};
(45,9)*+{\Longleftrightarrow}="Z";
\endxy
\xy
(0,18)*+{\Fb(\F^0)}="W";
(0,0)*+{\F^1}="X";
(30,0)*+{\G^1}="Y";
(30,18)*+{\Fb(\G^0)}="E";
{\ar@{->}^-{\sigma^{1}}"X";"Y"};
{\ar@{->}^-{\varphi(\pi_{\F})}"W";"X"};
{\ar@{->}^-{\Fb(\sigma^0)}"W";"E"};
{\ar@{->}^-{\varphi(\pi_{\G})}"E";"Y"};
\endxy
\]
Given such a morphism $\sg$ we will say that: \label{injective-data}
\begin{enumerate}
\item $\sigma$ is a injective (trivial) cofibration if $\sigma^0$ is a (trivial) cofibration in $\M$ and $\sigma^{1}$ is a (trivial) cofibration in $\ag$. 
\item $\sigma$ is a level-wise weak equivalence (resp. level-wise fibration) if:
\begin{itemize}
	\item  $\sigma^0$ is a weak equivalence (resp. fibration)  in $\M$ and
	\item  $\sigma^{1}$ is a weak equivalence (resp. fibration) in $\ag$.
\end{itemize}
\item $\sigma$ is an injective (trivial) fibration  if:
\begin{itemize}
	\item $\sigma^{1}: \F^1 \to \G^1$ is a  (trivial) fibration in $\ag$ and 
	\item the induced map $ \F^0 \to \Ub(\F^1) \times_{\Ub(\G^1)} \G^0 $
	is a (trivial) fibration in $\M$. 
\end{itemize}
\item $\sg$ is a projective (trivial) cofibration if:
\begin{itemize}
	\item $\sigma^0: \F^0 \to \G^0$ is a  (trivial) cofibration in $\M$ and 
	\item the induced map $ \F^1 \cup^{\Fb(\F^0)} \Fb(\G^0) \to  \G^1 $ is a (trivial) cofibration in $\ag$.
\end{itemize}
\end{enumerate}

\begin{thm}\label{inj-proj-thm}
\hfill
	\begin{enumerate}
		\item There is an \emph{injective model structure} $\muaij$, on the category $\mua$ where the cofibrations (resp. fibrations) are the injective cofibrations (resp. injective fibrations) and the weak equivalences are the level-wise weak equivalences.
		\item There is a \emph{projective model structure} $\muapj$, on the category $\mua$ where the cofibrations (resp. fibrations) are the projective cofibrations (resp. level-wise fibrations) and the weak equivalences are the level-wise weak equivalences.
	\item The identity functor $\Id: \muapj \to \muaij$ is a right Quillen equivalence.
	\item We have two factorizations of $\Ub$: 
		$ \ag \xhookrightarrow{\iota} \muapj \xrightarrow{\Pio} \M,$
		$ \ag \xhookrightarrow{\iota} \muaij \xrightarrow{\Pio} \M.$
	\end{enumerate}
	 
\end{thm}
\begin{proof}
This is to be found with details in \cite{Bacard_QS}.
\end{proof}
\noindent In virtue of this result and of Proposition \ref{fact-ri-li} we have:
\begin{thm}\label{first-fact}
Let $(\Fb, \Ub, \varphi)$ be a Quillen adjunction where $\Ub: \ag \to \M$ is right Quillen, and let  $\Ub_{(\Fb,\varphi)}: \ag \to \M$ be the corresponding morphism in $\modcat_r$. Then the following hold.
\begin{enumerate}
	\item There is a functorial factorization  $\Ub_{(\Fb,\varphi)}= G_2 \circ G_{1}$, with $G_{1} \in \Lci(\modcat_r)$ and $G_2 \in \Rci(\modcat_r)$:
	$ \ag \xhookrightarrow{G_{1}} \muapj \xrightarrow{G_2} \M.$
	
    \item There is a functorial factorization  $\Ub_{(\Fb,\varphi)}= G_2 \circ G_{1}$, with $G_{1} \in \Lci(\modcat_r)$ and $G_2 \in \Rci(\modcat_r)$:
	$ \ag \xhookrightarrow{G_{1}} \muaij \xrightarrow{G_2} \M.$

\end{enumerate}
\end{thm}
\begin{proof}
On a categorical level, i.e., if we discard the model structures, the two factorizations coincide. One sets:
\begin{itemize}
	\item $G_{1}=(G_{1}^l,G_{1}^r,\varphi)= (\Pi^{1}, \iota, \varphi)=\iota_{(\Pi^{1},\varphi)}$, where $\varphi$ is the implicit isomorphism in the adjunction $\iota \dashv \Pi^{1}$. 
	\item $G_2=(G_2^l,G_2^r,\varphi)= (\Fb^+, \Pi^0, \varphi)= \Pio_{(\Fb^+,\varphi)}$. 
\end{itemize}
There is a retraction $Z: \mua \to \ag$ of $G_{1}$ given by $Z=(Z^l,Z^r,\varphi)=(L_{1}, \Pi^{1}, \varphi)=\Piun_{(L_{1},\varphi)}$. The map $G_2$ admits a section $T: \M \to \mua$ given by $T=(T^l,T^r,\varphi)=(\Pi^0, R^0, \varphi) = {R^0}_{(\Pio,\varphi)}$. To prove that these factorizations are functorial, we must show that for $ms \in \{\tx{`inj'},\tx{ `proj'}\}$, the induced functor $E(H,K):\mdua_{ms} \to (\M'\downarrow\Ub')_{ms}$   is a right Quillen functor:
\[
\xy
(0,15)*+{\ag}="W";
(0,0)*+{\ag'}="X";
{\ar@{->}^-{H}"W";"X"};
(25,15)*+{\M}="U";
(25,0)*+{\M'}="V";
{\ar@{->}^-{\Ub'}"X";"V"};
{\ar@{->}_-{K}"U";"V"};
{\ar@{->}^-{\Ub}"W";"U"};
\endxy
\xy
(-5,7)*+{\Rightarrow}="L";
(0,15)*+{\ag}="W";
(0,0)*+{\ag'}="X";
(50,0)*+{\M'}="Y";
(50,15)*+{\M}="E";
{\ar@{->}^-{H}"W";"X"};
{\ar@{->}_-{K}"E";"Y"};
(25,15)*+{\mdua}="U";
(25,0)*+{(\M'\downarrow \Ub')}="V";
{\ar@{->}^-{\iota}"X";"V"};
{\ar@{->}^-{\Pio}"V";"Y"};
{\ar@{.>}^-{E(H,K)}"U";"V"};
{\ar@{->}^-{\Pio}"U";"E"};
{\ar@{->}^-{\iota}"W";"U"};
\endxy
\]
 By definition we have $E(H,K)[\sgo,\sgi]=[K(\sgo), H(\sgi)]$. Since $H$ and $K$ preserve the fibrations and the trivial fibrations, we see that  $E(H,K)[\sgo,\sgi]$ is a level-wise (trivial) fibration if $[\sgo,\sgi] $ is. This proves that $E(H,K): \mdua_{proj} \to (\M'\downarrow\Ub')_{proj}$ is a right Quillen functor.  To show that $E(H,K): \mdua_{inj} \to (\M'\downarrow\Ub')_{inj}$ is also a right Quillen functor, it is much easier to prove that its left adjoint $E(H,K)_{\ast}$ is a left Quillen functor, in that it preserves the level-wise (trivial) cofibrations. By definition, we have $E(H,K)_{\ast}[\theta^0, \theta^1] = [K_\ast(\theta^0), H_\ast(\theta^1)]$, where $H_\ast$ and $K_\ast$ are the respective left adjoint of $H$ and $K$. Since  $H_\ast$ and $K_\ast$ preserve the cofibrations and the trivial cofibrations, then clearly $E(H,K)_{\ast}[\theta^0, \theta^1] $ is a level-wise (trivial) cofibration if $[\theta^0, \theta^1] $ is.
\end{proof}
\subsection{Left Bousfield localizations}
\noindent
We are going to localize the two  model structures $\muaij$ and $\muapj$ along the functor $\Piun : \mua \to \ag$ without the hypothesis used in \cite{Bacard_QS}, wherein we've restricted ourselves to a right Quillen functor between combinatorial model categories. 
\subsubsection{Left injective model structure}
\begin{df}\label{left-injective-data}
	Let $\sigma: \Fc \xrightarrow{[\sgo,\sgi]} \Gc$ be a map in $\mua$. We will say that:
	\begin{enumerate}
		\item $\sigma$ is a \emph{left injective cofibration} if $\sigma^0$ is a cofibration in $\M$ and $\sigma^{1}$ is a  cofibration in $\ag$, that is if $\sigma$ is a  cofibration in $\muaij$. 
		\item $\sigma$ is a \emph{left weak equivalence} if  $\sigma^{1}$ is a weak equivalence  in $\ag$, that is if $\Pi^{1}(\sg)$ is a weak equivalence in $\ag$.
		
		\item $\sigma$ is a \emph{left injective fibration}  if:
		\begin{itemize}
			\item $\sigma^{1}: \F^1 \to \G^1$ is a  fibration in $\ag$ and if 
			\item the induced map $ \delta: \F^0 \to \Ub(\F^1) \times_{\Ub(\G^1)} \G^0 $
			is a \emph{trivial fibration} in $\M$. 
		\end{itemize}
	\end{enumerate}
\end{df}

\begin{nota}For simplicity we will adopt the following notation.
\begin{itemize}
	\item $\lcof(\muaij)= $ the class of all left injective cofibrations
	\item $\lfib(\muaij)= $ the class of all left injective fibrations
	\item $\lwe(\muaij) = $ the class of all left weak equivalences.
\end{itemize}
\end{nota}
\begin{thm}\label{left-inj-thm}
	Call a morphism $\sigma: \Fc \xrightarrow{[\sgo,\sgi]} \Gc$ in $\mua$
	\begin{enumerate}
		\item  a weak equivalence if it is a left injective weak equivalence,
		\item a cofibration if it is a left injective cofibration and 
		\item a fibration if it is a left injective fibration. 
	\end{enumerate}
	Then these choices provide $\mua$ with the structure of a model category that will be denoted by $\mua^{\bf{L}}_{inj}$. 
\end{thm}
An important consequence of the theorem is the following result: 
\begin{thm}\label{fact-inj-thm}
	With the notation above, the following statements are true.
	\begin{enumerate}
		\item We have two Quillen equivalences where the functor on the left hand side is the left Quillen functor:
		\begin{align*}
		L^{1}: \ag \leftrightarrows \mua^{\bf{L}}_{inj}: \Piun \hspace{1in} \Piun: \mua^{\bf{L}}_{inj} \leftrightarrows \ag : \iota
		\end{align*}
		\item We have a Quillen adjunction $ \Fb^+: \M \rightleftarrows \mua^{\bf{L}}_{inj}:  \Pio$.
		\item We have a factorization $\Ub = \Pio \circ \iota$:
		$ \ag \xhookrightarrow[\sim]{\iota} \mua^{\bf{L}}_{inj} \xrightarrow{\Pio} \M.$
		\item We have  $\Id_{\ag} =\Piun \circ  \iota $ and the unit $\eta : \Id_{\mua^{\bf{L}}_{inj}} \to \iota \circ \Piun $ is a right homotopy.
	\end{enumerate}
\end{thm}

\begin{proof}[Proof of Theorem \ref{fact-inj-thm}]
For $\Pa \in \ag $ and $\Fc=[\F^0,\F^{1},\pif] \in \mua$ we have $L^{1}(\Pa)= [\emptyset, \Pa, \emptyset \to  \Ub(\Pa)]$ and $\Piun(\Fc)= \F^1$. The functor $\Piun$ creates weak equivalences in $\mua^{\bf{L}}_{inj}$, therefore a map $L^{1}(\Pa) \to \Fc$ is a weak equivalence in $\mua^{\bf{L}}_{inj}$ if and only if -- by definition -- the map $(\Pa \to \F^1)=(\Pa \to \Piun(\F))$ is a weak equivalence in $\ag$. This is true in particular if $\Pa$ is a cofibrant and $\F$ is fibrant, which proves that the adjunction $L^{1} \dashv \Piun$ is a Quillen equivalence.  It also proves automatically that the other adjunction $\Piun \dashv \iota$ is a Quillen equivalence where $\Piun$ is a left Quillen functor. This gives Assertion $(1)$. Moreover,
for any left injective (trivial) fibration $\sg$, $\Pio(\sg)=\sg^0$ is a (trivial) fibration in $\M$, thus $\Pio$ is right Quillen, which proves Assertion $(2)$. Assertions $(3)$ is obvious. To prove Assertion $(4)$ we must show that for any fibrant object $\Fc \in \mua^{\bf{L}}_{inj}$ the map $\eta: \Fc \to \iota (\Piun(\Fc))$ is a weak equivalence. By inspection, the map $\eta$ is given by the couple  $[\pif, \Id_{\F^1}]$ which is (obviously) a weak equivalence in $\mua_{inj}^{\bf{L}}$ since $\Id_{\F^1}$ is a weak equivalence in $\ag$.
\end{proof}
The main ingredient for the proof of Theorem \ref{left-inj-thm} is the following lemma.
\begin{lem}\label{inj-lem}
	\ \\
	\begin{enumerate}
	
	\item Consider a lifting problem of solid arrows in $\mua$:
	\[
	\xy
	(0,18)*+{\Fc}="W";
	(0,0)*+{\Gc}="X";
	(30,0)*+{\Qc}="Y";
	(30,18)*+{\Pc}="E";
	{\ar@{->}^-{[\gamma^0, \gamma^{1}]}"X";"Y"};
	{\ar@{->}_-{[\sgo,\sgi]}"W";"X"};
	{\ar@{->}^-{[\theta^0,\theta^{1}]}"W";"E"};
	{\ar@{->}^-{[\beta^0,\beta^{1}]}"E";"Y"};
	{\ar@{-->}^-{[s^0,s^{1}]}"X";"E"};
	\endxy
	\]
If $\sg \in \lcof(\muaij) \cap \lwe(\muaij)$ and  $\beta  \in \lfib(\muaij)$, then a solution $s : \Gc \xrightarrow{[s^0,s^{1}]} \Pc$ exists.
\item Any morphism $\sg : \Fc \to \Gc$ of $\mua$ can be factored as:
$$[\F] \xhookrightarrow[\sim]{i} \Ec \xtwoheadrightarrow{p}[\G],$$
where $i \in  \lcof(\muaij) \cap \lwe(\muaij)$ and $p \in \lfib(\muaij)$. 
\end{enumerate}
\end{lem}
We defer the proof to Appendix \ref{proof-inj-lem}. 
\begin{proof}[Proof of Theorem \ref{left-inj-thm}]
It can be easily shown that the three classes of left weak equivalences, left injective cofibrations and left injective fibrations are closed under compositions and retracts. Moreover, one has:
\begin{itemize}
	\item $\lcof(\muaij)= \cof(\muaij)$,
	\item $\lfib(\muaij) \cap \lwe(\muaij) = \fib(\muaij) \cap \we(\muaij)$,
	\item $\we(\muaij) \subset \lwe(\muaij)$,
	\item $\lwe(\muaij)$ has the $3$-for-$2$ property.
\end{itemize}
For the remainder of the proof we will show that we have the required factorizations along with the lifting properties. Given any map $\sg$, we can use the axiom of the model category $\muaij$ to factor it as $\sg=p\circ i$ where $i \in \cof(\muaij)$ and $p\in \fib(\muaij) \cap \we(\muaij)$.  From the above observations, the map $i$ is a left injective cofibration and $p$ is a map that is simultaneously a left injective fibration and a left weak equivalence. This gives the first type of factorization.  By Lemma \ref{inj-lem}, any map $\sg$ can be factored as $\sg=p\circ i$, where $i \in  \lcof(\muaij) \cap \lwe(\muaij)$ and $p \in \lfib(\muaij)$. This gives the second type of factorization in $\mua$. The theorem will follow as soon as we show that the suitable lifting problems have a solution. Thanks to Assertion $(1)$ of Lemma \ref{inj-lem}, there is a solution to any lifting problem defined by an element of $\sg \in \lcof(\muaij) \cap \lwe(\muaij)$ and an element $\beta \in \lfib(\muaij)$. Finally, using the lifting property of model structure $\muaij$,  there is a solution to any lifting problem defined by an element $\sg \in \lcof(\muaij) =\cof(\muaij)$ and an element $\beta \in \lfib(\muaij) \cap \lwe(\muaij) = \fib(\muaij) \cap \we(\muaij)$.
\end{proof}
An object $\Fc \in \mua$ satisfies the \emph{Segal condition} if the map $\pif$ is a weak equivalence. We've called such object $\Fc$ a Quillen-Segal $\Ub$-object (see \cite{Bacard_QS}). There is a connection between Quille-Segal objects and the fibrant objects in $\mua^{\bf{L}}_{inj}$: 
\begin{cor}\label{cor-qs-uobj-inj}
The fibrant objects in $\mua^{\bf{L}}_{inj}$ are the objects $\Fc= [\F^0,\F^1,\pif]$ such that $\F^0 \in \M$ is fibrant, $\F^1 \in \ag$ is fibrant and $\pif: \F^0 \to  \Ub(\F^1)$ is a trivial fibration in $\M$. In particular $\Fc$ is a Quillen-Segal $\Ub$-object.
\end{cor}
\begin{proof}
	Let $\ast_{\ag}$ (resp. $\ast_{\M}$) be the  terminal object in $\ag$ (resp.  $\M$). Since $\Ub$ is a right adjoint, it preserves all limits, in particular it preserves the terminal object (obtained as the limit of the empty diagram). Therefore we can assume that $\Ub(\ast_{\ag})=\ast_{\M}$. The object $[\ast]=[\ast_{\M},\ast_{\ag}, \Id_{\ast_{\M}}]$ is a terminal object in $\mua$. An object $\Fc \in \mua^{\bf{L}}_{inj}$ is fibrant if and only if $\Fc \to [\ast]$ is a left injective fibration, if and only if  $\F^1 \twoheadrightarrow \ast_{\ag}$ is a fibration in  $\ag$  and if moreover $\F^0 \xtwoheadrightarrow{\sim} \Ub(\F^1)$ is a trivial fibration in $\M$. It follows that $\F^1 \in \ag$ is fibrant, and by composition $\F^0$ is also fibrant since the map $\F^0 \twoheadrightarrow \ast_{\M}$ is a fibration obtained as the composite of the trivial fibration $\pif: \F^0 \xtwoheadrightarrow{\sim} \Ub(\F^1)$ followed by the fibration $\Ub(\F^1 \twoheadrightarrow \ast_{\ag})=(\Ub(\F^1) \twoheadrightarrow \underbrace{\Ub(\ast_{\ag})}_{=\ast_{\M}})$.
\end{proof}
\begin{rmk}
It follows that if $\Fc$ is fibrant, then the map $\eta: \Fc \to \iota (\Piun(\Fc))$ defined by $[\pif, \Id_{\F^1}]$ is in fact a level-wise weak equivalence.
\end{rmk}
\subsubsection{Left projective model structure}
\begin{df}\label{left-projective-data}
	Let $\sigma: \Fc \xrightarrow{[\sgo,\sgi]} \Gc$ be a map in $\mua$. We will say that:
	\begin{enumerate}
		\item $\sigma$ is a \emph{left projective cofibration} if $\sg$ is a projective cofibration.
		\item $\sigma$ is a \emph{left weak equivalence} if  $\sigma^{1}$ is a weak equivalence  in $\ag$, that is if $\Pi^{1}(\sg)$ is a weak equivalence in $\ag$.
		
		\item $\sigma$ is a \emph{left projective fibration}  if:
		\begin{itemize}
			\item $\sg^0$ and $\sg^{1}$ are fibrations and if moreover

			\item the induced map $ \delta: \F^0 \to \Ub(\F^1) \times_{\Ub(\G^1)} \G^0 $
			is a \emph{weak equivalence} in $\M$. 
		\end{itemize}
	\end{enumerate}
\end{df}
\begin{nota}For simplicity we will also adopt the following notation.
	\begin{itemize}
\item $\lcof(\muapj)= $ the class of all left projective cofibrations
\item $\lfib(\muapj)= $ the class of all left projective fibrations
\item $\lwe(\muapj) = $ the class of all left weak equivalences.
	\end{itemize}
\end{nota}
\begin{thm}\label{left-proj-thm}
		Call a morphism $\sigma: \Fc \xrightarrow{[\sgo,\sgi]} \Gc$ in $\mua$
	\begin{enumerate}
		\item  a weak equivalence if it is a left projective weak equivalence,
		\item a cofibration if it is a left projective cofibration and 
		\item a fibration if it is a left projective fibration. 
	\end{enumerate}
	Then these choices provide $\mua$ with the structure of a model category that will be denoted by $\mua^{\bf{L}}_{proj}$. 
\end{thm}
A direct consequence is the following result which is the projective version of Theorem \ref{fact-inj-thm}.
\begin{thm}\label{fact-proj-thm}
	With the notation above, the following statements are true.
	\begin{enumerate}
        \item We have two Quillen equivalences:
        \begin{align*}
        L^{1}: \ag \leftrightarrows \mua^{\bf{L}}_{proj}: \Piun \hspace{1in} \Piun: \mua^{\bf{L}}_{proj} \leftrightarrows \ag : \iota
        \end{align*}
		\item We have a Quillen adjunction $\Fb^+ : \M  \rightleftarrows \mua^{\bf{L}}_{proj}: \Pio  $.
		\item We have a factorization $\Ub = \Pio \circ \iota$:
		$ \ag \xhookrightarrow[\sim]{\iota} \mua^{\bf{L}}_{proj} \xrightarrow{\Pio} \M.$
		\item We have  $\Id_{\ag} =\Piun \circ  \iota $ and the unit $\eta : \Id_{\mua^{\bf{L}}_{proj}} \to \iota \circ \Piun $ is a right homotopy.
	\end{enumerate}
\end{thm}

\begin{proof}[Proof of Theorem \ref{fact-proj-thm}]
	The proof is exactly the same as that of Theorem \ref{fact-inj-thm}, word for word, except that we replace $\mua^{\bf{L}}_{inj}$ by $\mua^{\bf{L}}_{proj}$.
\end{proof}
We can now give the lemmas that will be used to prove Theorem \ref{left-proj-thm}.
\begin{lem}\label{proj-lem-lift}

		Consider a lifting problem of solid arrows in $\mua$ as follows.
		\[
		\xy
		(0,18)*+{\Fc}="W";
		(0,0)*+{\Gc}="X";
		(30,0)*+{\Qc}="Y";
		(30,18)*+{\Pc}="E";
		{\ar@{->}^-{[\gamma^0, \gamma^{1}]}"X";"Y"};
		{\ar@{->}_-{[\sgo,\sgi]}"W";"X"};
		{\ar@{->}^-{[\theta^0,\theta^{1}]}"W";"E"};
		{\ar@{->}^-{[\beta^0,\beta^{1}]}"E";"Y"};
		{\ar@{-->}^-{[s^0,s^{1}]}"X";"E"};
		\endxy
		\]
		If $\sg=[\sgo,\sgi] \in \lcof(\muapj) \cap \lwe(\muapj)$ and  $\beta=[\beta^0,\beta^{1}] \in \lfib(\muapj)$ then a solution $s : \Gc \xrightarrow{[s^0,s^{1}]}\Pc$ exists.
\end{lem}
The proof of the last lemma is technical and a bit long, so we defer to the Appendix \ref{proof-proj-lem-lift}.
\begin{lem}\label{proj-lem-fact}
Any morphism $\sg : \Fc \to \Gc$ of $\mua$ can be factored as:
		$$[\F] \xhookrightarrow[\sim]{i} \Ec \xtwoheadrightarrow{j}[\G],$$
where $i \in  \lcof(\muapj) \cap \lwe(\muapj)$ and $j \in \lfib(\muapj)$. 
\end{lem}
The proof of this lemma is also long so we defer it to the Appendix \ref{proof-proj-lem-fact}.

\begin{lem}\label{proj-lem-tfib}
We have an equality of classes:
	 $$\lfib(\muapj) \cap \lwe(\muapj)= \fib(\muapj) \cap \we(\muapj).$$ 
In other words, a map $\sg: \Fc \xrightarrow{[\sg^0,\sg^{1}]} \Gc$ is a left projective fibration and a left weak equivalence if and only if it is a level-wise trivial fibration.
\end{lem}
\begin{proof}
   The map $\sg$ is displayed by a commutative square in $\M$ from which the universal property of the pullback gives a canonical factorization of $\sg^0$:
	$$\sg^0 = \F^0 \xrightarrow{\delta} \Ub(\F^1) \times_{\Ub(\G^1)} \G^0 \xrightarrow{\pig^\ast(\Ub(\sg^{1}))} \G^0.$$
	Here $\Ub(\F^1) \times_{\Ub(\G^1)} \G^0 \xrightarrow{\pig^\ast(\Ub(\sg^{1}))} \G^0$ is the base change of the map $\Ub(\sg^{1})$ along $\pig$. It follows that if $\sg$ is simultaneously a left projective fibration and a left weak equivalence, then $\sg$ is in particular a level-wise fibration such that $\sg^{1}$ is a trivial fibration. The map $\Ub(\sg^{1})$ is a trivial fibration as the image of a trivial fibration  under the right Quillen functor $\Ub$. Moreover, since the class of trivial fibrations is closed under base change we find that the canonical map $\pig^\ast(\Ub(\sg^{1}))$ is also a trivial fibration. The other part of being a left fibration is that $\delta$ is a weak equivalence which implies $\sg^0$ is also weak equivalence by composition. Consequently, $\sg$ is a level-wise fibration and a level-wise weak equivalence hence the inclusion:
	 $$\lfib(\muapj) \cap \lwe(\muapj)\subseteq \fib(\muapj) \cap \we(\muapj). $$
	With the same reasoning if $\sg$ is level-wise trivial fibration, then $\sg$ is already a left weak equivalence and a level-wise fibration, therefore $\sg$ will be a left projective fibration as soon as we show  that  $\delta$ is a weak equivalence. In the above factorization $\sg^0= \pig^\ast(\Ub(\sg^{1})) \circ \delta$, the two maps $\sg^0$ and $\pig^\ast(\Ub(\sg^{1}))$ are weak equivalences, therefore by $3$-for-$2$, $\delta$ is also a weak equivalence which means that $\sg$ is a left projective fibration as desired. This gives the other inclusion:
	$$\fib(\muapj) \cap \we(\muapj) \subseteq \lfib(\muapj) \cap \lwe(\muapj). $$
\end{proof}
A direct consequence of this lemma is the following obvious result.
\begin{lem}\label{old-fact-proj}
Any morphism $\sg : \Fc \to \Gc$ of $\mua$ can be factored as:
$$[\F] \xhookrightarrow{i} \Ec \xtwoheadrightarrow[\sim]{j}[\G],$$
where $i \in  \lcof(\muapj)$ and $j \in \lfib(\muapj) \cap \lwe(\muapj)$. 	
\end{lem}
\begin{proof}
	Indeed $\lcof(\muapj) = \cof(\muapj)$ by definition. Use the axiom of the model category $\muapj$ to factor any map as a cofibration followed by a  projective (=level-wise) trivial fibration  and Lemma \ref{proj-lem-tfib}.
\end{proof}
With the previous lemmas at hand, we can now give a proof of Theorem \ref{left-proj-thm}.
\begin{proof}[Proof of Theorem \ref{left-proj-thm}]
The classes $\lcof(\muapj)$, $\lfib(\muapj)$ and $ \lwe(\muapj)$ are closed under composition and retracts. Lemma \ref{proj-lem-fact}  and Lemma \ref{old-fact-proj} give the required factorizations. Moreover, any lifting problem defined by a left projective cofibration and a left projective trivial fibration admits a solution in the original model structure $\muapj$ since the class of left cofibration (resp. of left trivial fibrations) coincides with the class of cofibrations (resp. of trivial fibrations) in $\muapj$. The hardest part is the existence of a solution to a lifting problem defined by a map $\sg \in \lcof(\muapj) \cap \lwe(\muapj)$ and a map $\beta \in \lfib(\muapj)$ for which Lemma \ref{proj-lem-lift} guarantees the existence of such a solution
\end{proof}

\begin{cor}\label{cor-qs-uobj-proj}
	The fibrant objects in $\mua^{\bf{L}}_{proj}$ are the objects $\Fc= [\F^0,\F^1,\pif]$ such that $\F^0 \in \M$ is fibrant, $\F^1 \in \ag$ is fibrant and $\pif: \F^0 \to  \Ub(\F^1)$ is a weak equivalence in $\M$. In particular $\Fc$ is a Quillen-Segal $\Ub$-object.
\end{cor}
\begin{proof}
The same proof as in Corollary \ref{cor-qs-uobj-inj} considering the description of the fibrations in $\mua^{\bf{L}}_{proj}$.
\end{proof}
\subsection{Right Bousfield localizations}
\noindent Our goal here is to localize the two original model structures $\muaij$ and $\muapj$ along the functor $\Pio : \mua \to \M$. Rather than building the new model structures from zero, we are going to ``dualize'' the previous theorems using the involution $(-)^{\op}: \modcat_r^{op} \to \modcat_r$. Indeed, for any model category $\C$, the opposite category $\C^{\op}$ has a model structure and  we have an equality of model categories $(\C^{\op})^{\op} = \C$ (see \cite{Hov-model} for details). This functor maps a right Quillen functor $(\Ub:\ag \to \M)$ to the right Quillen functor $(\Fb^{op}: \M^{op} \to \ag^{op})$. 
\begin{nota} For a matter of clarity we will use some additional notation.
	\begin{itemize}
		\item $\fdua=$ the comma category whose objects are triples $\Xc= [\Xa^0, \Xa^{1}, \varepsilon ] \in \M \times \ag \times \Ar(\ag)$, with $\Fb \Xa^0 \xrightarrow{\varepsilon} \Xa^{1}$.
		\item $\Ub^{\op}: \ag^{\op} \to \M^{\op}$, $\Fb^{\op} : \M^{\op} \to \ag^{\op} $ are the induced functor between the dual categories.
		\item $(\Ub^{\op} \downarrow \M^{\op})$ and $(\ag^{\op} \downarrow \Fb^{\op})$ are the corresponding comma categories.
		\item $\Pio : (\ag^{\op} \downarrow \Fb^{\op}) \to \ag^{\op}$.
		\item $\Piun: (\ag^{\op} \downarrow \Fb^{\op}) \to \M^{\op}$.
	\end{itemize}
\end{nota}
The  result hereafter is obvious.
\begin{prop}\label{prop-dual-mod}
With the above notation the following hold. 
\begin{enumerate}
	\item  We have a Quillen adjunction $\Ub^{\op} \dashv \Fb^{\op}$, where $\Fb^{\op}$ is right Quillen.
	\item We have an isomorphism of categories $\mdua \xrightarrow{\cong} \fdua$ that maps $[\F^0,\F^1,\pif] \mapsto [\F^0,\F^1,\varphi(\pif)]$.
	\item Similarly, we have an isomorphism of categories $(\Ub^{\op} \downarrow \M^{\op}) \xrightarrow{\cong} (\ag^{\op} \downarrow \Fb^{op})$.
	\item There are isomorphisms of categories:
	$\mdua \cong  (\Ub^{\op} \downarrow \M^{\op})^{\op} \cong (\ag^{\op} \downarrow \Fb^{op})^{\op}.$
	\item Under the last isomorphism $\mdua \cong (\ag^{\op} \downarrow \Fb^{op})^{\op}$ we have an isomorphism of functors:
	$ (\Pio : \mdua \to \M) \cong (\Piun : (\ag^{\op} \downarrow \Fb^{\op}) \to \M^{\op})^{\op},$
	that is $\Pio \cong (\Piun)^{\op}$.
\end{enumerate}
\end{prop}
\begin{proof}
	The image of  an adjunction $(\Fb \dashv \Ub)$ under the dual functor is the adjunction $(\Ub^{op} \dashv \Fb^{op})$. In particular if $\Ub$ is right Quillen functor then $\Ub^{op}$ is a left Quillen functor. The first three assertions are easily verified to be true. The isomorphism $\mdua \cong  (\Ub^{\op} \downarrow \M^{\op})^{\op}$ maps $[\F^0,\F^1,\pif] \mapsto [\F^1,\F^0,\pif^{\op}]$ and $\sg=[\sg^0,\sg^{1}] \mapsto ([{(\sg^{1})}^{\op},{(\sg^0)}^{\op}])^{\op}$. One can also easily verify the other assertions.
\end{proof}

\subsubsection{Right injective model structure}
\begin{df}\label{right-injective-data}
	Let $\sigma: \Fc \xrightarrow{[\sgo,\sgi]} \Gc$ be a map in $\mua$. We will say that:
	\begin{enumerate}
		\item $\sigma$ is a right injective cofibration if:
		\begin{itemize}
			\item $\sigma^0$ is a  cofibration in $\M$, $\sigma^{1}$ is a  cofibration in $\ag$,
			\item and if moreover the induced map $ \F^1 \cup^{\Fb(\F^0)} \Fb(\G^0) \to  \G^1 $ is a weak equivalence in   $\ag$. 
		\end{itemize}
		
		\item $\sigma$ is a right weak equivalence if  $\Pi^0(\sg)=\sigma^0$ is a weak equivalence  in $\M$.
		
		\item $\sigma$ is a right injective fibration if:
		\begin{itemize}
			\item $\sigma^{1}: \F^1 \to \G^1$ is a  fibration in $\ag$,
			\item   and if moreover the induced map $ \F^0 \to \Ub(\F^1) \times_{\Ub(\G^1)} \G^0 $
			is a fibration in $\M$. 
		\end{itemize}
	\end{enumerate}
\end{df}

\begin{thm}\label{right-inj-thm}
	Call a morphism $\sigma: \Fc \xrightarrow{[\sgo,\sgi]} \Gc$ in $\mua$:
	\begin{itemize}
	\item  a weak equivalence if and only if it is a right weak equivalence.
	\item  a cofibration if it is a right injective cofibration. 
	\item a fibration if it is a right injective fibration. 
\end{itemize}
	Then these choices provide $\mua$ with the structure of a model category that will be denoted by $\mua^{\bf{R}}_{inj}$. 
\end{thm}
\begin{proof}
Consider the right Quillen functor $\Fb^{\op}: \M^{\op} \to \ag^{op}$. We can apply Theorem \ref{left-proj-thm} with respect to ``$\Ub$'' $ = \Fb^{\op}$, ``$\M$'' $= \ag^{\op}$ and ``$\ag$'' $= \M^{\op}$. According to our notation therein, this gives the model category $\ag^{\op}_{\Fb^{\op}}[\M^{\op}]^{\bf{L}}_{\tx{proj}} := (\ag^{\op} \downarrow \Fb^{\op})^{\bf{L}}_{\tx{proj}}$. This model category is a localization of the projective model structure. With the dual functor, we get a model structure on $\mua:= \mdua$ under the  isomorphism $\mdua \cong (\ag^{\op} \downarrow \Fb^{op})^{\op}$:
$$ \mua^{\bf{R}}_{inj} \cong ((\Ub^{\op} \downarrow \M^{\op})^{\bf{L}}_{proj})^{\op} \cong ((\ag^{\op} \downarrow \Fb^{op})^{\bf{L}}_{proj})^{\op} $$
As $\Pio \cong (\Piun)^{\op}$, one immediately sees that the class of weak equivalences in $\mua^{\bf{R}}_{inj} \cong  ((\Ub^{\op} \downarrow \M^{\op})^{\bf{L}}_{proj})^{\op} $ coincides with the class of right weak equivalences. Moreover, with a tedious but straightforward checking one can show the class of cofibrations (resp. of fibrations) therein coincides the class of right injective cofibrations (resp. fibrations). To check this, one needs to keep in mind that given any (model) category $\C$, limits in $\C$ corresponds to colimits in $\C^{op}$ and vice-versa. In particular there is a mutual correspondence between pullback squares in $\C$ and pushout squares in $\C^{op}$.
\end{proof}
\begin{thm}\label{right-fact-inj-thm}
For any right Quillen functor $\Ub : \ag \to \M$,  the following statements are true.
	\begin{enumerate}
		\item  We have two Quillen equivalences:
		 \begin{align*}
		 \Fb^+:\M  \rightleftarrows \mua^{\bf{R}}_{inj} : \Pio \hspace{1in} \Pio: \mua^{\bf{R}}_{inj}  \rightleftarrows  \M : R^0
		 \end{align*}
		\item  We have a Quillen adjunction:   $\Piun: \mua^{\bf{R}}_{inj} \leftrightarrows \ag : \iota$.
		\item We have a factorization $\Ub = \Pio \circ \iota$:
		$\quad  \ag \xhookrightarrow{\iota} \mua^{\bf{R}}_{inj} \xrightarrow[\sim]{\Pio} \M.$
		\item We have a $\Id_{\M} = R^0 \circ \Pio $  and the unit $\eta: \Id_{\mua^{\bf{R}}_{inj}} \to R^0 \circ \Pio$ is a homotopy. 
	\end{enumerate}
\end{thm}
\begin{proof}[Proof of Theorem \ref{right-fact-inj-thm}]
	For $m \in \M $ and $\Fc=[\F^0,\F^{1},\pif] \in \mua$ we have $\Fb^+(m)= [m, \Fb(m), \eta_m]$ and $\Pio(\Fc)= \F^0$. The functor $\Pio$ creates weak equivalences in $\mua^{\bf{R}}_{inj}$, therefore a map $\Fb^+(m) \to \Fc$ is a weak equivalence in $\mua^{\bf{R}}_{inj}$ if and only if -- by definition -- the map $m \to \Pio(\F)$ is a weak equivalence in $\M$. This is true in particular if $\Fc$ is fibrant and if $m$ is cofibrant, which proves that the adjunction $\Fb^+ \dashv \Pio$ is a Quillen equivalence, and that  $\Pio$ is a right Quillen equivalence. It also proves automatically that the other adjunction $\Pio \dashv R^0$ is a Quillen equivalence where $\Pio$ is a left Quillen functor. This gives Assertions $(1)$.
	For any right injective (trivial) cofibration $\sg$, $\Piun(\sg)=\sg^{1}$ is a (trivial) cofibration in $\ag$, thus $\Piun$ is left Quillen and Assertion $(2)$ follows. Assertion $(3)$ is obvious. Assertion $(4)$ will follow if we show that the map $\eta: \Fc \to R^0\circ \Pio(\F)$ is a weak equivalence for any fibrant object $\Fc \in \mua_{inj}^{\bf{R}}$. By definition we have $R^0\circ \Pio(\Fc)= [\F^0, \ast_{\ag}, \F^0 \xrightarrow{!} \overbrace{\Ub(\ast_{\ag})}^{\ast_{\M}}]$. The map $\eta$ is given by the couple $[\Id_{\F^0}, \G^1 \xrightarrow{!} \ast_{\ag}]$. As $\Id_{\F^0}$ is a weak equivalence, it follows that $\eta$ is a weak equivalence in $\mua^{\bf{R}}_{inj}$.
\end{proof}
\subsubsection{Right projective model structure}
\begin{df}\label{right-projective-data}
	Let $\sigma: \Fc \xrightarrow{[\sgo,\sgi]} \Gc$ be a map in $\mua$. We will say that:
	\begin{enumerate}
		\item $\sigma$ is a right projective cofibration if:
		\begin{itemize}
			\item $\sigma^0$ is a  cofibration in $\M$ and if
			\item  the induced map $ \F^1 \cup^{\Fb(\F^0)} \Fb(\G^0) \to  \G^1 $ is a trivial cofibration in   $\ag$. 
		\end{itemize}
		
		\item $\sigma$ is a right weak equivalence if  $\sigma^0$ is a weak equivalence  in $\M$, that is if $\Pi^0(\sg)$ is a weak equivalence in $\M$.
		
		\item $\sigma$ is a right projective fibration if it is a level-wise fibration.
	\end{enumerate}
\end{df}
\begin{thm}\label{right-proj-thm}
Call a morphism $\sigma: \Fc \xrightarrow{[\sgo,\sgi]} \Gc$ in $\mua$:
	\begin{itemize}
		\item a weak equivalence if and only if it is a right weak equivalence.
		\item a cofibration if it is a right projective cofibration. 
		\item a fibration if it is a right projective fibration. 
	\end{itemize}
	Then these choices provide $\mua$ with the structure of a model category that will be denoted by $\mua^{\bf{R}}_{proj}$. 
\end{thm}
\begin{proof}
	We dualize the left injective model structure on $(\ag^{\op} \downarrow \Fb^{op})$ under the isomorphism $\mdua \cong (\ag^{\op} \downarrow \Fb^{op})^{\op}$:
	$$ \mua^{\bf{R}}_{proj} \cong ((\Ub^{\op} \downarrow \M^{\op})^{\bf{L}}_{inj})^{\op} \cong ((\ag^{\op} \downarrow \Fb^{op})^{\bf{L}}_{inj})^{\op}.$$
\end{proof}

The projective version of Theorem \ref{right-fact-inj-thm} is:
\begin{thm}\label{right-fact-proj-thm}
	Let $\Ub : \ag \to \M$ be a right Quillen functor. Then the following statements are true.
	\begin{enumerate}
		\item We have two Quillen equivalences:
		\begin{align*}
		\Fb^+:\M  \rightleftarrows \mua^{\bf{R}}_{proj} : \Pio \hspace{0.5in} \Pio: \mua^{\bf{R}}_{proj}  \rightleftarrows  \M : R^0
		\end{align*}
		\item  We have a Quillen adjunction $\Piun: \mua^{\bf{R}}_{proj} \leftrightarrows \ag : \iota$.
		\item We have a factorization $\Ub = \Pio \circ \iota$:
		$ \quad \ag \xhookrightarrow{\iota} \mua^{\bf{R}}_{proj} \xrightarrow[\sim]{\Pio} \M.$
		\item We have a $\Id_{\M} = R^0 \circ \Pio $  and the unit $\eta: \Id_{\mua^{\bf{R}}_{proj}} \to R^0 \circ \Pio$ is a homotopy.
	\end{enumerate}
\end{thm}

\begin{proof}
The proof is exactly the same as that of Theorem \ref{right-fact-inj-thm}.
\end{proof}
\subsection{Proof of the main theorem}\label{proof-main-thm-fact}
\begin{lem}\label{funct-inj-fact}
Consider the following diagrams.
\[
\xy
(0,15)*+{\ag}="W";
(0,0)*+{\ag'}="X";
{\ar@{->}^-{H}"W";"X"};
(25,15)*+{\M}="U";
(25,0)*+{\M'}="V";
{\ar@{->}^-{\Ub'}"X";"V"};
{\ar@{->}_-{K}"U";"V"};
{\ar@{->}^-{\Ub}"W";"U"};
\endxy
\xy
(-5,7)*+{\Rightarrow}="L";
(0,15)*+{\ag}="W";
(0,0)*+{\ag'}="X";
(50,0)*+{\M'}="Y";
(50,15)*+{\M}="E";
{\ar@{->}^-{H}"W";"X"};
{\ar@{->}_-{K}"E";"Y"};
(25,15)*+{\mua}="U";
(25,0)*+{\mupa}="V";
{\ar@{->}^-{\iota}"X";"V"};
{\ar@{->}^-{\Pio}"V";"Y"};
{\ar@{.>}^-{E(H,K)}"U";"V"};
{\ar@{->}^-{\Pio}"U";"E"};
{\ar@{->}^-{\iota}"W";"U"};
\endxy
\]
For $ms \in \{ \tx{`$inj$', `$proj$'} \}$, the following hold.
\begin{enumerate}
	\item We have a Quillen adjunction $E(H,K)_\ast :   \mupa_{ms}^{\bf{L}} \rightleftarrows  \mua_{ms}^{\bf{L}} : E(H,K)$\\
	\medskip
	\item We have a Quillen adjunction $E(H,K)_\ast :   \mupa_{ms}^{\bf{R}} \rightleftarrows  \mua_{ms}^{\bf{R}} : E(H,K)$
\end{enumerate}
\end{lem}

\begin{proof}
To prove Assertion $(1)$ we will show that the left adjoint $E(H,K)_\ast$ preserves the cofibration and the trivial cofibrations.  By definition we have $\cof(\mupa_{ms}^{\bf{L}}) = \cof(\mupa_{ms})$ and $\cof(\mua_{ms}^{\bf{L}}) =\cof(\mua_{ms})$. By Theorem \ref{first-fact}, $E(H,K)_\ast$  is left Quillen functor therefore we have an inclusion: $E(H,K)_\ast[ \cof(\mupa_{ms})] \subseteq \cof(\mua_{ms}) $. This inclusion can also be written as:
$ E(H,K)_\ast[ \cof(\mupa_{ms}^{\bf{L}})] \subseteq \cof(\mua_{ms}^{\bf{L}})$, which means that $E(H,K)_\ast$  preserves the cofibrations.
Recall that a map  $\theta = [\theta^0, \theta^1]$  in $\mupa$  is a trivial cofibration $\mupa_{ms}^{\bf{L}}$  if and only if it is a cofibration in $\mupa_{ms}$ such that $\theta^1$  is a trivial cofibration in $\ag'$. The same description holds in $\mupa_{ms}^{\bf{L}}$. Furthermore, as $H_\ast$ is a left Quillen functor, it is clear that in the formula $E(H,K)_\ast(\theta)= [K_\ast(\theta^0), H_\ast(\theta^1)]$, the map $H_\ast(\theta^1)$ is a trivial cofibration if $\theta^1$ is. In other words, $E(H,K)_\ast(\theta)$ is a trivial cofibration in $\mua_{ms}^{\bf{L}}$ if $\theta$ is trivial cofibration in $\mupa_{ms}^{\bf{L}}$, whence the assertion.

We will prove Assertion $(2)$ by showing that $E(H,K)$ preserves the fibrations and the trivial fibrations. By definition, we have  $\fib(\mupa_{ms}^{\bf{R}}) = \fib(\mupa_{ms})$ and $\fib(\mua_{ms}^{\bf{R}}) =\fib(\mua_{ms})$, therefore $ E(H,K)[ \fib(\mupa_{ms}^{\bf{R}})] \subseteq \fib(\mua_{ms}^{\bf{R}})$ since we have a right Quillen functor: $E(H,K): \mua_{ms} \to \mupa_{ms}$. Moreover, a map $\sg= [\sgo, \sgi]$  is a  trivial fibration in  $\mupa_{ms}^{\bf{R}}$ if and only if it is a fibration such that $\sgo$ is a trivial fibration. We have the same characterization for the trivial fibration in $\mupa_{ms}^{\bf{R}}$. Since $K$ preserves the trivial fibrations, it is clear that in the formula $E(H,K)(\sg) = [K(\sgo), H(\sgi)]$, the map  $K(\sgo)$ is a trivial fibration if $\sgo$ is. It follows that $E(H,K)(\sg)$ is a trivial fibration $\mupa_{ms}^{\bf{R}}$ if $\sg$ is a trivial fibration in $\mua_{ms}^{\bf{R}}$ and the assertion follows.
\end{proof}
\noindent With the material of the previous section, we can now prove Theorem \ref{main-thm-fact}.
\begin{proof}[Proof of Theorem \ref{main-thm-fact}]
\hfill
\ \\
We will show that any morphism in $\modcat_r$ has two ``injective factorizations'' and two ``projective factorizations''. We will give a proof for the injective factorizations, the projective factorization are treated the same way. With the same notation as in the proof of Theorem \ref{first-fact}, given a Quillen adjunction $(\Fb, \Ub, \varphi)$ we consider the morphisms below:
\begin{itemize}
	\item $G_1= (G_1^l, G_1^r,\varphi) = (\Piun, \iota, \varphi) = \iota_{(\Pi^{1},\varphi)} \in \Hom(\ag, \mua_{inj}^{\bf{L}})$
	\item $G_2= (G_2^l,G_2^r,\varphi) = (\Fb^+, \Pio, \varphi) = \Pio_{(\Fb^+,\varphi)} \in \Hom( \mua_{inj}^{\bf{L}}, \M)$
\end{itemize}
By Theorem \ref{fact-inj-thm}, we have $G_1 \in \Lci_w(\modcat_r) \cap Mor_{inj_{ob}}(\modcat_r)$ and by Proposition \ref{fact-ri-li} we have $G_2 \in Mor_{ifib} (\modcat_r)$. Clearly, we have $\Ub_{(\Fb,\varphi)} = G_2 \circ G_1$ and  the factorization is functorial by Lemma \ref{funct-inj-fact}. This proves Assertion $(1)$. \\
To prove Assertion $(2)$ we consider the two maps below:
\begin{itemize}
	\item $G_1= (G_1^l, G_1^r,\varphi) = (\Piun, \iota, \varphi) = \iota_{(\Pi^{1},\varphi)} \in \Hom(\ag, \mua_{inj}^{\bf{R}})$
	\item $G_2= (G_2^l,G_2^r,\varphi) = (\Fb^+, \Pio, \varphi) = \Pio_{(\Fb^+,\varphi)} \in \Hom( \mua_{inj}^{\bf{R}}, \M)$
\end{itemize}
By Proposition \ref{fact-ri-li} we have $G_1 \in Mor_{inj_{ob}}(\modcat_r)$ and with  Theorem \ref{right-fact-inj-thm} we have $G_2 \in \Rci_w(\modcat_r) \cap Mor_{ifib} (\modcat_r)$. We have $\Ub_{(\Fb,\varphi)} = G_2 \circ G_1$ with the functoriality of the factorization also given by Lemma \ref{funct-inj-fact}.
With the same reasoning we get the projective factorizations with $\mua_{proj}^{\bf{L}}$, $\mua_{proj}^{\bf{R}}$ using  Theorem \ref{fact-proj-thm} and Theorem \ref{right-fact-proj-thm} instead of Theorem \ref{fact-inj-thm} and Theorem \ref{right-fact-inj-thm}, respectively.
\end{proof}

\section{Homotopy extension and lifting problem}\label{sec-lifting}
\begin{lem}\label{lem-lifting}
Consider a lifting problem of solid arrows in $\modcat_r$ as follows.
\[
\xy
(0,18)*+{\C}="W";
(0,0)*+{\D}="X";
(30,0)*+{\Ba}="Y";
(30,18)*+{\Aa}="E";
{\ar@{->}^-{\Phi^1}"X";"Y"};
{\ar@{->}_-{G}"W";"X"};
{\ar@{->}^-{\Phi^0}"W";"E"};
{\ar@{->}^-{K}"E";"Y"};
{\ar@{-->}^-{}"X";"E"};
\endxy
\]
\begin{enumerate}
	\item If $G \in \Lci_w(\modcat_r)$ and $K \in \Mor(\modcat_r)$ there is a solution $T: \D \to \Aa$ \emph{up-to-homotopy}, in that we have $T \circ G = \Phi^0$ and there is a homotopy $h: \Phi^1 \xrightarrow{\sim} K\circ T$.
	\item  If $G \in Mor(\modcat_r)$ and $K \in \Rci_w(\modcat_r)$ there is a solution $T: \D \to \Aa$ \emph{up-to-homotopy}, in that we have $K \circ T = \Phi^1$ and there is a homotopy $h: \Phi^0 \xrightarrow{\sim} T\circ G$.
\end{enumerate}
\end{lem}
\begin{proof}
By assumption there is a retraction $H: \D \to \C$ of $G \in \Lci_w(\modcat_r)$ equipped with a homotopy $\tau: \Id_{\D} \xrightarrow{\sim} G \circ H$. We shall prove that $T = \Phi^0 \circ H$ is a solution up-to-homotopy. Using the fact $H$ is a retraction of $G$ one clearly has: $T \circ G = \Phi^0$. By inspection, the horizontal composite $ \Id_{\Phi^1} \otimes \tau$ is a right homotopy whose domain is $\Phi^1 \circ \Id_{\D} = \Phi^1$ and whose  codomain is $\Phi^1 \circ G \circ H = K  \circ \Phi^0 \circ H = K \circ T$. In other words, we have found a homotopy $h = \Id_{\Phi^1} \otimes \tau$ between $\Phi^1$ and $K \circ T$ which proves Assertion $(1)$. To prove Assertion $(2)$, consider a weakly invertible section $H$ of $K$ equipped with a homotopy $\tau:\Id_{\Aa} \to H \circ K$. Then we claim that $T = H \circ \Phi^1$ is a solution up-to-homotopy. To see this, first observe that $H$ being a section of $K$ implies that  $K \circ T = K \circ H \circ \Phi^1 = \Phi^1$.  Moreover, it can be easily checked that the homotopy $h = \tau \otimes \Id_{\Phi^0}$ is a homotopy whose domain is $\Phi^0$ and whose codomain is $T \circ G$.
\end{proof}
\begin{lem}\label{lem-lift-2}
With the same notation, the following hold.
\begin{enumerate}
	\item If $G \in \Lci_{wrel}(\modcat_r)$ and $K \in \Mor(\modcat_r)$ there is a solution $T: \D \to \Aa$ \emph{up-to-homotopy}, in that we have $T \circ G = \Phi^0$ and there is a homotopy $h: \Phi^1 \xrightarrow{\sim} K\circ T$ relative to $G$.
	\item  If $G \in Mor(\modcat_r)$ and $K \in \Rci_{wcor}(\modcat_r)$ there is a solution $T: \D \to \Aa$ \emph{up-to-homotopy}, in that we have $K \circ T = \Phi^1$ and there is a homotopy $h: \Phi^0 \xrightarrow{\sim} T\circ G$ co-relative to $K$.
\end{enumerate}
\end{lem}
\begin{proof}
In the proof of Lemma \ref{lem-lifting}, if $\tau$ is homotopy relative to $G$ then so is  $h \circ G$ by inspection. This gives Assertion $(1)$. Dually, $K \circ \tau$ is a homotopy co-relative to $K$ if $\tau$ is, which gives Assertion $(2)$.
\end{proof}
\begin{df}
	In the diagram of Lemma \ref{lem-lifting}, when a solution up-to-homotopy exists we will say that the map $G$ has the h-LLP against $K$ and that $K$ has the h-RLP against $G$.
\end{df}

\begin{prop}
	In the category $\modcat_r$, the following hold.
\begin{enumerate}
	\item Any map $K : \Aa \to \Ba$ that posses the $h$-RLP against every map $G= (G^l,G^r,\varphi)$ such that $G^r$ is injective on objects,  is a Quillen equivalence.
	\item Any map $G: \C \to \D$ that posses the $h$-LLP against every map in $G= (G^l,G^r,\varphi)$ such that $G^r$ is an isofibration, is a Quillen equivalence.
\end{enumerate}

\begin{proof}
We use the retract argument for (right) Quillen functors to prove this. We will only give the proof of Assertion $(1)$, the argument is the same for Assertion $(2)$. Assume that $K : \Aa \to \Ba$  posses the $h$-RLP against every map in $\Lci(\modcat_r)$. By Theorem \ref{main-thm-fact}, there is a factorization $K= G_2 \circ G_1$ such that $G_1 \in \Lci(\modcat_r)$ and $G_2 \in \Rci_w(\modcat_r)$. Consider the commutative square of solid arrows as follows.
\[
\xy
(0,18)*+{\Aa}="W";
(0,0)*+{\D}="X";
(30,0)*+{\Ba}="Y";
(30,18)*+{\Aa}="E";
{\ar@{->}^-{G_2}"X";"Y"};
{\ar@{->}_-{G_1}"W";"X"};
{\ar@{->}^-{\Id}"W";"E"};
{\ar@{->}^-{K}"E";"Y"};
{\ar@{-->}^-{}"X";"E"};
\endxy
\]
There is a solution $T: \D \to \Aa$ up-to-homotopy such that $G_2= K \circ T$ and a homotopy $\tau: \Id_{\Aa} \to T \circ G_1$. We now use the same diagram as in the retract argument, except that the top horizontal composite is not the identity but homotopic to the identity:
\[
\xy
(0,15)*+{\A}="W";
(0,0)*+{\Ba}="X";
(50,0)*+{\Ba}="Y";
(50,15)*+{\A}="E";
{\ar@{->}^-{K}"W";"X"};
{\ar@{->}^-{K}"E";"Y"};
(25,15)*+{\D}="U";
(25,0)*+{\Ba}="V";
{\ar@{->}^-{\Id}"X";"V"};
{\ar@{->}^-{\Id}"V";"Y"};
{\ar@{->}^-{G_2}"U";"V"};
{\ar@{->}^-{T}"U";"E"};
{\ar@{->}^-{G_1}"W";"U"};
\endxy
\] 
 We have an explicit weak inverse of $K$ using the existing weak inverse $H_2$ of $G_2$. Indeed the composite $T \circ H_2$ is a weak inverse as we are going to explain. Clearly, we have an equality $K\circ (T \circ H_2) = G_2 \circ H_2 = \Id_{\Ba}$. There is a homotopy $\Id_{\Aa} \to (T\circ H_2) \circ K$ obtained as the composite:
$$\Id_{\Aa} \xrightarrow{\tau} \underbrace{T \circ G_1}_{= T \circ \Id_{\D} \circ G_1} \xrightarrow{\Id \otimes \tau(G_2) \otimes \Id } T \circ H_2 \circ G_2 \circ G_1 = (T\circ H_2) \circ K.$$
This proves that $K \in  \Rci_w(\modcat_r)$ which implies that $K$ is a Quillen equivalence by Proposition \ref{prop-retract-Quillen}.
\end{proof}
\end{prop}
\section{Homotopy theory of Quillen-Segal algebras}\label{sec-qs-alg}
\noindent In this section we assume that $\M$ is a (monoidal) model category that is also cofibrantly generated. We let $\ag = \oalg(\M) $ be the category of $\O$-algebras, where $\O$ is an operad, monad, properad or  a PROP. We denote by $\Ub: \oalg(\M) \to \M$ the forgetful functor and by $\mualg= \mdua$. We remind the reader that the adjunction $\Fb \dashv \Ub$ gives an adjunction $\Gamma: \arm \leftrightarrows \mualg : \Piar$ (see \cite{Bacard_QS}). 
Call a morphism $\sg: \Fc \xrightarrow{\osg} \Gc$ of $\mualg$:
\begin{itemize}
	\item a $\Piar$-injective fibration if $\Piar(\sg)=[\sgo, \Ub(\sgi)]$ is a fibration in $\armij$.
	\item a $\Piar$-projective fibration if $\Piar(\sg)$ is a fibration in $\armpj$
	\item a $\Piar$-weak equivalence if $\Piar(\sg)$ is a level-wise weak equivalence in $\arm$.
	\item a  $\Piar$-injective (resp. $\Piar$-projective) cofibration it possesses the LLP against any map that is simultaneously a $\Piar$-injective  (resp. a $\Piar$-projective) fibration and a $\Piar$-weak equivalence.
\end{itemize}
If the $\Piar$-injective data determine a model structure, we will say that the \emph{right induced injective model structure} on $\mualg$ exists. Similarly, if the $\Piar$-projective data define a model structure we will say that the \emph{right induced projective model structure} exists.
We generalize a result in \cite{Bacard_QS} that was established under the hypothesis that $\M$ is combinatorial.

\begin{thm}\label{simultan-exist}
The following are equivalent. 
\begin{enumerate}
	\item The right induced model structure on $\oalg(\M)$ exists.
	\item  The right induced injective model structure on $\mualg$ exists.
	\item  The right induced projective model structure on $\mualg$ exists.
\end{enumerate}
\end{thm}
\begin{proof}
	\hfill
	\ \\
$(1) \Longrightarrow (2)$:  this is given by Theorem \ref{inj-proj-thm}. Indeed, the weak equivalences and the fibrations in $\mualg_{inj}$ coincide with the $\Piar$-weak equivalences and the $\Piar$-fibrations respectively. Therefore the right induced injective model structure coincide with the (simply) injective model structure of Theorem \ref{inj-proj-thm} applied to $\Ub : \oalg(\M) \to \M$.  \\
$(2) \Longrightarrow (3)$: we create the projective model structure through the adjunction $\Gamma \dashv \Piar$. By the classical argument of right induced model structure it suffices to show that for any element  $j \in \Mor(\armpj)$ of the generating set for the projective trivial cofibrations, then any cobase change of $\Gamma(j)$ is a weak equivalence in $\mualg$. We note that $j$ is also an injective trivial cofibration since the identity $\Id :\armpj \to \armij$ is a left Quillen functor. With the Quillen adjunction $\Gamma: \armij \leftrightarrows \mualg_{inj}: \Piar$ , $\Gamma(j)$ is an injective trivial cofibration in $\mualg_{inj}$ whose cobase change is necessarily a weak equivalence in $\mualg_{inj}$, hence in $\mualg_{proj}$ as desired.\\
$(3) \Longrightarrow (1)$: this is also easy, it suffices to show that if $f$ is an element of the generating set of the trivial cofibration in $\M$, then the cobase change of $\Fb(f)$ along any attaching map $u$ is a weak equivalence. We use the same argument as in \cite[Theorem 6.1]{Bacard_QS}. With the adjunctions $(L^{1} \dashv \Piun)$ and $(\Fb \dashv \Ub)$, it can be easily seen that the map $L^{1}(\Fb(f))$ has the LLP against any $\Piar$-projective fibration, thus $L^{1}(\Fb(f))$ is a $\Piar$-projective trivial cofibration. Given a pushout diagram in $\oalg(\M)$, its image under $L^{1}$ is also a pushout diagram in $\mualg$ since $L^{1}$ preserves colimits (as does any left adjoint). We note that $\Piun \circ L^{1} = \Id_{\oalg(\M)}$ and  colimits are computed level-wise, then the cobase change of $\Fb(f)$ along  $u$ is the image under $\Piun$ of the cobase change of $L^{1}(\Fb(f))$ along the attaching map $L^{1}(u)$. Since $L^{1}(\Fb(f))$ is a $\Piar$-projective trivial cofibration, the cobase change of it along any map is a $\Piar$-weak equivalence, which implies that the cobase change of $\Fb(f)$ along the attaching map $u$ is also a weak equivalence as desired.
\end{proof}
Say that a morphism $\sg: \Fc \xrightarrow{\osg} \Gc$ is:
\begin{itemize}
	\item a new injective (resp. projective) fibration if it is a  left injective fibration in the sense of Definition \ref{left-injective-data} (resp. Definition \ref{left-projective-data}).
	\item a new weak equivalence if $\Ub(\sgi)$ is a weak equivalence in $\M$
	\item a new injective (resp. projective) cofibration if it has the LLP against any map that is simultaneously a new injective (resp. projective) fibration and a new weak equivalence.
\end{itemize}
If the new injective data (resp. projective data) determine a model structure on $\mualg$, we will say that the $\Piun$-localized injective (resp. projective) model structure exists.
\begin{cor}
	The following are equivalent.
	\begin{enumerate}
		\item The right induced model structure on $\oalg(\M)$ exists.
		\item The $\Piun$-localized injective model structure on $\mualg$ exists.
		\item The $\Piun$-localized projective model structure on $\mualg$ exists.
	\end{enumerate}
\end{cor}
\begin{proof}
	We will show directly that $(1) \Longleftrightarrow (2)$ and that $(1) \Longleftrightarrow (3)$.\\
 $(1) \Longrightarrow (2)$: From the right induced model structure on $\oalg(\M)$ we build successively the model categories $\mualg_{inj}$ (Theorem \ref{inj-proj-thm}) then  $\mualg_{inj}^{\bf{L}}$ (Theorem \ref{left-inj-thm}). The new injective fibrations and the new weak equivalences are precisely the left injective fibrations and the left injective weak equivalences, therefore the left injective model structure coincides with the $\Piun$-localized model structure. \\
$(1) \Longrightarrow (3)$: Same proof using  $\mualg_{proj}$ (Theorem \ref{inj-proj-thm}) and 
		 $\mualg_{proj}^{\bf{L}}$ (Theorem \ref{left-proj-thm}) instead of $\mualg_{inj}$ and $\mualg_{inj}^{\bf{L}}$ (Theorem \ref{left-inj-thm}) respectively.\\
$(2) \Longrightarrow (1)$ and $(3) \Longrightarrow (1)$: We reuse the argument in Theorem \ref{simultan-exist}. We must show that if $f$ is an element of the generating set for the trivial cofibrations in $\M$, then the cobase change of $\Fb(f)$ along any attaching map $u$ is a weak equivalence. First observe that
	a lifting problem in $\mualg$ defined by $L^{1}(\Fb(f))$ and $\sg=\osg$ is equivalent, by  $(L_1\dashv \Piun)$, to a lifting problem in $\oalg(\M)$ defined by $\Fb(f)$ and $\sgi$. The later problem is equivalent, by $(\Fb \dashv \Ub)$, to a lifting problem in $\M$ defined by $f$ and $\Ub(\sgi)$. Therefore if $\Ub(\sgi)$ is a fibration in $\M$, then there is a solution to the aforementioned lifting problems. By definition, the map $\Ub(\sgi)$ is a fibration for any new injective (resp. projective) fibration in $\mualg$, thus $L^{1}(\Fb(f))$ has the LLP against any new injective (resp. projective) fibration. In other words, $L^{1}(\Fb(f))$  is a trivial cofibration in the $\Piun$-localized injective (resp. projective)  model structure. In particular, the cobase change of $L^{1}(\Fb(f))$ along $L^1(u)$ is - necessarily - a $\Piun$-equivalence.  As observed earlier, pushouts are computed level-wise, $L^1$ preserves pushouts and we have $\Piun\circ L^1 = \Id_{\oalg(\M)}$. By definition of a $\Piun$-equivalence, we see that the cobase change of $\Fb(f)$ must be a weak equivalence of $\O$-algebras, whence Assertion $(1)$.
\end{proof}
\begin{cor}
There is a Quillen model structure for $\O$-algebras if and only if there is Quillen model structure for Quillen-Segal $\O$-algebras. Moreover, the homotopy theory of $\O$-algebras is Quillen equivalent to that of Quillen-Segal $\O$-algebras.
\end{cor}
\begin{proof}
By the previous results, the right induced model structure on $\oalg(\M)$ exists (if and only) if any (hence all) of the following model categories exists:
\begin{center}
$\mualg_{ms}$,  $\mualg_{ms}^{\bf{L}}$, for $ms \in \{ \tx{`inj', `proj'}\}$.
\end{center}
  By Theorem \ref{fact-inj-thm} and Theorem \ref{fact-proj-thm}  the functor $\iota : \oalg(\M) \xhookrightarrow{\sim} \mualg_{ms}^{\bf{L}}$ is a right Quillen equivalence for $ms \in \{ \tx{`inj', `proj'}\}$. Corollary \ref{cor-qs-uobj-inj} and Corollary \ref{cor-qs-uobj-proj} assert that the fibrant objects in $\mualg_{ms}^{\bf{L}}$ are Quillen-Segal algebras.
\end{proof}
\section{Factorization of a lax monoidal functor}\label{sec-fact-lax}
\noindent In this section we show that the factorization $\ag \xhookrightarrow{\iota} \mdua \xrightarrow{\Pi^0} \M$ holds naturally if $\Ub$ is a lax monoidal functor. The material of this section will be needed for our upcoming work on monoidal Quillen adjunction and related topics.
 The laxity map for $\Ub$ will be denoted by $\psi: \Ub(a) \otimes \Ub(b) \to \Ub(a\otimes b)$. The unit objects will be denoted respectively by $I_{\ag}$ and $I_\M$, and if there is no potential confusion we will simply write $I$ for both. We remind the reader that the data of a lax functor include a laxity map $\psi_I: I_{\M} \to \Ub(I_{\ag})$ which is furthermore subjected to the unit axioms (see \cite{Leinster_higher_op}). We will prove in the first place the following:
\begin{prop}\label{fact-lax}
Assume that $\Ub: \ag \to \M$ is a lax functor between monoidal categories. Then there is a monoidal structure on the category $\mdua$ such that:
\begin{enumerate}
	\item The functor $\iota :\ag \hookrightarrow \mdua$ is a lax monoidal functor
	\item The functor $\Pio: \mdua \to \M$  and $\Piun: \mdua \to \ag$ are strong monoidal functors.
	\item If in addition $\ag$ and $\M$ are  closed monoidal , then so is $\mdua$.
\end{enumerate}
\end{prop}
\noindent If $\Ub=\Id_{\M}$, then $\mdua$ is simply the arrow-category $\Arr(\M)$. For this case Hovey  \cite{Hovey_Arr} defined two monoidal (closed) structures on $\Arr(\M)$: the point-wise product and the pushout-product. In what follows we will consider for the moment the point-wise product.
\begin{prop}
Let $\Ub: \ag \to \M$ be a lax monoidal functor. The category $\mdua$ has a monoidal structure.
The tensor product of $\Fc$ and $\Gc$ is the object 
 $\Fc \otimes \Gc = [\F^0\otimes \G^0, \F^1 \otimes \G^1, \pi_{\Fc \otimes \Gc}]$ where $\pi_{\Fc \otimes \Gc}$ is the composite of the path:
$$ \F^0\otimes \G^0 \xrightarrow{\pif \otimes \pig} \Ub(\F^1) \otimes \Ub(\G^1) \xrightarrow{\psi} \Ub(\F^1 \otimes \G^1)$$
The unit object is $I=[I_{\M}, I_{\ag}, I_{\M} \xrightarrow{\psi_I}\Ub( I_{\ag})]$. 
\end{prop}

\begin{proof}
The proof is tedious but straightforward.  For the associativity of this product,  one simply uses the coherence axioms of for the laxity maps of $\Ub$ along with the associativity of the respective products in $\ag$ and $\M$. Given $\Fc, \Gc$ and $[\H]$ in $\mua$, the two paths going from $[\Ub(\F^1) \otimes\Ub(\G^1)] \otimes \Ub(\H^1)$ to  $\Ub([\F^1 \otimes \G^1] \otimes \H^1)$ which utilize $\psi$ are equal. After this it suffices to precompose with $[\pif \otimes \pig] \otimes \pi_{\H}$ and use the isomorphism  $[\pif \otimes \pig] \otimes \pi_{\H} \cong \pif \otimes [\pig \otimes \pi_{\H}]$ in $\Arr(\M)$. This will give the isomorphism $(\Fc \otimes \Gc)\otimes [\H] \cong \Fc \otimes (\Gc \otimes [\H])$.
The isomorphism $\Fc \otimes I \xrightarrow[\cong]{r} \Fc$  is given by pasting vertically the two commutative diagrams:
\[
\xy
(0,18)*+{\F^0\otimes I}="W";
(0,0)*+{\Ub(\F^1)\otimes I}="X";
(30,0)*+{\Ub(\F^1)}="Y";
(30,18)*+{\F^0}="E";
{\ar@{->}^-{r}_{\cong}"X";"Y"};
{\ar@{->}_-{\pif \otimes \Id_I}"W";"X"};
{\ar@{->}^-{r}_{\cong}"W";"E"};
{\ar@{->}^-{\pif}"E";"Y"};
\endxy
\xy
(0,18)*+{\Ub(\F^1)\otimes I}="W";
(0,0)*+{\Ub(\F^1)\otimes \Ub(I)}="X";
(30,0)*+{\Ub(\F^1 \otimes I )}="Y";
(30,18)*+{\Ub(\F^1)}="E";
(60,0)*+{\Ub(\F^1)}="F";
{\ar@{->}^-{\psi}"X";"Y"};
{\ar@{->}_-{\Id \otimes \psi_I}"W";"X"};
{\ar@{->}^-{r}_{\cong}"W";"E"};
{\ar@{->}^-{\Ub(r^{-1})}_{\cong}"E";"Y"};
{\ar@{->}^-{\Ub(r)}_{\cong}"Y";"F"};
{\ar@{->}^-{\Id}"E";"F"};
\endxy
\]
The diagram on the left commutes by naturality of $r: - \otimes I \xrightarrow{\cong} \Id $ while the one the right commutes by the coherence axiom of the lax functor $\Ub$ with respect to $\psi_I$ and $\Ub(\F^1)$. Explicitly, the map $\Fc \otimes I \xrightarrow[\cong]{r} \Fc$ is given component-wise the isomorphism in `$r$' in $\ag$ and $\M$. With the same argument we also have the isomorphism $I \otimes \Fc \xrightarrow[\cong]{l} \Fc$.
\end{proof}
\begin{rmk}
This monoidal structure is not necessarily symmetric unless $\Ub$ is a symmetric lax monoidal functor.
\end{rmk}
\begin{lem}\label{lem-lax}
With the notation above, the following statement are true.
\begin{enumerate}
	\item The functor $\iota: \ag \hookrightarrow \mua$ is a lax monoidal functor with laxity map $\psi: \iota(\Pa) \otimes \iota(\Qa) \to \iota(\Pa \otimes \Qa)$ given by the couple $[\psi^0, \psi^1]$, where:
	\begin{itemize}
		\item $\psi^0=\psi$ is the laxity map: $ \Ub(\Pa) \otimes \Ub(\Qa) \to \Ub(\Pa  \ot \Qa) $.
		\item $\psi^1= \Id_{\Pa \ot \Qa}$ 
	\end{itemize}
	\item The functor $\Pio: \mdua \to \M$ is a strong monoidal functor.
	\item The functor $\Piun: \mdua \to \ag$ is a strong monoidal functor.
\end{enumerate}
\end{lem}
\begin{proof}
	Assertion $(2)$ and $(3)$ are  obvious by inspection. For Assertion $(1)$ it boils down to checking that the couple $[\psi^0,\psi^1]$ defines a map $\iota(\Pa) \otimes \iota(\Qa) \to \iota(\Pa \otimes \Qa)$. This simply follows from the fact that we have commutative diagram:
	\[
	\xy
(0,18)*+{\Ub(\Pa)\otimes\Ub(\Qa)}="W";
(0,0)*+{\Ub(\Pa)\otimes \Ub(\Qa)}="X";
(30,0)*+{\Ub(\Pa \otimes \Qa )}="Y";
(30,18)*+{\Ub(\Pa \ot \Qa)}="E";
{\ar@{->}^-{\psi}"X";"Y"};
{\ar@{->}_-{\Id \otimes \Id}"W";"X"};
{\ar@{->}^-{\psi}"W";"E"};
{\ar@{->}^-{\Ub(\Id)= \pi_{\Pa \ot \Qa}}"E";"Y"};
{\ar@{.>}^-{\pi_{\iota(\Pa) \otimes \iota(\Qa)}}"W";"Y"};
\endxy
	\]
\end{proof}
\paragraph{\textbf{Hypothesis}} 
\noindent We assume that $\ag$ and $\M$ have a closed monoidal structure in the sense of Hovey \cite{Hov-model}, in that we have two adjunctions with two variables $\ot :\ag \times \ag \to \ag$, $\ot : \M \times \M \to \M$ with left and right internal hom $\HOM_l(\ast,-)$ and $\HOM_r(\ast, -)$. Our goal is to show that we have a closed monoidal structure on $\mdua$. We will focus on the construction of $\HOM_r(\ast,-)$, the same method applies for $\HOM_l(\ast,-)$.\ \\
 Recall that we have - by definition - the adjunctions $( -\ot a \dashv \HOM_r(a,-))$ and $( -\ot m \dashv \HOM_r(m,-))$. If we unravel each adjunction then:
 \begin{itemize}
 	\item any map $(f: b \otimes a \to c ) \in \Mor(\ag)$ has an adjoint-transpose map $\ol{f}: b \to \HOM_r(a, c)$, such that $f$ is the composition:
 	$ b \ot a \xrightarrow{\ol{f} \otimes \Id_a} \HOM_r(a,c) \ot a \xrightarrow{ev} c;$
 	\item any map $(g: n \otimes m \to m' ) \in \Mor(\M)$ has an adjoint-transpose map $\ol{g}: n \to \HOM_r(m, m')$, such that $g$ is the composition:
 	$ n \ot a \xrightarrow{\ol{g} \otimes \Id_m} \HOM_r(m,m') \ot m \xrightarrow{ev} m'.$
 \end{itemize}
 Let $\Fc$ and $\Gc$ be objects in $\mua$. Consider the composition in $\M$:
 \begin{align}
 \delta: \Ub(\HOM_r(\F^1,\G^1)) \ot \Ub(\F^1) \xrightarrow{\psi} \Ub( \HOM_r(\F^1, \G^1)  \ot \F^1) \xrightarrow{\Ub(ev)}\Ub(\G^1).
 \end{align}
 This map has a unique adjoint-transpose map $\ol{\delta}: \Ub(\HOM_r(\F^1,\G^1))  \to \HOM_r(\Ub(\F^1), \Ub(\G^1))$ such that:
 \begin{align}\label{eq-const-delta}
 \delta = ev_{\Ub(\F^1)} \circ (\ol{\delta} \ot \Id_{\Ub(\F^1)}).
 \end{align}
 Moreover, the map $\pif: \F^0 \to \Ub(\F^1)$  induces a map $$\HOM_r(\pif, \Ub(\G^1)): \HOM_r(\Ub(\F^1), \Ub(\G^1)) \to \HOM_r(\F^0, \Ub(\G^1)).$$
 The composition of $\ol{\delta}$ followed by the last map gives new map with codomain $\HOM_r(\F^0, \Ub(\G^1)$:
 \begin{align}
 \HOM_r(\pif, \Ub(\G^1)) \circ \ol{\delta}:  \Ub(\HOM_r(\F^1,\G^1)) \to \HOM_r(\F^0, \Ub(\G^1).
 \end{align}
 We also have a map $\HOM_r(\F^0, \pig): \HOM_r(\F^0, \G^0) \to \HOM_r(\F^0, \Ub(\F^1))$. The last two maps have a common codomain $\HOM_r(\F^0, \Ub(\G^1))$ and we can form the pullback square in $\M$:
 \[
 \xy
 (-30,18)*+{\HOM_r(\F^0, \G^0) \times_{\HOM_r(\F^0, \Ub(\G^1))}\Ub(\HOM_r(\F^1,\G^1))}="W";
 (-30,0)*+{\HOM_r(\F^0, \G^0) }="X";
 (50,0)*+{\HOM_r(\F^0, \Ub(\G^1))}="Y";
 (50,18)*+{\Ub(\HOM_r(\F^1,\G^1))}="E";
 {\ar@{->}^-{\HOM_r(\F^0, \pig)}"X";"Y"};
 {\ar@{.>}_-{p^0}"W";"X"};
 {\ar@{.>}^-{p^1}"W";"E"};
 {\ar@{->}^-{\HOM_r(\pif, \Ub(\G^1)) \circ \ol{\delta}}"E";"Y"};
 \endxy
 \]
 
 \begin{df}
 	Define  $\HOM_r(\Fc,\Gc) =[\HOM_r(\Fc,\Gc)^0, \HOM_r(\Fc,\Gc)^1, \pi_{\HOM_r}]$  by:
 	\begin{itemize}
 		\item $\HOM_r(\Fc, \Gc)^0:= \HOM_r(\F^0, \G^0) \times_{\HOM_r(\F^0, \Ub(\G^1))}\Ub(\HOM_r(\F^1,\G^1)) \in \M$
 		\item $\HOM_r(\Fc, \Gc)^1:= \HOM_r(\F^1,\G^1) \in \ag$
 		\item $\pi_{\HOM_r}= (\HOM_r(\F^0, \G^0) \times_{\HOM_r(\F^0, \Ub(\G^1))}\Ub(\HOM_r(\F^1,\G^1)) \xrightarrow{p^1} \Ub(\HOM_r(\F^1,\G^1) ))$.
 	\end{itemize}
 \end{df}
 
 \begin{lem}\label{adj-intern-hom-mua}
 	Let $\Ec, \Fc$ and $\Gc$ be objects in $\mua$. Then we have a functorial isomorphism of hom-sets:
 	$ \Hom(\Ec \ot \Fc, \Gc) \xrightarrow{\cong} \Hom(\Ea , \HOM_r(\Fc,\Gc)).$
 \end{lem}
 \begin{proof}
 	We will construct the isomorphism between the hom-sets, leaving the functoriality to the reader.
 	A map $\osg: \Ec \ot \Fc \to \Gc$ satisfies the equation: $\pig \circ \sgo = \Ub(\sgi) \circ \pi_{\Ea \ot \F}$. 
 	The map $\sg^i: \Ea^i \ot \F^i \to \G^i$ has a unique adjoint-transpose map $\ol{\sg^i} : \Ea^i \to \HOM_r(\F^i,\G^i)$, such that for $i \in \{0;1\}$:
 	\begin{align}\label{pf-int-hom-1}
 	\sg^i= ev_{\F^i} \circ (\ol{\sg^i} \ot \Id_{\F^i}).
 	\end{align}
 	We claim that the two maps:
 	\begin{center}
 		 $\Ub(\ol{\sgi}) \circ \pi_{\Ea} \in \Hom(\Ea^0, \Ub(\HOM_r(\F^1,\G^1)))$ and $\ol{\sgo} \in \Hom(\Ea^0, \HOM_r(\F^0, \G^0))$
 	\end{center}
 	 complete the pullback data hereafter - defining $\HOM_r(\Fc,\Gc)$ -  into a commutative square:
 	 \begin{align}\label{pb-data-homr}
 	 \HOM_r(\F^0,\G^0) \xrightarrow{\HOM_r(\F^0, \pig)} \HOM_r(\F^0,\Ub(\G^1)) \xleftarrow{\HOM_r(\pif, \Ub(\G^1)) \circ \ol{\delta}} \Ub(\HOM_r(\F^1,\G^1))).
 	 \end{align}
 	To prove this, first observe that the map $\pig \circ \sgo \in \Hom(\Ea^0 \ot \F^0, \Ub(\G^1))$ has a unique adjoint-transpose $\ol{\pig \circ \sgo} : \Ea^0 \to  \HOM_r(\F^0, \Ub(\G^1))$ satisfying the equation:
 	\begin{align}\label{pf-int-hom-2}
 	\pig \circ \sgo = ev_{\F^0} \circ (\ol{\pig \circ \sgo} \ot \Id_{\F^0}).
 	\end{align}
 	Recall that the evaluation ($ev: \Hom(\F^0, \ast) \ot \F^0 \to \ast) $ is simply the counit of the adjunction $-\ot \F^0 \dashv \Hom(\F^0,\ast)$. Then the naturality of the counit with respect to $\pig$ gives the equality:
 	\begin{align}
 	\pig \circ ev_{\F^0} = ev_{\F^0} \circ (\HOM_r(\F^0,\pig ) \ot \Id_{\F^0}).
 	\end{align}
 	If we precompose the last equality with $\ol{\sg^0} \ot \Id_{\F^0}$ and use (\ref{pf-int-hom-1}) for $i=0$ we get:
 	\begin{align}\label{pf-int-hom-3}
 	\underbrace{\pig \circ ev_{\F^0} \circ (\ol{\sg^0} \ot \Id_{\F^0})}_{= \pig \circ \sgo}
 	&= ev_{\F^0} \circ (\HOM_r(\F^0,\pig ) \ot \Id_{\F^0}) \circ (\ol{\sg^0} \ot \Id_{\F^0}) \\
 	&= ev_{\F^0}  \circ ((\HOM_r(\F^0,\pig )\circ \ol{\sg^0}) \ot \Id_{\F^0})
 	\end{align}
 	With (\ref{pf-int-hom-2})  and (\ref{pf-int-hom-3}), we have by uniqueness of the adjoint-transpose map the equality:
 	\begin{align}\label{pf-int-hom-4}
 	\HOM_r(\F^0,\pig )\circ \ol{\sg^0} = \ol{\pig \circ \sgo}.
 	\end{align}
 	The remainder of the proof is to establish that $\ol{\pig \circ \sgo}= \HOM_r(\pif, \Ub(\G^1)) \circ( \ol{\delta} \circ \Ub(\ol{\sgi}) \circ \pi_{\Ea})$. By adjointness, this boils down to showing that: 
 	\begin{align}
 	\pig \circ \sgo =ev_{\F^0} \circ(\ol{\pig \circ \sgo} \ot \Id_{\F^0}) = ev_{\F^0} \circ [(\HOM_r(\pif, \Ub(\G^1)) \circ( \ol{\delta} \circ \Ub(\ol{\sgi}) \circ \pi_{\Ea})) \ot \Id_{\F^0}].
 	\end{align}
 	With the various maps considered previously, we get a commutative diagram below  by puzzling smaller commutative squares. We've included a number inside each small square for explanation.
 	\begin{align}\label{pf-int-hom-5}
 	\xy
 	(-70,38)*+{\Ea^0\ot \F^0}="O";
 	(-70,-38)*+{\HOM_r(\F^0, \Ub(\G^1))\ot \F^0}="P";
 	(65,38)*+{\G^0}="A";
 	(65,-38)*+{\Ub(\G^1)}="B";
 	(-70,-18)*+{\HOM_r(\Ub(\F^1), \Ub(\G^1))\ot \F^0}="W";
 	(-10,18)*+{\Ub(\Ea^1) \ot \Ub(\F^1)}="X";
 	(-10,0)*+{\Ub(\HOM_r(\F^1,\G^1)) \ot \Ub(\F^1)}="M";
 	(40,0)*+{\Ub(\HOM_r(\F^1,\G^1) \ot \F^1)}="Y";
 	(40,18)*+{\Ub(\Ea^1 \ot \F^1)}="E";
 	(-10,38)*+{\Ea^0 \ot \F^0}="Z";
 	{\ar@{->}_-{\pi_{\Ea}\ot \pif}"Z";"X"};
 	{\ar@{->}_-{\Ub(\ol{\sgi})\ot \Ub(\Id)}"X";"M"};
 	{\ar@{->}^-{\Ub(\ol{\sgi} \ot \Id)}"E";"Y"};
 	(-10,-18)*+{\HOM_r(\Ub(\F^1), \Ub(\G^1))\ot \Ub(\F^1)}="S";
 	(40,-18)*+{\Ub(\G^1)}="T";
 	{\ar@{->}^-{\psi}"M";"Y"};
 	{\ar@{->}^-{\psi}"X";"E"};
 	{\ar@{->}^-{\Ub(ev_{\F^1})}"Y";"T"};
 	{\ar@{->}^-{ev_{\Ub(\F^1)}}"S";"T"};
 	{\ar@{->}^-{\HOM_r(\pif, \Ub(\G^1)) \ot \Id}"W";"P"};
 	{\ar@{->}^-{(\ol{\delta} \circ \Ub(\ol{\sgi}) \circ \pi_{\Ea}) \ot \Id_{\F^0}}"O";"W"};
 	{\ar@{->}^-{\Id}"T";"B"};
 	{\ar@{->}^-{ev_{\F^0}}"P";"B"};
 	{\ar@{=}_-{\Id}"Z";"O"};
 	{\ar@{->}^-{\sgo}"Z";"A"};
 	(65,18)*+{\Ub(\G^1)}="U";
 	{\ar@{->}^-{\Id}"U";"B"};
 	{\ar@{->}^-{\pig}"A";"U"};
 	{\ar@{->}^-{\Id \ot \pif}"W";"S"};
 	{\ar@{->}^-{\Ub(\sgi)}"E";"U"};
 	{\ar@{->}_-{\ol{\delta} \ot \Id}"M";"S"};
 	(-40,18)*+{\large \textcircled{\small 1}};
 	(30,28)*+{\large \textcircled{\small 2}};
 	(15,10)*+{\large \textcircled{\small 3}};
 	(15,-10)*+{\large \textcircled{\small 4}};
 	(0,-28)*+{\large \textcircled{\small 5}};
 	(60,-18)*+{\large \textcircled{\small 6}};
 	\endxy
 	\end{align}
 	\begin{itemize}
 		\item $\large \textcircled{\small 1}$ commutes by inspection.
 		\item $\large \textcircled{\small 2}$ represents the given map $\osg: \Ec \ot \Fc \to \Gc$.
 		\item $\large \textcircled{\small 3}$ commutes by functoriality of the laxity map $\psi$.
 		\item $\large \textcircled{\small 4}$ is commutative by definition of $\ol{\delta}$ (\ref{eq-const-delta}).
 		\item $\large \textcircled{\small 5}$ commutes by definition of $\HOM_r(\pif, \Ub(\G^1))$.
 		\item $\large \textcircled{\small 6}$ is given by applying $\Ub$ to the equality (\ref{pf-int-hom-1}).
 	\end{itemize}
 	The perimeter of the above diagram along with (\ref{pf-int-hom-4}) give:
 	\begin{align*}
 	& ev_{\F^0} \circ [(\HOM_r(\pif, \Ub(\G^1)) \circ( \ol{\delta} \circ \Ub(\ol{\sgi}) \circ \pi_{\Ea})) \ot \Id_{\F^0}]\\
 	&=\pig \circ \sgo \\
 	&= ev_{\F^0} \circ(\underbrace{\ol{\pig \circ \sgo}}_{=\HOM_r(\F^0,\pig )\circ \ol{\sg^0}} \ot \Id_{\F^0}) \\
 	& = ev_{\F^0} \circ [(\HOM_r(\F^0,\pig )\circ \ol{\sg^0}) \ot \Id_{\F^0}].
 	\end{align*}
 	The uniqueness of the adjoint-transpose map yields:
 	\begin{align}\label{pf-int-hom-6}
 	\HOM_r(\pif, \Ub(\G^1)) \circ( \ol{\delta} \circ \Ub(\ol{\sgi}) \circ \pi_{\Ea}) = \HOM_r(\F^0,\pig )\circ \ol{\sg^0}.
 	\end{align}
 	The last equality implies that we've completed the aforementioned pullback data (\ref{pb-data-homr}) into a commutative square. Using the universal property of the pullback  $ \HOM_r(\F^0, \G^0) \times_{\HOM_r(\F^0, \Ub(\G^1))}\Ub(\HOM_r(\F^1,\G^1)) = \HOM_r(\Fc,\Gc)^0$, there is a unique map
 	$$\theta : \Ea^0 \to \HOM_r(\F^0, \G^0) \times_{\HOM_r(\F^0, \Ub(\G^1))}\Ub(\HOM_r(\F^1,\G^1)),$$
 	such that:
 	\begin{align}
 	& \ol{\sgo} = p^0 \circ \theta \\
 	& \Ub(\ol{\sgi}) \circ \pi_{\Ea} = \overbrace{p^1}^{\pi_{\HOM_r}} \circ \theta.
 	\end{align}
 	The last equality is equivalent to saying that the couple $[\theta, \ol{\sgi}]$ defines a unique map $\Ec \to \HOM_r(\Fc, \Gc)$. This process clearly defines a one-one function: $$\phi : \Hom(\Ec \ot \Fc, \Gc) \to \Hom(\Ec, \HOM_r(\Fc,\Gc)).$$ 
 	If we reverse the argument we see that this function has an inverse that takes $[\theta, \ol{\sgi}] \mapsto [\varphi^{-1}(p^0\circ \theta), \varphi^{-1}(\ol{\sgi})]$, where $\varphi= \overline{(-)}$ is the isomorphism $\Hom(- \ot \F^i, \ast) \xrightarrow{\cong} \Hom(-, \HOM_r(\F^i, \ast))$.
\end{proof}
\begin{df}
Define $\HOM_l(\Fc,\Gc) =[\HOM_l(\Fc,\Gc)^0, \HOM_l(\Fc,\Gc)^1, \pi_{\HOM_l}] \in \mua$ by: 
\begin{itemize}
	\item $\HOM_l(\Fc, \Gc)^0:= \HOM_l(\F^0, \G^0) \times_{\HOM_l(\F^0, \Ub(\G^1))}\Ub(\HOM_l(\F^1,\G^1)) \in \M$
	\item $\HOM_l(\Fc, \Gc)^1:= \HOM_l(\F^1,\G^1) \in \ag$
	\item $\pi_{\HOM_l}= (\HOM_l(\F^0, \G^0) \times_{\HOM_l(\F^0, \Ub(\G^1))}\Ub(\HOM_l(\F^1,\G^1)) \xrightarrow{p^1} \Ub(\HOM_l(\F^1,\G^1) ))$.
\end{itemize}
\end{df}

\begin{lem}\label{lem-hom-l}
We have a functorial isomorphism of hom-sets:
 $$\Hom(\Fc \ot \Ec, \Gc) \cong \Hom(\Ec, \HOM_l(\Fc,\Gc)).$$
\end{lem}
\begin{proof}
The proof is the same as that of Lemma \ref{adj-intern-hom-mua}.
\end{proof}
With the previous material we can now prove Proposition \ref{fact-lax}.
\begin{proof}[Proof of Proposition \ref{fact-lax}]
	\hfill
	\ \\
Assertion $(1)$ and $(2)$ are the content of Lemma \ref{lem-lax}. Assertion $(3)$ is given by Lemma \ref{adj-intern-hom-mua} and Lemma \ref{lem-hom-l}.
\end{proof}
\subsection{Factorization of a monoidal Quillen adjunction}
We assume that $\ag$ and $\M$ are closed monoidal model categories as in \cite{Hov-model} except that we chose the simplified version for the \emph{Unit axiom} of Schwede-Shipley \cite{Sch-Sh-equiv}:\\
\textbf{Unit axiom:} Let $q:I^c \xrightarrow{\sim} I$ be a cofibrant replacement of the unit object. Then for every cofibrant object $A$ the map $q \otimes \Id_{A}: I^c \ot  A  \to I \ot A \cong A$ is a weak equivalence.
\begin{lem}\label{lem-push-prod}
The pushout product $\sg \square \theta$ of two maps in $\mua$ is given by the couple of pushout products $[\sgo \square \theta^0, \sgi \square \theta^1]$.
\end{lem}
\begin{proof}
This is obvious by inspection, considering that pushouts are computed level-wise.
\end{proof}
\begin{thm}\label{muaij-mon-mod}
	Assume that $\Ub : \ag \to \M$ is a right Quillen functor which is  a lax monoidal functor between monoidal model categories. \\
	Then the model category $\mua_{inj}$ with the point-wise product, is a monoidal model category.
\end{thm}
\begin{proof}
We simply need to verify the \emph{Pushout product axiom} and the \emph{Unit axiom}. In the model category $\mua_{inj}$ the weak equivalences  and the (trivial) cofibrations are precisely the level-wise weak equivalences and the level-wise (trivial) cofibrations respectively. With Lemma \ref{lem-push-prod}, the pushout product axiom in $\ag$ and $\M$ give the result. For the Unit axiom, we note that a cofibrant replacements $[I_{\M}^c, I_{\ag}^c,\pi^c] \xrightarrow{\sim} [I_{\M}, I_{\ag}, \psi_I]$ induces two cofibrant replacement  $I_{\M}^c\xrightarrow{\sim} I_{\M}$, $I_{\ag}^c\xrightarrow{\sim} I_{\ag}$. The  Unit axiom in $\ag$ and $\M$  implies that the map $[I_{\M}^c, I_{\ag}^c,\pi^c] \ot \Fc \to  [I_{\M}, I_{\ag}, \psi_I] \ot \Fc$ is a level-wise weak equivalence for every (level-wise) cofibrant object $\Fc \in \muaij$.
\end{proof}
\begin{thm}\label{muaij-left-mon-mod}
Under the same hypothesis, the model category $\mua_{inj}^{\bf{L}}$ with the point-wise product, is a monoidal model category.
\end{thm}
\begin{proof}
A trivial cofibration $\osg$ in $\mua_{inj}^{\bf{L}}$ is a level-wise cofibration such that $\sgi$ is a trivial cofibration. The pushout product axiom of the monoidal model category $\ag$ implies that if $\sg=\osg$ or $\theta=[\theta^0, \theta^1]$ is a (trivial) cofibration in $\mua_{inj}^{\bf{L}}$, then so is the map $[\sgo \square \theta^0, \sgi \square \theta^1]$. A cofibrant replacement $[I_{\M}^c, I_{\ag}^c,\pi^c] \xrightarrow{\sim} [I_{\M}, I_{\ag}, \psi_I]$ yields a cofibrant replacement $I_{\ag}^c\xrightarrow{\sim} I_{\ag}$. Therefore, by the Unit axiom of the monoidal model category $\ag$, the map $[I_{\M}^c, I_{\ag}^c,\pi^c] \ot \Fc \to  [I_{\M}, I_{\ag}, \psi_I] \ot \Fc$ is a weak equivalence in $\mua_{inj}^{\bf{L}}$, for every cofibrant $\Fc \in \mua_{inj}^{\bf{L}}$.
\end{proof}
We remind the reader that in an adjunction $\Fb \dashv \Ub$, if $\Ub$ is a lax monoidal then $\Fb$ is necessarily a colax monoidal functor by doctrinal adjunction (see \cite{Kelly-doctrine}).
\begin{df}
A Quillen adjunction $\Fb \dashv \Ub$ is a \emph{a weak monoidal adjunction} if $\Ub$ is lax monoidal and $\Fb$ is colax monoidal such that the following two conditions hold:
\begin{enumerate}
	\item for all cofibrant objects $\Pa$ and $\Qa$ in $\M$, the colaxity map $\Fb(\Pa \ot \Qa) \xrightarrow{\tld{\psi}} \Fb(\Pa) \ot \Fb(\Qa)$ is a weak equivalence and
	\item for some (hence any) cofibrant replacement $q:I_{\M}^c \xrightarrow{\sim} I_{\M}$ of the unit object, the composite map $\Fb(I_{\M}^c) \xrightarrow{\Fb(q)} \Fb(I_{\M}^c) \xrightarrow{\tld{\psi}} I_{\ag}$ is a weak equivalence.
\end{enumerate}
\end{df}
This definition appears in \cite{Sch-Sh-equiv}. If in addition $\Fb$  a strong monoidal colax functor, we will say that $(\Fb \dashv \Ub)$ is a strong monoidal Quillen adjunction. The later notion is the one appearing in Hovey's book \cite{Hov-model}. For simplicity, we will say that $\Ub$ is a lax monoidal right Quillen functor if $\Ub$ is part of weak (or strong) monoidal Quillen adjunction $(\Fb \dashv \Ub)$.
\begin{thm}\label{monoidal-mod-adj}
Assume that $\Fb : \M \leftrightarrows \ag: \Ub$ is a weak (resp. strong) monoidal Quillen adjunction. Then the following hold.
\begin{enumerate}
	\item The adjunctions $\Fb^+ : \M  \leftrightarrows \mua_{inj}: \Pio$ and    $\Fb^+ : \M  \leftrightarrows \mua_{inj}^{\bf{L}}: \Pio$ are both weak (resp. strong)  monoidal Quillen adjunctions.
	\item The adjunction $\Piun : \mua_{inj}  \leftrightarrows \ag :\iota$ is a strong monoidal Quillen  adjunction.
	\item The adjunction $\Piun : \mua_{inj}^{\bf{L}}  \leftrightarrows \ag :\iota$  is a strong monoidal Quillen equivalence.
\end{enumerate}
\end{thm}
\begin{proof}
 Let $m, m'$ be cofibrant objects in $\M$. The colaxity map $\Fb^+(m \ot m') \xrightarrow{} \Fb^+(m) \ot \Fb^+(m')$
  is given by the couple $[\Id_{m\ot m'}, \tld{\psi}]$ as illustrated in this diagram:
	\[
\xy
(0,30)*+{m\otimes m'}="W";
(0,0)*+{\Ub(\Fb(m\ot m'))}="X";
(40,15)*+{\Ub(\Fb(m))\otimes \Ub(\Fb(m'))}="Z";
(40,0)*+{\Ub(\Fb(m) \ot \Fb(m'))}="Y";
(40,30)*+{m \ot m'}="E";
{\ar@{->}^-{\Ub(\tld{\psi})}"X";"Y"};
{\ar@{->}_-{\eta_{m\ot m'}}"W";"X"};
{\ar@{->}^-{\Id_{m\ot m'}}"W";"E"};
{\ar@{->}^-{ \eta_m \ot \eta_{m'}}"E";"Z"};
{\ar@{->}^-{\psi}"Z";"Y"};
\endxy
\]
By assumption, $\Fb(m \ot m') \xrightarrow{\tld{\psi}} \Fb(m) \ot \Fb(m')$ is weak equivalence, therefore  $[\Id_{m\ot m'}, \tld{\psi}]$ is  a weak equivalence in  $\mua_{inj}$, hence in $\mua_{inj}^{\bf{L}}$. Moreover, if $q:I_{\M}^c \xrightarrow{\sim} I_{\M}$ is a cofibrant replacement, the composite $\Fb^+(I_{\M}^c) \xrightarrow{\Fb^+(q)} \Fb^+(I_{\M}) \xrightarrow{\tld{\psi}} [I_{\M}, I_{\ag}, \psi_I] $ is given by the couple $[q, \tld{\psi} \circ \Fb(q)]$:
\[
\xy
(0,15)*+{I_{\M}^c}="W";
(0,0)*+{\Ub(\Fb(I_{\M}^c))}="X";
(60,0)*+{\Ub(I_{\ag})}="Y";
(60,15)*+{I_{\M}}="E";
{\ar@{->}^-{\eta_{I_{\M}^c}}"W";"X"};
{\ar@{->}^-{\psi}"E";"Y"};
(30,15)*+{I_{\M}}="U";
(30,0)*+{\Ub(\Fb(I_{\M}))}="V";
{\ar@{->}^-{\Ub(\Fb(q))}"X";"V"};
{\ar@{->}^-{\Ub(\tld{\psi})}"V";"Y"};
{\ar@{->}^-{\eta_{I_{\M}}}"U";"V"};
{\ar@{->}^-{\Id}"U";"E"};
{\ar@{->}^-{q}"W";"U"};
\endxy
\] 
By assumption the composition $\tld{\psi} \circ \Fb(q)$ is a weak equivalence in $\ag$ which implies that $[q, \tld{\psi} \circ \Fb(q)]$ is a weak equivalence in $\mua_{inj}$, hence in $\mua_{inj}^{\bf{L}}$. This gives Assertion $(1)$.

 The other assertions are easily checked considering that $\Piun$ is a strong monoidal functor that creates the weak equivalences in  $\mua_{inj}^{\bf{L}}$ and that in $\mua_{inj}$ everything is level-wise. The Quillen equivalence $\Piun : \mua_{inj}^{\bf{L}}  \leftrightarrows \ag :\iota$ is given by Theorem \ref{fact-inj-thm}.
\end{proof}
\begin{cor}\label{cor-fact-mon}
\
\begin{enumerate}
	\item Any lax monoidal right Quillen functor  can be factored as a lax monoidal functor which is part of a strong monoidal adjunction, followed by strong monoidal right Quillen functor.
	\item Any lax monoidal right Quillen functor can be factored as a lax monoidal right  Quillen equivalence, followed by a strong monoidal right Quillen functor.
\end{enumerate}
\begin{rmk}
The category $\mua_{inj}^{\bf{R}}$ appearing in the other factorization is a right (Bousfield) localization and we have a very little control on the new cofibrations. Without any further hypothesis it is difficult to check the pushout product axiom if $\Fb$ is colax monoidal functor that is not a strong monoidal functor. However, if $\M$ and $\ag$ are cofibrantly generated and if $\Fb$ is a strong monoidal functor, then we can endow $\mua_{inj}^{\bf{R}}$ with the structure of a monoidal model structure such that $\Pio$ becomes a monoidal Quillen equivalence. 
\end{rmk}
\end{cor}
\appendix
\section{Other factorizations}\label{sec-other-fact}
In this section we show that there are alternative factorizations for a (right) Quillen functor.\\
\noindent Say that a morphism $\sg: \Fc \xrightarrow{\osg} \Gc$ is:
\begin{itemize}
	\item a $\Pio$-fibration (resp. $\Pio$-equivalence) if $\sgo: \F^0 \to \G^0$ is a fibration in $\M$ (resp. weak equivalence in $\M$).
	\item a strong (trivial) cofibration if:
	\begin{itemize}
		\item $\sigma^0: \F^0 \to \G^0$ is a  (trivial) cofibration in $\M$ and if
		\item the equivalent diagram representing $\sg$ is a pushout square:
		\[
		\xy
		(0,18)*+{\Fb(\F^0)}="W";
		(0,0)*+{\F^1}="X";
		(30,0)*+{\G^1}="Y";
		(30,18)*+{\Fb(\G^0)}="E";
		{\ar@{->}^-{\sigma^{1}}"X";"Y"};
		{\ar@{->}^-{\varphi(\pi_{\F})}"W";"X"};
		{\ar@{->}^-{\Fb(\sigma^0)}"W";"E"};
		{\ar@{->}^-{\varphi(\pi_{\G})}"E";"Y"};
		\endxy
		\]
	\end{itemize}
\item a trivial $\Pio$-fibration is it is simultaneously a $\Pio$-fibration and a $\Pio$-equivalence, i.e, if $\sgo$ is a trivial fibration in $\M$.
\end{itemize}
\noindent We note that the above definition doesn't use the model structure on $\ag$. When $\ag$ has a model structure such that $\Ub$ is right Quillen, then a strong cofibration is in particular a projective cofibration since $\sgo$ is a cofibration and $\F^1 \cup^{\Fb(\F^0)} \Fb(m^0) \to  \G^1$ is an isomorphism.
\begin{lem}\label{lem-lift-zero}
\ 
\begin{enumerate}
	\item Any map $\sg: \Fc \xrightarrow{\osg} \Gc$ can be factored as a strong cofibration followed by a trivial $\Pio$-fibration.
	\item Any map $\sg: \Fc \xrightarrow{\osg} \Gc$ can be factored as a strong trivial cofibration followed by a $\Pio$-fibration.
\end{enumerate}
\end{lem}

\begin{proof}
We will follow the same idea as in the proofs of the injective and projective model structures in \cite{Bacard_QS}. However,  the argument here only uses the model structure of $\M$, that is, we do not require $\ag$ to be a model category for the lemma to be true. Let  $(\lm, \mr)$ be one of the factorization systems $(\cof(\M), \fib(\M)\cap \we(\M))$, $(\cof(\M) \cap \we(\M), \fib(\M))$. The axiom of the model category $\M$ gives a factorization $\sgo= r(\sgo) \circ l(\sgo)$:
$$\F^0 \xrightarrow{\sgo} \G^0 = \F^0\xhookrightarrow{l(\sgo)}  m^0 \xtwoheadrightarrow{r(\sgo)} \G^0,$$ 
with $ r(\sgo) \in \mr $ and $ l(\sgo) \in \lm$.
The image under $\Fb$ of this factorization, gives a factorization  $\Fb(\sgo)= \Fb(r(\sgo)) \circ \Fb(l(\sgo))$. Consider the pushout data $ \F^0 \xleftarrow{\varphi(\pif)} \Fb(\F^0) \xrightarrow{\Fb(l(\sgo))} \Fb(m^0)$ and let  $i_1: \Fb (m^0)   \to E^1$, $i_2 : \F^1 \to  \F^1 \cup^{\Fb(\F^0)} \Fb(\G^0)$ be the canonical maps. 
The universal property of the pushout square gives a unique map 
$ \zeta : \F^1 \cup^{\Fb(\F^0)} \Fb(m^0) \to  \G^1,$
such that everything below commutes. 
\[
\xy
(0,25)*+{\Fb(\F^0)}="W";
(0,0)*+{\F^1}="X";
(60,0)*+{\G^1}="Y";
(60,25)*+{\Fb(\G^0)}="E";
{\ar@{->}^-{\varphi(\pif)}"W";"X"};
{\ar@{->}^-{\sgi}"X";"Y"};
{\ar@{->}^-{\varphi(\pig)}"E";"Y"};
(30,10)*+{\F^1 \cup^{\Fb(\F^0)} \Fb(m^0)}="U";
(30,25)*+{\Fb(m^0)}="V";
{\ar@{->}^-{\Fb(l(\sgo))}"W";"V"};
{\ar@{->}^-{\Fb(r(\sgo))}"V";"E"};
{\ar@{->}^-{i_2}"X";"U"};
{\ar@{->}^-{i_1}"V";"U"};
{\ar@{->}^-{\zeta}"U";"Y"};
\endxy
\]
Let $\Ec=[\Ea^0,\Ea^1,\pi_{\Ea}]$ be the object of $\mua$ defined by
$$ \Ea^0= m^0, \quad \Ea^1=\F^1 \cup^{\Fb(\F^0)} \Fb(m^0), \quad \pi_{\Ea}= \varphi^{-1} (i_1) \in  \Hom_{\M}(m^0,\Ub(\F^1 \cup^{\Fb(\F^0)} \Fb(m^0))).$$
The couple $[l(\sgo), i_2]$ defines a morphism $l(\sg): \Fc \to \Ec$ and the couple $[r(\sgo), \zeta]$ defines a morphism $r(\sg): \Ec \to \Gc$. Clearly, the factorization $\sg= r(\sg) \circ l(\sg)$ gives both assertions.
\end{proof}
\begin{lem}\label{lem-fact-zero}
Consider a lifting problem in $\mua$ as follows.
\[
\xy
(0,18)*+{\Fc}="W";
(0,0)*+{\Gc}="X";
(30,0)*+{\Qc}="Y";
(30,18)*+{\Pc}="E";
{\ar@{->}^-{[\gamma^0, \gamma^{1}]}"X";"Y"};
{\ar@{->}_-{[\sgo,\sgi]}"W";"X"};
{\ar@{->}^-{[\theta^0,\theta^{1}]}"W";"E"};
{\ar@{->}^-{[\beta^0,\beta^{1}]}"E";"Y"};
\endxy
\]
\begin{enumerate}
	\item If $\sg=[\sgo,\sgi]$ is a strong cofibration and $\beta = [\beta^0, \beta^1]$ a trivial $\Pio$-fibration, then a solution $s : \Gc \xrightarrow{[s^0,s^{1}]}\Pc$ exists.
	\item  If $\sg=[\sgo,\sgi]$ is a strong trivial cofibration and $\beta = [\beta^0, \beta^1]$ a $\Pio$-fibration, then a solution $s : \Gc \xrightarrow{[s^0,s^{1}]}\Pc$ exists.
\end{enumerate}
\end{lem}

\begin{proof}
Let $(\lm, \mr)$ be one of the aforementioned factorization systems of the model category $\M$. 
 Applying  the functor $\Pio$ to the commutative square of the lemma gives a lifting problem defined by $\sgo$ and $\beta^0$. The latter problem admits a solution  solution $s^0: \G^0 \to \Pa^0$ if  $(\sgo, \beta^0) \in \lm \times \mr$. We note that the lifting problem of the lemma is represented by a commutative cube in $\ag$ which is adjoint to another commutative cube in $\M$. By inspection, the co-pushout data $\F^1 \xrightarrow{\theta^1} \Pa^1 \xleftarrow{\varphi(\pi_{\Pa}) \circ \Fb(s^0)} \Fb(\G^0)$ completes the pushout data  $\F^1 \xleftarrow{\varphi(\pif)} \Fb(\F^0) \xrightarrow{\Fb(\sgo)} \Fb(\G^0)$ into a commutative square. Moreover, by assumption, the commutative square in $\ag$ that represents $\sg$ is a pushout square, therefore there is a unique map $s^1:\G^1 \to \Pa^1$ satisfying the obvious equations. The uniqueness of a map out of a pushout object implies that the couple $[s^0,s^1]$ defines a morphism $s : \Gc \to \Pc$ which is a solution to our lifting problem. The commutative cube below helps in visualizing the situation. 
 \[
\xy
(-60,10)*+{\Fb(\F^0)}="A";
(10,20)*+{\Fb(\Pa^0)}="B";
(-20,-10)*+{\Fb(\G^0)}="C";
(50,0)*+{\Fb(\Qa^0)}="D";
{\ar@{->}^{\Fb(\theta^0)}"A";"B"};
{\ar@{->}_{\Fb(\sg^0)}"A";"C"};
{\ar@{->}_{\Fb(s^0)}"C";"B"};
{\ar@{->}^{\Fb(\beta^0)}"B";"D"};
{\ar@{->}_{}"C";"D"};
(-60,-30)*+{\F^1}="X";
(10,-20)*+{\Pa^1}="Y";
(-20,-50)*+{\G^1}="Z";
(50,-40)*+{\Qa^1}="W";
{\ar@{-->}^{\quad \quad ~~~~~~~~~~~~~~~ \theta^1}"X";"Y"};
{\ar@{->}_{\varphi(\pif)}"A";"X"};
{\ar@{->}^{\beta^1}"Y";"W"};
{\ar@{->}^{\varphi(\pig)}"C";"Z"};
{\ar@{->}_-{}"B";"Y"};
{\ar@{->}^{\varphi(\piq)}"D";"W"};
{\ar@{->}_{\sg^1}"X";"Z"};
{\ar@{->}_{\gamma^1}"Z";"W"};
{\ar@{.>}_-{s^1}"Z";"Y"};
\endxy
\]
\end{proof}
\begin{thm}\label{new-fact-pio}
	Call a morphism $\sigma: \Fc \xrightarrow{[\sgo,\sgi]} \Gc$ in $\mua$:
	\begin{itemize}
		\item  a weak equivalence if it is a $\Pio$-equivalence
		\item  a cofibration if it is a strong cofibration.
		\item a fibration if it is a $\Pio$-fibration. 
	\end{itemize}
	Then these choices provide $\mua$ with the structure of a model category that will be denoted by $\mua^{0}$. 
\end{thm}
\begin{proof}
    Limits and colimits are computed level-wise.
	The three classes of maps are clearly closed under retracts and composition. Lemma \ref{lem-lift-zero} and Lemma \ref{lem-fact-zero} provide the other axioms of a model structure.
\end{proof}
\begin{cor}
For any right Quillen functor $\Ub: \ag \to \M$, the following hold.
\begin{enumerate}
	\item We have two Quillen equivalences: $$\Fb^+: \M \leftrightarrows \mua^{0}:\Pio \hspace{0.5in} \Pio : \mua^{0} \leftrightarrows \M: R^{0}$$
	\item We have a Quillen adjunction: $\Piun : \mua^{0} \leftrightarrows \ag : \iota$.
	\item The identity induces a left Quillen functor $\mua^{0} \to \mua_{proj}^{\bf{R}}$ which is a Quillen equivalence.
	\item The functor $\Ub$ is the composite: $ \quad \ag \xhookrightarrow{\iota} \mua^{0} \xrightarrow[\sim]{\Pio} \M.$
\end{enumerate}
\end{cor}
\begin{proof}
Assertion $(1)$ is proved the same way as in Theorem \ref{right-fact-inj-thm} and Theorem \ref{right-fact-proj-thm}. If $\Fc \in \mua^{0}$ is cofibrant and $m \in \M$ is fibrant, then a map $\Fc \xrightarrow{} R^0(m)$ is a $\Pio$-equivalence in $\mua^{0}$ if and only if the adjoint-transpose map $(\F^0 \xrightarrow{} m)= (\Pio(\Fc) \xrightarrow{} m)$ is a weak equivalence in $\M$, whence the assertion. Given a map $j: \Pa \to \Qa$ in $\ag$, then the map $\iota(j)$ is given by the couple $[\Ub(j), j]$. Since $\Ub$ is a right Quillen, the map $\iota(j)$ is clearly a $\Pio$-fibration (resp.  trivial $\Pio$-fibration) if $j$ is, which proves that $\iota$ is a right Quillen functor, thus Assertion $(2)$. As observed earlier, a strong (trivial) cofibration is a projective (trivial) cofibration. Moreover, we have the same weak equivalences in $\mua^{0}$ and $\mua_{inj}^{\bf{R}}$ which proves Assertion $(3)$. Assertion $(4)$ is obvious.
\end{proof}
\noindent We can dualize the previous results using the isomorphisms of Proposition \ref{prop-dual-mod}:
$$ \mua:= \mdua \cong  (\Ub^{\op} \downarrow \M^{\op})^{\op} \cong (\ag^{\op} \downarrow \Fb^{op})^{\op}.$$
\noindent Say that a morphism $\sg: \Fc \xrightarrow{\osg} \Gc$ is:
\begin{itemize}
	\item a $\Piun$-cofibration (resp. $\Piun$-equivalence) if $\sgi: \F^1 \to \G^1$ is a cofibration in $\ag$ (resp. weak equivalence in $\ag$).
	\item a strong (trivial) fibration if:
	\begin{itemize}
		\item $\sigma^1: \F^1 \to \G^1$ is a  (trivial) fibration in $\ag$ and if
		\item the equivalent diagram representing $\sg$ is a pullback square in $\M$:
		\[
		\xy
		(0,18)*+{\F^0}="W";
		(0,0)*+{\Ub(\F^1)}="X";
		(30,0)*+{\Ub(\G^1)}="Y";
		(30,18)*+{\G^0}="E";
		{\ar@{->}^-{\Ub(\sigma^{1})}"X";"Y"};
		{\ar@{->}^-{\pi_{\F}}"W";"X"};
		{\ar@{->}^-{\sigma^0}"W";"E"};
		{\ar@{->}^-{\pi_{\G}}"E";"Y"};
		\endxy
		\]
	\end{itemize}
	\item a trivial $\Piun$-cofibration is it is simultaneously a $\Piun$-cofibration and a $\Piun$-equivalence, i.e, if $\sgo$ is a trivial cofibration in $\M$.
\end{itemize}
\begin{thm}\label{new-fact-piun}
	Call a morphism $\sigma: \Fc \xrightarrow{[\sgo,\sgi]} \Gc$ in $\mua$:
	\begin{itemize}
		\item  a weak equivalence if it is a $\Piun$-equivalence
		\item  a cofibration if it is a $\Piun$-cofibration.
		\item a fibration if it is a  strong fibration. 
	\end{itemize}
	Then these choices provide $\mua$ with the structure of a model category that will be denoted by $\mua^{1}$. 
	
\end{thm}
\begin{proof}
	Apply Theorem \ref{new-fact-pio} to the right Quillen functor $\Fb^{\op}: \M^{\op} \to \ag^{\op}$ to get the model category $\ag^{\op}_{\Fb^{\op}}[\M^{\op}]^{0}$. Then dualize it with the isomorphism $ \mua \cong (\ag^{\op}_{\Fb^{\op}}[\M^{\op}])^{\op}$.
\end{proof}
We also have:
\begin{cor}
	For any right Quillen functor $\Ub: \ag \to \M$, the following hold.
	\begin{enumerate}
		\item We have two Quillen equivalences: $$L^1: \ag \leftrightarrows \mua^{1}:\Piun \hspace{0.5in} \Piun: \mua^{1} \leftrightarrows \ag :\iota$$
		\item We have a Quillen adjunction: $\Fb^{+} : \M  \leftrightarrows \mua^{1} : \Pio$.
		\item The identity induces a right Quillen functor $\mua^{1} \to \mua_{inj}^{\bf{R}}$ which is a Quillen equivalence.
		\item The functor $\Ub$ is the composite: $ \quad \ag \xhookrightarrow[\sim]{\iota} \mua^{1} \xrightarrow{\Pio} \M.$
	\end{enumerate}
\end{cor}
\begin{proof}
	Left to the reader.
\end{proof}
\begin{cor}
With the notation above, we have two commutative diagrams of right Quillen functors:
	\[
\xy
(0,18)*+{\ag}="W";
(0,0)*+{\mua_{proj}^{\bf{R}}}="X";
(30,0)*+{\M}="Y";
(30,18)*+{\mua^{0}}="E";
{\ar@{->}^-{\Pio}_-{\sim}"X";"Y"};
{\ar@{->}^-{\iota}"W";"X"};
{\ar@{->}^-{\iota}"W";"E"};
{\ar@{->}^-{\Pio}_-{\sim}"E";"Y"};
{\ar@{.>}^-{\Id}_-{\sim}"X";"E"};
\endxy
\xy
(0,18)*+{\ag}="W";
(0,0)*+{\mua^{1}}="X";
(30,0)*+{\M}="Y";
(30,18)*+{\mua_{inj}^{\bf{R}}}="E";
{\ar@{->}^-{\Pio}"X";"Y"};
{\ar@{->}^-{\iota}_-{\sim}"W";"X"};
{\ar@{->}^-{\iota}_-{\sim}"W";"E"};
{\ar@{->}^-{\Pio}"E";"Y"};
{\ar@{.>}^-{\Id}_-{\sim}"X";"E"};
\endxy
\]
\end{cor}
\begin{proof}
	Clear.
\end{proof}
\subsection{Cofibrantly generation and monoidal model categories}
If $\M$ is cofibrantly generated we will denote by $\Iam$ and $\Jam$ the respective generating sets of cofibrations and of trivial cofibrations. In the sequel we are interested in the model category $\mua^{0}$ with the adjunction $\Fb^{+} \dashv \Pio$. 
\begin{lem}
Assume that $\M$ is a cofibrantly generated model category. Then for any Quillen adjunction $\Fb \dashv \Ub$ the following hold.
\begin{enumerate}
	\item The set $\Fb^{+}(\Iam)$ is a generating set for the cofibrations in $\mua^{0}$
	 \item The set $\Fb^{+}(\Jam)$ is a generating set for the trivial cofibrations in $\mua^{0}$
\end{enumerate}
\end{lem}
\begin{proof}
A map $\sg$ in $\mua^{0}$ has the RLP against all elements in $\Fb^{+}(\Iam)$ if and only if, by $\Fb^{+} \dashv \Pio$, the map $\Pio(\sg)$ has the RLP against all elements in $\Iam$, if and only if, $\Pio(\sg)$ is a trivial fibration , if and only if $\sg$ is a trivial fibration in $\mua^{0}$. This gives Assertion $(1)$. The other assertion is proved the same way.  
\end{proof}
\begin{lem}
Let $K: (\C, \ot, I_\C) \to (\D, \ot, I_\D)$ be strong monoidal functor that preserves pushout squares.
Then for any maps $f, g$ in $\C$ we have an isomorphism $K(f \square g) \cong K(f) \square K(g)$ in $\Arr(\D)$.
\end{lem}
\begin{proof}
	Let $\Un= [0 \to 1]$ be the \emph{walking-morphism category} and let $f:X^0 \to X^1$ and $g: Y^0 \to Y^1$ be maps in $\C$. The map $f\otimes g$ has two presentations that allows us to define a diagram $\Un^2(f\otimes g): \Un^2 \to \C$ that maps $(i,j) \mapsto X^i \ot Y^j$. Consider the punctured square $\Un^{2,\ast}:= \Un^2 \setminus \{(1,1)\}$. Then the inclusion $\Un^{2,\ast} \hookrightarrow \Un^2$ gives a diagram $\Un_{2,\ast}(f\ot g): \Un^{2,\ast} \to \C$. By definition, the map $f \square g$ is simply the induced map $\colim \Un^{2,\ast}(f\ot g) \to \colim \Un^{2}(f\ot g)$. The colimit on the left hand side is a pushout while the colimit on the right hand side is simply the evaluation at $(1,1)$ since $(1,1)$ is a terminal object in $\Un^2$. As $K$ is a strong monoidal functor, we have these isomorphisms of diagrams: $K(\Un^{2,\ast}(f\ot g)) \cong \Un^{2,\ast}(K(f)\ot K(g))$ and  $K(\Un^{2}(f\ot g)) \cong \Un^{2}(K(f)\ot K(g))$. Finally, since $K$ preserves pushouts, we get:
	 $$\colim \Un^{2,\ast}(K(f)\ot K(g)) \cong \colim K(\Un^{2,\ast}(f\ot g)) \cong K(\colim\Un^{2,\ast}(f\ot g)) $$
	 The uniqueness of the map of out of the colimit implies that $K(f)\square K(g)$ is the composite of the path below which gives the result.
	 $$ \colim \Un^{2,\ast}(K(f)\ot K(g)) \cong K(\colim\Un^{2,\ast}(f\ot g)) \xrightarrow{K(f \square g)} K(X^1\ot Y^1) \cong K(X^1) \ot K(Y^1).$$
\end{proof}
Since a left adjoint preserves all kind of colimits, the previous lemma applies if $\Fb$ is part of strong monoidal Quillen adjunction.
\begin{lem}\label{lem-push-mua-0}
	If  $\Fb \dashv \Ub$ is a strong monoidal Quillen adjunction, then the \emph{Pushout product axiom} and the \emph{Unit axiom} hold in $\mua^0$ with the point-wise tensor product.
\end{lem}
\begin{proof}
Following \cite[Corollary 4.2.5]{Hov-model}, it suffices to have the pushout product axiom for $u \in \Fb^+(\Iam)$  $v \in \Fb^+(\Iam)$ and also consider the case $u \in \Fb^+(\Jam)$ and/or $v \in \Fb^+(\Jam)$. Set $u = \Fb^+(f)$ and $v= \Fb^+(g)$. Then by the previous lemma we have $u\square v = \Fb^+(f) \square \Fb^+(g) \cong \Fb^+(f\square g)$, since $\Fb^+$ is strong monoidal. By assumption the map $f\square g$ is a  cofibration in $\M$ whose image under the left Quillen functor $\Fb^+$ is also cofibration in $\mua^0$. Furthermore, if either $f \in \Jam$ or $g \in \Jam$, $f\square g$ and $\Fb^+(f \square g)$ are trivial cofibrations. This gives the Pushout product axiom. A cofibrant replacement $[I_{\M}^c, I_{\ag}^c,\pi^c] \xrightarrow{\sim} [I_{\M}, I_{\ag}, \psi_I]$ yields a cofibrant replacement $I_{\M}^c\xrightarrow{\sim} I_{\M}$. Therefore, by the Unit axiom of the monoidal model category $\M$, the map $[I_{\M}^c, I_{\ag}^c,\pi^c] \ot \Fc \to  [I_{\M}, I_{\ag}, \psi_I] \ot \Fc$ is a weak equivalence in $\mua^{0}$, for every cofibrant $\Fc \in \mua^0$.
\end{proof}
\begin{thm}
	For any strong monoidal Quillen adjunction  $\Fb \dashv \Ub$, the following statements are true.
	\begin{enumerate}
		\item The model category $\mua^0$ is also a monoidal model category with the point-wise tensor product.
		\item We have a strong monoidal Quillen equivalence $\Fb^+: \M \leftrightarrows \mua^{0}:\Pio$
		\item We have a strong monoidal Quillen adjunction $\Piun: \mua^0 \leftrightarrows \ag : \iota$.
	\end{enumerate}
\end{thm}
\begin{proof}
We will only prove Assertion $(1)$, the others ones are straightforward.
The underlying category $\mua$ is closed by Proposition \ref{fact-lax}. Lemma \ref{lem-push-mua-0} gives the remaining axioms of a monoidal model category.
\end{proof}
\begin{rmk}
It can be easily shown that the two factorizations of $\Ub$ obtained above are also functorial.
\end{rmk}
\section{Remarks on Abelian categories and Grothendieck topologies}\label{sec-hom-alg-top}
\noindent We show that the factorization $\ag \xrightarrow{\iota} \mdua \xrightarrow{\Pi^0} \M$  can be made for abelian categories and Grothendieck sites.
\begin{thm}
	Let $\Ub: \ag \to \M$ be a functor with the factorization $\ag \xrightarrow{\iota} \mdua \xrightarrow{\Pi^0} \M$. Then the following hold.
	\begin{enumerate}
		\item If $\Ub$ is a left exact functor between abelian categories that possesses a left adjoint, then $\mdua$ is an abelian category and the functors $\iota$ and $\Pi^0$ are also left exact.
		\item If $\Ub$ is a morphisms of sites, then $\mdua$ is a Grothendieck site and the functors $\iota$ and $\Pi^0$ are also morphisms of sites.
	\end{enumerate}
\end{thm}
The proof of the theorem is not hard but long winded, so we divide it into small intermediate lemmas. The definition of Grothendieck topologies is to be found for example in \cite{Maclane_Moer_sh, Vistoli_fib}. 

\begin{prop}
	For any (finite) limit preserving functor $\Ub: \ag \to \M$ between (finitely) complete categories, the following hold.
	\begin{enumerate}
		\item The category $\mua$ is (finitely) complete and limits are computed level-wise.
		\item The functors $\iota$ and $\Pio$ preserve also (finite) limits.
		\item  Any epimorphism in $\mua$ is a level-wise epimorphism.
	\end{enumerate}
\end{prop}
\begin{proof}
Assertion $(1)$ and Assertion $(2)$ can be found in \cite{Bacard_QS}. For Assertion $(3)$ we proceed as follows.
Since $\ag$ and $\M$ are finitely complete, there are terminal objects $\ast_{\ag}$ and $\ast_{\M}$, which give the adjunction $\Pio \dashv R^0$. This adjunction induces an equivalence of commutative diagrams:
$$(\F^0 \xrightarrow{\sgo} \G^0 \rightrightarrows m) = (\Pio(\Fc) \xrightarrow{\Pio(\sg)} \Pio(\Gc) \rightrightarrows m) \Longleftrightarrow (\Fc \xrightarrow{\sg} \Gc \rightrightarrows R^0(m)).$$
If $\sg$ is an epimorphism, then the two maps $\Gc \rightrightarrows R^0(m)$ are equal which implies that their adjoint-transpose maps $\G^0 \rightrightarrows m$ are also equal, whence $\sg^0$ is an epimorphism. Similarly, using the adjunction $\Piun \dashv \iota$, one has an equivalence of commutative diagrams:
$$(\F^1 \xrightarrow{\sgi} \G^1 \rightrightarrows a) = (\Piun(\Fc) \xrightarrow{\Piun(\sg)} \Piun(\Gc) \rightrightarrows a) \Longleftrightarrow (\Fc \xrightarrow{\sg} \Gc \rightrightarrows \iota(a)).$$
Therefore if $\sg$ is an epimorphism, the parallel maps $ \Gc \rightrightarrows \iota(a)$ must be equal which means that the adjoint-transpose maps $\G^1 \rightrightarrows a$ are also equal, proving that $\sgi$ is an epimorphism.
\end{proof}
\begin{lem}\label{lem-eq-coeq}
	Given an adjunction $\Fb \dashv \Ub$ with $\Ub: \ag \to \M$, the following hold.
	\begin{enumerate}
		\item A commutative diagram $D=(\Ec \xrightarrow{\sg} \Fc \rightrightarrows \Gc) $ is an equalizer diagram in $\mua$ if and only if $\Pio(D)$ and $\Piun(D)$ are equalizer diagrams in $\M$ and $\ag$ respectively.
		\item Dually, a commutative diagram $D=( \Fc \rightrightarrows \Gc \xrightarrow{p} \Qc) $ is a coequalizer diagram in $\mua$ if and only if $\Pio(D)$ and $\Piun(D)$ are coequalizer diagrams in $\M$ and $\ag$ respectively.
	\end{enumerate}
\end{lem}

\begin{proof}
This is easy considering the fact that the functor $\mua \xrightarrow{\Pio \times \Piun} \M \times \ag$ creates limits and colimits (see \cite{Bacard_QS}). However, for completeness and for the reader's convenience, we give a detailed explanation for Assertion $(1)$. The same argument applies for Assertion $(2)$, with the difference that $\Ub$ needs not be a right adjoint.
Consider the equalizer $(\Xc \xrightarrow{\theta}  \Fc \rightrightarrows \Gc )$ diagram in $\mua$. Then there is a unique map $\gamma : \Ec \to \Xc$ such that $\sg = \theta \circ \gamma$. We shall see that the map $\gamma$ is an isomorphism: this will prove that $\sg$ is the equalizer of the two maps $(\Fc \rightrightarrows \Gc )$. The map $\gamma$ will be an isomorphism if we can show that $\gamma^0$ and $\gamma^1$ are isomorphisms. We give the argument for $\gamma^0$, leaving $\gamma^1$ to the reader. We note that $\sgo = \theta^0 \circ \gamma^0$. By assumption the two diagrams $(\Ea^0 \xrightarrow{\sgo} \F^0 \rightrightarrows \G^0 )$, $(\Xa^0 \xrightarrow{\theta^0} \F^0 \rightrightarrows \G^0 )$ are equalizer diagrams, therefore an application of the universal property of the equalizer gives a unique map $j^0: \Xa^0 \to \Ea^0$, such that $\theta^0 = \sgo \circ j^0.$
If we put the last two equalities together we have on the one hand:
$$\sgo = \theta^0 \circ \gamma^0 = \sgo \circ j^0  \circ \gamma^0 \Longleftrightarrow \sgo \circ \Id_{\F^0} = \sgo \circ(j^0  \circ \gamma^0).$$
Since $\sgo$ is a monomorphism - like any equalizer -, the last equality gives  $\Id_{\F^0} = j^0  \circ \gamma^0$.  On the other hand, we also have:
$$ \theta^0 = \sgo \circ j^0 = \theta^0 \circ \gamma^0 \circ   j^0 \Longleftrightarrow \theta^0 \circ \Id_{\Xa^0} = \theta^0 \circ ( \gamma^0 \circ   j^0).$$
As $\theta^0$ is also a monomorphism, the last equality gives  $\Id_{\Xa^0} = \gamma^0 \circ   j^0$.  It follows that  $j^0$ and $\gamma^0$ are inverses to each other which means that $\gamma^0$ is an isomorphism as desired. This proves Assertion $(1)$ which completes the proof the lemma.
\end{proof}
With the functors $\Pio$ and $\Piun$, it can be easily seen that any level-wise monomorphism is a monomorphism, i.e., if $\sgo$ and $\sgi$ are monomorphisms, then so is $\sg=\osg$. The proposition hereafter gives the converse statement under some hypotheses. 
\begin{prop}\label{char-mono-mua}
Given an adjunction $\Fb \dashv \Ub$ with $\Ub: \ag \to \M$, the following hold.
\begin{enumerate}
	\item If  there is an initial object $\emptyset_{\ag} \in \ag$,  then any monomorphism in $\mdu$ is a level-wise monomorphism.
	\item  Moreover, if the monomorphisms in $\ag$ and $\M$ are effective monomorphisms, then so are the monomorphisms of $\mua$.
\end{enumerate}
\end{prop}

\begin{proof}
  For Assertion $(1)$, we shall proceed as follows. Let $\sg : \Fc \xrightarrow{\osg} \Gc$ be a monomorphism. The existence of an initial object $\emptyset_{\ag}$ gives the adjunction $L^{1} \dashv \Piun$, thus any map $a \to \F^1$  is uniquely equivalent to a map $L^{1}(a) \to \Fc$. Given two parallel morphisms $j, j': a \rightrightarrows \F^1 $, we have by adjointness, an equivalence of commutative diagrams:
$$( a \rightrightarrows \F^1 \xrightarrow{\sgi} \G^1) = (a \rightrightarrows  \Piun(\Fc) \xrightarrow{\Piun(\sg)} \Piun(\Gc) ) \Longleftrightarrow (L^{1}(a)\rightrightarrows  \Fc \xrightarrow{\sg} \Gc ).$$
It follows that if $\sg$ is a monomorphism then the two parallel maps $L^{1}(a)\rightrightarrows  \Fc$ are equal, and so are their adjoint-transpose maps $j, j': a \rightrightarrows \F^1 $, which proves that $\sgi$ is a monomorphism. To show that $\sgo$ is also a monomorphism, the existence of a left adjoint is essential. First note that the adjunction $\Fb \dashv \Ub$ gives another adjunction $\Fb^+ \dashv \Pio$. As in the first case, given two parallel maps $m \rightrightarrows \F^0$ we have an equivalence of commutative diagrams:
$$( m \rightrightarrows \F^0 \xrightarrow{\sgo} \G^0) = (m \rightrightarrows  \Pio(\Fc) \xrightarrow{\Pio(\sg)} \Pio(\Gc) ) \Longleftrightarrow (\Fb^+(m)\rightrightarrows  \Fc \xrightarrow{\sg} \Gc ).$$
$\sg$ being a monomorphism gives an equality between the parallel maps $\Fb^+(m)\rightrightarrows  \Fc$, which gives by adjointness, an equality of the parallel adjoint-transposes $m \rightrightarrows \F^0$, proving that $\sgo$ is a monomorphism. This gives Assertion $(1)$.

With Assertion $(1)$ at hand, if $\sg=\osg$ is a monomorphism , then $\sgo$ and $\sgi$ are monomorphisms which are effective monomorphisms by assumption. Recall that colimits and  limits, thus pushouts and equalizers, are computed level-wise in $\mua$. Consider the diagram $(\Gc \rightrightarrows \Gc \cup^{\Fc} \Gc )$ given by the two parallel maps obtained by forming the pushout of $\sg: \Fc \to \Gc$ along itself and let $D = (\Fc \xrightarrow{\sg} \Gc \rightrightarrows \Gc \cup^{\Fc} \Gc )$. Note that since $\sgo$ and $\sgi$ are effective monomorphisms, the diagrams $ \Pio(D)=(\F^0 \xrightarrow{\sgo} \G^0 \rightrightarrows \G^0 \cup^{\F^0} \G^0 )$  and $\Piun(D)=(\F^1 \xrightarrow{\sgi} \G^1 \rightrightarrows \G^1 \cup^{\F^1} \G^1)$ are equalizer diagrams. By Lemma \ref{lem-eq-coeq}   $D$ is an equalizer diagram which means that $\sg$ is an effective monomorphism, whence the assertion.
\end{proof}
\subsection{Left exact functors between abelian categories}
\begin{prop}\label{prop-abelian}
	 If $\Ub: \ag \to \M$ is a left exact functor between abelian categories that posses a left adjoint, then $\mdua$ is an abelian category and the functors $\iota$ and $\Pi^0$ are also left exact.
\end{prop}
A typical example is the functor $\Ub=\HOM(P,-): k-Mod \to k-Mod$, where $P$ a left $k$-module for some commutative ring $k$. It is well known that $\HOM(P,-)$ is a left exact endofunctor whose left adjoint is $-\otimes P$ (see \cite{Weibel_HA}).
\begin{rmk}\label{zero-ab-comma}
	We note that the functor $\Ub$ sends zero object to zero object, i.e., $\Ub(\0_{\ag}) =\0_{\M}$, therefore
	$\0 = [\0_{\M}, \0_{\ag} ,  \Id_{\0_{\M}}]$ is clearly a zero object in $\mdua$. If $\Fc$ and $\Gc$ are objects in $\mua$, we have a map $\Fc \to \Fc \times \Gc$ given by the identity $\Id_{\Fc}$ and the zero map $\Fc \to \Gc$. Similarly we have a map $\Gc \to \Fc \times \Gc$. The two maps induce a unique map $\Fc \oplus \Gc \to \Fc \times \Gc$ satisfying the obvious equations.
\end{rmk}
We will need the following lemma which is classic in the theory of additive functors.
\begin{lem}
 Any left exact functor  $\Ub : \ag \to \M$ between abelian categories is additive, i.e., the canonical map
 $\Ub(X) \oplus \Ub(Y) \to \Ub(X \oplus Y) $ is an isomorphism for any $(X,Y) \in \ag^2$. 
\end{lem}
\begin{proof}
By assumption, the functor $\Ub$ preserves all finite limits, in particular it preservers binary products, that is, we have a canonical isomorphism $\Ub(X \times Y) \xrightarrow[\cong]{\beta} \Ub(X) \times \Ub(Y)$. The canonical isomorphism $X \oplus Y \xrightarrow[\cong]{\alpha_{\ag}} X \times Y$ fits in the commutative diagram below.
\[
\xy
(-30,20)*+{X}="W";
(-30,0)*+{Y}="X";
(50,0)*+{Y}="Y";
(50,20)*+{X}="E";
(25,10)*+{X \times Y}="U";
(0,10)*+{X \oplus Y}="V";
{\ar@{->}_-{i_2}"X";"V"};
{\ar@{->}_-{p_2}"U";"Y"};
{\ar@{->}^-{\alpha_{\ag}}_{\cong}"V";"U"};
{\ar@{->}^-{p_1}"U";"E"};
{\ar@{->}^-{i_1}"W";"V"};
{\ar@{->}^-{\Id}"W";"E"};
{\ar@{->}^-{\Id}"X";"Y"};
\endxy
\] 
The image under $\Ub$ of the last diagram gives a new commutative diagram in which we've included the isomorphism $\beta$ as well as the canonical map of the lemma:
\begin{align}\label{mid-row}
\xy
(-60,20)*+{\Ub(X)}="W";
(-60,-20)*+{\Ub(Y)}="X";
(90,-20)*+{\Ub(Y)}="Y";
(90,20)*+{\Ub(X)}="E";
(-30,0)*+{\Ub(X)\oplus \Ub(Y)}="F";
(63,0)*+{\Ub(X) \times  \Ub(Y)}="Q";
(31,0)*+{\Ub(X \times Y)}="U";
(3,0)*+{\Ub(X \oplus Y)}="V";
{\ar@{->}_-{\Ub(i_2)}"X";"V"};
{\ar@{-->}^-{i_2}"X";"F"};
{\ar@{-->}^-{can}"F";"V"};
{\ar@{->}_-{\Ub(p_2)}"U";"Y"};
{\ar@{->}^-{\Ub(\alpha_{\ag})}_{\cong}"V";"U"};
{\ar@{->}^-{\Ub(p_1)}"U";"E"};
{\ar@{->}^-{\Ub(i_1)}"W";"V"};
{\ar@{-->}_-{i_1}"W";"F"};
{\ar@{->}^-{\Id}"W";"E"};
{\ar@{->}^-{\Id}"X";"Y"};
{\ar@{.>}^-{p_2}"Q";"Y"};
{\ar@{.>}^-{p_1}"Q";"E"};
{\ar@{->}^-{\beta}_{\cong}"U";"Q"};
\endxy
\end{align}
The composite of the middle row is exactly the isomorphism $\alpha_\M: \Ub(X) \oplus \Ub(Y) \xrightarrow{\cong} \Ub(X) \times \Ub(Y) $ in the abelian category $\M$. Since the other morphisms in this row are invertible, we see that   $\Ub(X) \oplus \Ub(Y) \xrightarrow{can} \Ub(X \oplus Y) $ is also an isomorphism as claimed.
\end{proof}
\begin{lem}\label{can-coprod-prod-comma}
Let $\Fc$ and $\Gc$  be objects of $\mua$, where $\Ub$ is a left exact functor between abelian categories. Then the canonical map $\Fc \oplus \Gc \xrightarrow{\alpha} \Fc \times \Gc$ in Remark \ref{zero-ab-comma}  is an isomorphism given by $\alpha_\M : \F^0 \oplus \G^0 \xrightarrow{\cong} \F^0 \times \G^0$ and 
$\alpha_{\ag} : \F^1 \oplus \G^1 \xrightarrow{\cong} \F^1 \times \G^1$.
\end{lem}
\begin{proof}
Since $\Ub$ preserves products and coproducts, then $\Fc \oplus \Gc$ and $\Fc \times \Gc$ are computed level-wise. By inspection we have $ \Fc \oplus \Gc= [\F^0 \oplus \G^0, \F^1 \oplus \G^1, \pi_{\oplus}]$ where $\pi_{\oplus}$ is the composite of the path: $\F^0 \oplus \G^0 \xrightarrow{\pif \oplus \pig} \Ub(\F^1) \oplus \Ub(\G^1) \xrightarrow[\cong]{can} \Ub(\F^1 \oplus \G^1)$. We also have  $ \Fc \times \Gc= [\F^0 \times \G^0, \F^1 \times \G^1, \pi_{\times}]$, where $\pi_{\times}$ is the composite of the path: 
$$\F^0 \times \G^0 \xrightarrow{\pif \times \pig} \Ub(\F^1) \times \Ub(\G^1) \xrightarrow[\cong]{\beta^{-1}} \Ub(\F^1 \times \G^1).$$
In the abelian category $\M$ the isomorphism $\alpha_\M$ is natural in the two variables, thus we have a natural isomorphism in the arrow-category $\Arr(\M)$: $\pif \oplus \pig \xrightarrow{\cong} \pif \times \pig$, given by $\alpha_\M : \F^0 \oplus \G^0 \xrightarrow{\cong} \F^0 \times \G^0$ and 
$\alpha_{\M}^1 : \Ub(\F^1) \oplus \Ub(\G^1) \xrightarrow{\cong} \Ub(\F^1) \times \Ub(\G^1)$. In other words we have $\alpha_{\M}^1 \circ (\pif \oplus \pig)= (\pif \times \pig) \circ \alpha_{\M}$, from which we get: 
\begin{align}\label{eq-ab-1}
\beta^{-1} \circ \alpha_{\M}^1 \circ (\pif \oplus \pig)= \underbrace{\beta^{-1} \circ (\pif \times \pig)}_{\pi_{\times}} \circ \alpha_{\M}.
\end{align}
The middle row of (\ref{mid-row}) gives: 
\begin{align}\label{eq-ab-2}
\beta^{-1} \circ \alpha_{\M}^1 = \Ub(\alpha_{\ag}) \circ can.
\end{align}
Clearly, (\ref{eq-ab-1}) and (\ref{eq-ab-2}) give:
 $$ \Ub(\alpha_{\ag}) \circ can \circ  (\pif \oplus \pig) = \pi_{\times} \circ \alpha_{\M} \Longleftrightarrow \Ub(\alpha_{\ag}) \circ \pi_{\oplus} =  \pi_{\times} \circ \alpha_{\M}.$$
 The last equality implies that $[\alpha_{\M}, \alpha_{\ag}]$ defines an isomorphism $\alpha \in \Hom(\Fc \oplus \Gc, \Fc \times \Gc)$. By inspection, the map $\alpha$  satisfies the same equations as the canonical map in Remark \ref{zero-ab-comma}, whence the equality by uniqueness.
\end{proof}

\begin{lem}\label{lem-mono-epi}
 If $\Ub: \ag \to \M$ is a left exact functor between abelian categories that posses a left adjoint, then the following statements are true:
 \begin{enumerate}
 	\item Every monomorphism is the kernel of some morphism.
 	\item Every epimorphism is the cokernel of some morphism.
 \end{enumerate}
\end{lem}

\begin{proof}
	Let $\sg:\Fc \xrightarrow{\osg} \Gc$ be a monomorphism and consider its cokernel $p:\G \to coker(\sg)$,  obtained as the coequalizer of $\sg$ and the zero map. Since (co)equalizers are computed level-wise, we have $p=[p^0,p^1]$ where $p^0: \G^0 \to coker(\sgo)$ and $p^1: \G^1 \to coker(\sgi)$ are the respective coequalizers of $\sgo$ and $\sgi$. 
	By Proposition \ref{char-mono-mua}, the map $\osg$ is a level-wise monomorphism, that is $\sgo$ and $\sgi$ are also monomorphism in $\M$ and $\ag$ respectively. By the axioms of the abelian categories $\M$ and $\ag$, we have:
	$\sgo \cong \ker(coker(\sgo) )$ and $\sgi \cong \ker(coker(\sgi))$. In particular the diagram hereafter are equalizer diagrams:
	\begin{align}
    & \xy
	(0,0)*+{\F^0}="X";
	(20,0)*+{\G^0}="Y";
	(50,0)*+{coker(\sgo)}="E";
	{\ar@{->}^-{\sgo}"X";"Y"};
	{\ar@<-1.ex>_-{0}"Y";"E"};
	{\ar@<1.ex>^-{p^0}"Y";"E"};
	\endxy
	\\
   &\xy
     (0,0)*+{\F^1}="X";
     (20,0)*+{\G^1}="Y";
     (50,0)*+{coker(\sgi)}="E";
     {\ar@{->}^-{\sgi}"X";"Y"};
     {\ar@<-1.ex>_-{0}"Y";"E"};
     {\ar@<1.ex>^-{p^1}"Y";"E"};
	\endxy
	\end{align}
	Let $D$ be the commutative diagram:
	$
	\xy
	(0,0)*+{\Fc}="X";
	(20,0)*+{\Gc}="Y";
	(40,0)*+{coker(\sg)}="E";
	{\ar@{->}^-{\sg}"X";"Y"};
	{\ar@<-1.ex>_-{[0]}"Y";"E"};
	{\ar@<1.ex>^-{[p^0,p^1]}"Y";"E"};
	\endxy
	$. 
	By the above, $\Pio(D)$ and $\Piun(D)$ are equalizer diagrams, thus $D$ is an equalizer diagram by Lemma \ref{lem-eq-coeq}. It follows that $\sg \cong ker(coker(\sg))$, whence Assertion $(1)$. The other assertion is proved the same way using  $\sgo \cong coker(ker(\sgo))$ and $\sgi \cong coker(ker(\sgi))$ with the dual diagram $D=( \xy
	(40,0)*+{\Fc}="X";
	(0,0)*+{ker(\sg)}="Y";
	(20,0)*+{\Gc}="E";
	{\ar@{->}^-{\sg}"E";"X"};
	{\ar@<-1.ex>_-{[0]}"Y";"E"};
	{\ar@<1.ex>^-{[i^0,i^1]}"Y";"E"};
	\endxy)
	$.
\end{proof}
We can now prove Proposition \ref{prop-abelian}.
\begin{proof}[Proof of Proposition \ref{prop-abelian}]
 We note that colimits are always computed level-wise, in particular coproducts, coequalizers, whence cokernels exist in $\mua$. Moreover, being left-exact means that $\Ub$ preserves finite limits, therefore equalizers and kernels exist in $\mua$ and are computed level-wise. 
\begin{itemize}
	\item By Remark \ref{zero-ab-comma}, we have a zero-object: $\0 = [\0_{\M}, \0_{\ag} ,  \Id_{\0_{\M}}]$.
	\item Lemma \ref{can-coprod-prod-comma} asserts that the canonical map $\Fc \oplus \Gc \to \Fc \times \Gc$ is an isomorphism given by $[\alpha_{\M}, \alpha_{\ag}]$. By a classical argument $\mua$ becomes an additive category.
	\item Lemma \ref{lem-mono-epi} shows that any monomorphism is a kernel and every epimorphism is a cokernel.
\end{itemize}
\end{proof}
\subsection{Grothendieck topology on a comma category}
It is shown for example in \cite{Vistoli_fib} that if $\C$ is a Grothendieck site, then there is an induced Grothendieck topology on any comma category $\C_{ / X}$. It's easy to see that  $\C_{\/X}\cong\C \downarrow \1_X$ where $\1_X: \1 \to \C$ is the functor that picks out $X$. We show below that this result can be generalized to other functors $\Ub: \ag \to \M$ under reasonable hypotheses. First, we recall a definition that can be found in \cite[Ch. VII]{Maclane_Moer_sh}. 
\begin{df}
Let $(\C,J)$ and $(\D,K)$ be Grothendieck sites and suppose that $\C$ and $\D$ are closed under finite limits. Say that a functor $\phi: \C \to \D$ is a morphism of sites if the following two conditions hold.
\begin{enumerate}\label{df-site-mor}
	\item $\phi$ preserves finite limits, i.e. $\phi$ is left exact.
	\item $\phi$ preserves covers: if $S\in J(C)$ is a covering sieve for $C \in \C$, then the sieve $\phi(S)$ generated by the set $\{\phi(u) | \quad (u: C' \to C)\in S\}$ is a covering sieve for $\phi(C)$ in $\D$.
\end{enumerate}
\end{df}

\begin{df}
	Let $\Gc=\ogcr$ be an object of $\mdua$, where $\Ub$ is a left exact functor between finitely complete categories with a Grothendieck topology
	\begin{enumerate}
		\item Say that a set of morphism $S = \{\Xci \xrightarrow{\theta_i} \Gc \}$ is a covering family for $\Gc$ if 
		\begin{enumerate}
			\item $\Pi^0(S)= \{\Xa_i^0 \xrightarrow{\theta_i^0} \G^0\}$ is a covering family for $\G^0$ and if
			\item $\Pi^1(S)= \{\Xa_i^1 \xrightarrow{\theta_i^1} \G^1\}$ is a covering family for $\G^1$.
		\end{enumerate}
		\item Denote by $\bcov(\Gc)$ the collection of all covering families for $\Gc$.
	\end{enumerate}
\end{df}
\begin{prop}
	With the notation above, the following statements are true.
	\begin{enumerate}
		\item The collections $\bcov(\Gc)$ determine a basis of a Grothendieck topology -generated by $\bcov$- on the comma category $\mdua$.
		\item If $\Ub$ is a morphism of sites, then so are $\iota :\ag \to \mua$ and $\Pio: \mdua \to \M$. Moreover $\Ub$ can be factored as a morphism of sites : $(\ag \xhookrightarrow{\iota} \mdua \xrightarrow{\Pio} \M)$
	\end{enumerate} 
\end{prop}
\begin{proof}
	We verify the axioms of a basis for a topology (see \cite[Ch. III]{Maclane_Moer_sh}).
	If $\Fc \xrightarrow[\cong]{\sg} \Gc$ is an isomorphism then - by definition - $\Pio(\sg)$ and $\Piun(\sg)$ are isomorphisms, thus $\{\Pio(\sg): \F^0 \xrightarrow{\cong} \G^0\}$ and $\{\Piun(\sg): \F^1 \xrightarrow{\cong} \G^1\}$ are covering families in the respective sites, whence $\{ \Fc \xrightarrow[\cong]{\sg} \Gc\} \in \bcov(\Gc)$. The stability axioms for a basis of topology is also straightforward. Indeed, under the hypothesis of left exactness, fiber products are computed level-wise, therefore If $S =  \{\Xci \xrightarrow{\theta_i} \Gc\} \in \bcov(\G)$ and $\sg :\Fc \to \Gc$ is an arbitrary map, one has: $\sg^{\ast}(S) =  \{\Fc \times_{\Gc}  \Xci \xrightarrow{[(\sg^0)^{\ast}\theta_i^0, (\sg^1)^{\ast} \theta_i^1]} \Fc \}$. By the stability axiom of the respective topologies on $\ag$ and $\M$, the set $(\sg^1)^{\ast}(\Piun(S)) = \{(\sg^1)^{\ast}\theta_i^0 \quad | \quad \theta_i \in S\}$ is a covering family for $\F^1$ and $(\sg^0)^{\ast}(\Pio(S)) = \{(\sg^0)^{\ast}\theta_i^0 \quad | \quad \theta_i \in S\}$ is a covering for $\F^0$, thus $\sg^{\ast}(S) \in \bcov(\Fc)$. The transitivity axiom is always true without any assumption on the functor $\Ub$, since the composition of morphisms in $\mdua$ is defined level-wise. Put differently, if $\{\Xci \xrightarrow{\theta_i} \Gc \} \in \bcov(\Gc)$, and if for every $i$ we have a family $\{[\Ya_{ij}] \xrightarrow{\sg_{ij}} \Xci \} \in \bcov(\Xci)$, then $\{[\Ya_{ij}] \xrightarrow{\theta_i \circ \sg_{ij}} \Gc \} \in \bcov(\Gc)$, which proves Assertion $(1)$. 
	
	By definition of the topology on $\mdua$, the functor $\Pio$ clearly preserves covers. Moreover, since limits are created along  the functor $\Pio \times \Piun: \mdua \to \M \times \ag$, it is also clear that $\Pio$ preserves finite limits, therefore $\Pio$ is a morphism of sites. The functor $\iota$ is always a right adjoint by the adjunction $\Piun \dashv \iota$, so it preserves any kind of limits, in particular it is a left exact functor (= preserves finite limits). It remains to prove that $\iota$ preserves covers in the sense of Definition \ref{df-site-mor}. If $S$ is a covering sieve for $\Qa \in \ag$, then the sieve $\iota(S)$ generated  by the set $\{ \iota(u) = [\Ub(u), u] | \quad (\Pa \xrightarrow{u} \Qa) \in S\}$ is the family: 
	\begin{align}
	\iota(S)=\{ \iota(u) \circ \osg \quad | u \in S, \quad cod(\osg) = \iota(\Pa)= [\Ub(\Pa), \Pa, \Id_{\Ub(\Pa)}]\}.
	\end{align}
We note that given any map $f:m \to \Ub(\Pa)$ in $\M$, we have an object $[m, \Pa, f] \in \mdua$ with a canonical map $[f, \Id_{\Pa}]: [m, \Pa, f]  \to \iota(\Pa)$. By assumption, $\Ub$ preserves covers, therefore the sieve $\Ub(S) =\{ \Ub(u) \circ f  | \quad (\Pa \xrightarrow{u}\Qa) \in S, \quad cod(f)= \Ub(\Pa)\}$ is a covering sieve for $\Ub(\Qa)$. It follows that the family $\{\iota(u) \circ [f,\Id_{\Pa}] \quad | \quad (\Pa \xrightarrow{u}\Qa) \in S, \quad cod(f)= \Ub(\Pa)\} \in \bcov(\iota(\Qa))$. Clearly, we have the inclusion $\{\iota(u) \circ [f,\Id_{\Pa}] \quad | \quad (\Pa \xrightarrow{u}\Qa) \in S, \quad cod(f)= \Ub(\Pa)\} \subseteq \iota(S)$ which implies that $\iota(S)$ is a covering sieve for $\iota(\Qa)$ by definition of the topology generated by the basis $\bcov(\iota(\Qa))$. This proves Assertion $(2)$. 
\end{proof}
\begin{rmk}
For Assertion $(1)$, one simply needs to have a functor $\Ub$ that preserves pullbacks and not necessarily all finite limits.
\end{rmk}
\section{Proofs}

\subsection{Proof of Proposition \ref{prop-left-ehk}}\label{proof-prop-left-ehk}
	\begin{proof}
		We define the left adjoint $E(H,K)_\ast$ as follows. If $[X]= [X^0,X^{1}, \pi_X] \in (\M'\downarrow \Ub')$, then $E(H,K)_\ast([X])= [K_\ast X^0, H_\ast X^{1}, K_\ast X^0\to \Ub(H_\ast X^{1})]$, where $K_\ast X^0\to \Ub(H_\ast X^{1})$  is obtained through the following steps. Consider the map $\alpha \in \Hom(X^0, K\Ub(H_\ast (X^1)))$ displayed as the composite:
		\begin{align}
		\alpha = [\pi_X: X^0 \to \Ub'(X^1) \xrightarrow{\Ub(\eta_{X^1})} \Ub'H(H_\ast X^1) \xrightarrow{\Id} K(\Ub(H_\ast(X^1)))]
		\end{align}
		With the adjunction $K_\ast \dashv K$ the map $\alpha: X^0 \to K(\Ub(H_\ast(X^1)))$ has a unique adjoint-transpose $\alpha_\ast([X]) : K_\ast(X^0) \to \Ub(H_\ast X^1)$  fitting in the equality:
		\begin{align}
		\alpha = X^0 \xrightarrow{\eta_{X^0}} KK_\ast(X^0) \xrightarrow{K(\alpha_\ast)}  K(\Ub(H_\ast(X^1))).
		\end{align}
		If $\varepsilon: K_\ast K \to \Id_{\M}$ is the counit in the adjunction $(K_\ast \dashv K)$, we have an explicit formula $\alpha_\ast([X]) = \varepsilon_{\Ub(H_\ast(X^1))} \circ K_\ast (\alpha)$, i.e, $\alpha_\ast([X]) $ is the composite of the path:
		\begin{align}
	    K_\ast(X_0) \xrightarrow{K_\ast(\alpha)} K_\ast K(\Ub(H_\ast(X^1))) \xrightarrow{\varepsilon_{\Ub(H_\ast(X^1))}} \Ub(H_\ast(X^1))
		\end{align}
		Then we have $E(H,K)_\ast([X])= [K_\ast X^0, H_\ast X^{1},\alpha_\ast([X])]$. By construction, we have a commutative diagram:
		\begin{align}\label{unit-adj-comma}
		\xy
		(0,18)*+{X^0}="W";
		(0,0)*+{\Ub'(X^1)}="X";
		(30,0)*+{\Ub'(HH_\ast(X^1))}="Y";
		(30,18)*+{KK_\ast (X^0)}="E";
		(60,0)*+{K(\Ub H_\ast(X^1))}="Z";
		{\ar@{->}^-{\Ub'(\eta_{X^1})}"X";"Y"};
		{\ar@{->}^-{\pi_{X}}"W";"X"};
		{\ar@{->}^-{\eta_{X^0}}"W";"E"};
		{\ar@{->}_-{K(\alpha_\ast([X]))}"E";"Y"};
		{\ar@{->}^-{\Id}"Y";"Z"};
		{\ar@{->}^-{K(\alpha_\ast([X]))}"E";"Z"};
		\endxy
		\end{align}
		Moreover, if $\theta:[X] \xrightarrow{[\theta^0, \theta^1] } [Y]$ is a morphism in $(\M'\downarrow \Ub')$ the map $E_\ast(H,K)(\theta): E_\ast(H,K)([X]) \to E_\ast(H,K)([Y])$ is given by the couple $[K_\ast(\theta^0), H_\ast(\theta^1)]$ as displayed by the commutative diagram on the right: 
		\[
		\xy
		(0,18)*+{X^0}="W";
		(0,0)*+{K\Ub(H_\ast(X^1))}="X";
		(40,0)*+{K\Ub(H_\ast(Y^1))}="Y";
		(40,18)*+{Y_0}="E";
		(50,8)*+{}="A";
		(60,8)*+{}="B";
		{\ar@{->}^-{K\Ub H_\ast(\theta^1)}"X";"Y"};
		{\ar@{->}^-{\alpha}"W";"X"};
		{\ar@{->}^-{\theta^0}"W";"E"};
		{\ar@{->}^-{\alpha}"E";"Y"};
		{\ar@{=>}^-{K_\ast(-)~~\tx{and} ~~\varepsilon}"A";"B"};
		\endxy
		\xy
		(0,18)*+{K_\ast X^0}="W";
		(0,0)*+{K_\ast K\Ub(H_\ast(X^1))}="X";
		(50,0)*+{K_\ast K \Ub(H_\ast(Y^1))}="Y";
		(50,18)*+{Y_0}="E";
		{\ar@{->}^-{K_\ast K\Ub(H_\ast(\theta^1))}"X";"Y"};
		{\ar@{->}^-{K_\ast(\alpha)}"W";"X"};
		{\ar@{->}^-{K_\ast(\theta^0)}"W";"E"};
		{\ar@{->}^-{K_\ast(\alpha)}"E";"Y"};
		(0,-18)*+{\Ub(H_\ast(X^1))}="S";
		(50,-18)*+{\Ub(H_\ast(Y^1))}="T";
		{\ar@{->}^-{\varepsilon_{\Ub(H_\ast(X^1))}}"X";"S"};
		{\ar@{->}^-{\varepsilon_{\Ub(H_\ast(Y^1))}}"Y";"T"};
		{\ar@{->}^-{\Ub(H_\ast(\theta^1))}"S";"T"};
		\endxy
	\]
We leave the reader to check that these data define a functor $E(H,K)_\ast : (\M'\downarrow \Ub') \to \mdua$. It remains to prove that we have some functorial isomorphisms of hom-sets:
		$$\Hom[E(H,K)_\ast([X]), \Fc]  \cong \Hom[[X], E(H,K)(\Fc) ].$$
		To prove this, assume that we have two maps $\sgo: K_\ast(X^0) \to \F^0$ and $\sgi : H_\ast(X^1) \to \F_1$ defining a morphism $\sg: E(H,K)_\ast([X]) \xrightarrow{[\sgo, \sgi]} \Fc$. By definition, we have a commutative diagram:
		\begin{align}
		\xy
		(0,18)*+{K_\ast(X^0)}="W";
		(0,0)*+{\Ub(H_\ast(X^1))}="X";
		(30,0)*+{\Ub(\F^1)}="Y";
		(30,18)*+{\F^0}="E";
		{\ar@{->}^-{\Ub(\sigma^{1})}"X";"Y"};
		{\ar@{->}^-{\alpha_\ast([X])}"W";"X"};
		{\ar@{->}^-{\sigma^0}"W";"E"};
		{\ar@{->}^-{\pi_{\G}}"E";"Y"};
		\endxy
		\end{align}
		If we apply the functor $K$ to the last diagram, we get a commutative diagram to which we've concatenated the diagram (\ref{unit-adj-comma}):
		\begin{align}\label{adj-diag}
		\xy
		(-30,18)*+{X^0}="O";
		(-30,-18)*+{\Ub'(X^1)}="P";
		(0,18)*+{K(K_\ast(X^0))}="W";
		(0,0)*+{K(\Ub(H_\ast(X^1)))}="X";
		(40,0)*+{K(\Ub(\F^1))}="Y";
		(40,18)*+{K(\F^0)}="E";
		{\ar@{->}^-{K(\Ub(\sigma^{1}))}"X";"Y"};
		{\ar@{->}^-{K(\alpha_\ast([X]))}"W";"X"};
		{\ar@{->}^-{K(\sigma^0)}"W";"E"};
		{\ar@{->}^-{K(\pi_{\F})}"E";"Y"};
		(0,-18)*+{\Ub'H(H_\ast(X^1))}="S";
		(40,-18)*+{\Ub'H((\F^1))}="T";
		{\ar@{=}^-{\Id}"X";"S"};
		{\ar@{=}^-{\Id}"Y";"T"};
		{\ar@{->}^-{\Ub'H(\sgi)}"S";"T"};
		{\ar@{->}^-{\pi_X}"O";"P"};
		{\ar@{->}^-{\Ub'(\eta_{X^1})}"P";"S"};
		{\ar@{->}^-{\eta_{X^0}}"O";"W"};
		\endxy
		\end{align}
		Clearly, the last diagram defines a map $[X] \to E(H,K)(\Fc)$ whose components are $[K(\sgo) \circ \eta_{X^0}, H(\sgi) \circ \eta_{X^1}] = [\varphi_0(\sgo), \varphi_1(\sgi)]$, where $\varphi_0$ and $\varphi_1$ are the isomorphism of hom-sets:
		$$\varphi_0: \Hom(K_\ast(X^0), \F^0) \xrightarrow{\cong} \Hom(X^0, K(\F^0)) $$
		$$\varphi_1: \Hom(H_\ast(X^1), \F^1) \xrightarrow{\cong} \Hom(X^1, H(\F^1)).$$
		Conversely, assume that we are given a morphism $\theta: [X] \xrightarrow{[\theta^0, \theta^1]} E(H,K)(\Fc)$ where $\theta^0 \in \Hom(X^0, K(\F^0))$ and $\theta^1 \in \Hom(X^1, H(\F^1))$. If we set $\sgo =\varphi_0^{-1}(\theta^0)$ and $\sgi =\varphi_1^{-1}(\theta^1)$, then the map $\theta$ is displayed as a commutative diagram identical to the perimeter of(\ref{adj-diag}). However, the upper inner square involving $K(\alpha_\ast[X])$ and $K(\pif)$ does not commutes yet. Using the uniqueness of the adjoint-transpose map in the adjunction $K_\ast \dashv K$ with respect to the diagonal map $X^0 \to K(\Ub(\F^1))$, we see that this square does commute and is equal to the image under $K$ of the commutative square hereafter:
		\begin{align}
		\xy
		(0,18)*+{K_\ast(X^0)}="W";
		(0,0)*+{\Ub(H_\ast(X^1))}="X";
		(30,0)*+{\Ub(\F^1)}="Y";
		(30,18)*+{\F^0}="E";
		{\ar@{->}^-{\Ub(\sigma^{1})}"X";"Y"};
		{\ar@{->}^-{\alpha_\ast([X])}"W";"X"};
		{\ar@{->}^-{\sigma^0}"W";"E"};
		{\ar@{->}^-{\pi_{\G}}"E";"Y"};
		\endxy
		\end{align}
		The last diagram defines a map $E(H,K)_\ast([X]) \to \Fc$ given by $[\varphi_0^{-1}(\theta^0), \varphi_1^{-1}(\theta^1)]$.
		By the above, the function $[\sgo, \sgi] \mapsto [\varphi_1(\sgo), \varphi_2(\sgi)]$ provides the required isomorphisms of hom-sets.
	\end{proof}

\subsection{Proof of Proposition \ref{prop-retract-Quillen} \label{proof-prop-retract}} 
\begin{proof}
Let $RG$ and $RH$ be the respective right derived functors, where $H$ is a weakly invertible retraction of $G$ . The right homotopy $\Id_{\D} \xrightarrow{\sim} G\circ H$ gives a natural isomorphism after deriving it $\Id_{\Ho(\D)} \xrightarrow{\cong} RG\circ RH$ (see \cite{Hov-model}). Moreover, from the equality $H \circ G = \Id_{\C}$, one has $R G \circ R H \cong \Id_{\Ho(\C)}$, which proves that $RG$ and $RH$ are inverses to each other. It follows that $RG$ and $RH$ are equivalences of categories. This gives Assertion $(1)$. 
To prove the second assertion, consider a commutative diagram in $\modcat_r$ where the horizontal composites are identities:
\[
\xy
(0,15)*+{\A}="W";
(0,0)*+{\Ba}="X";
(50,0)*+{\Ba}="Y";
(50,15)*+{\A}="E";
{\ar@{->}^-{K}"W";"X"};
{\ar@{->}^-{K}"E";"Y"};
(25,15)*+{\C}="U";
(25,0)*+{\D}="V";
{\ar@{->}^-{\alpha^{1}}"X";"V"};
{\ar@{->}^-{\beta^{1}}"V";"Y"};
{\ar@{->}^-{G}"U";"V"};
{\ar@{->}^-{\beta^0}"U";"E"};
{\ar@{->}^-{\alpha^0}"W";"U"};
\endxy
\]
If $H: \D \to \C$ is a retraction of $G$, i.e., $H\circ G = \Id_{\C}$ then $T=\beta^0 \circ H \circ \alpha^{1}$ is a retraction of $K$. Moreover, if there is a homotopy $\tau : \Id_{\D} \to G\circ H$,
then the horizontal composite of $2$-morphisms $\Id_{\beta^1} \otimes \tau \otimes \Id_{\alpha^1} $ is a homotopy  whose domain and codomain are respectively $\beta^1 \circ \Id_{\D} \circ \alpha^1 = \Id_{\Ba}$ and $\beta^1 \circ G \circ H \circ \alpha^{1}$. A direct checking shows that $ \beta^1 \circ G \circ H \circ \alpha^{1}= K \circ T$,  which proves that $\Id_{\beta^1} \otimes \tau \otimes \Id_{\alpha^1} $ is a right homotopy $\Id_{\Ba} \to K \circ T$. It follows that the class $\Lci_w(\modcat_r)$ is closed under retracts.  It can be easily seen that $\Lci_w(\modcat_r)$ is also closed under composition. Let $G_i$ be composable maps in $\Lci_w(\modcat_r)$, with weakly invertible retraction $H_i$ and right homotopies $\tau_1: \Id \to G_i \circ H_i$, $i \in \{1,2\}$.   Then $H_1 \circ H_2$ is a weakly invertible retractions  of $G_2 \circ G_1$ with a right homotopy $ \Id \to (G_2 \circ G_1) \circ (H_1 \circ H_2)$ obtained as the vertical composite:
$ \Id \xrightarrow{\tau_2} \underbrace{G_2 \circ H_2}_{G_2 \circ \Id \circ H_2} \xrightarrow{\Id_{G_2} \otimes \tau_1 \otimes \Id_{H_2}} G_2 \circ G_1 \circ H_1 \circ H_2.$ If moreover $\tau : \Id_{\D} \to G\circ H$ is a homotopy relative to $G$, then it can be easily shown by inspection that $\Id_{\beta^1} \otimes \tau \otimes \Id_{\alpha^1} : \Id_{\Ba} \to K \circ T$ is a homotopy relative to $K$. The same conclusion holds for the homotopy $\Id \xrightarrow{(\Id_{G_2} \otimes \tau_1 \otimes \Id_{H_2}) \circ \tau_2} G_2 \circ G_1 \circ H_1 \circ H_2$ if we assume that $\tau_i$ is a homotopy relative to $G_i$, $i \in \{1,2\}$. The proof of the assertions involving the classes $\Rci_w(\modcat_r)$ and $\Rci_{wcor}(\modcat_r)$ goes the same way.
\end{proof}
\subsection{Proof of Lemma \ref{inj-lem}}\label{proof-inj-lem}
\begin{proof}[Proof of Lemma \ref{inj-lem}]
We will follow the same argument as in \cite{Bacard_QS} to prove the lemma. For Assertion $(1)$ we start by projecting the lifting problem in $\ag$ using the functor $\Pi^{1}$. This gives a lifting problem defined by $\sg^{1}$ and $\beta^{1}$. The assumption that $\sg \in  \lcof(\muaij) \cap \lwe(\muaij)$ implies that $\sg^{1}$ is a trivial cofibration and similarly since $\beta \in \lfib(\muaij)$ then $\beta^{1}$ is a fibration. The axiom of the model category $\ag$ gives a solution $s^{1}: \G^1 \to \Pa^{1}$ to this lifting problem. Part of $s^{1}$ being a solution gives an equality $\Ub(\gamma^{1})= \Ub(\beta^{1}) \circ \Ub(s^{1})$. Moreover,  $[\gamma]=[\gamma^0,\gamma^{1}]$ being a morphism in $\mua$ implies that $\pi_{\Qa} \circ \gamma^0 = \Ub(\gamma^{1}) \circ \pig$.

Now consider the map $\Ub(s^{1}) \circ \pig \in \Hom_{\M}(\G^0,\Ub(\Pa^{1}))$ and the map $\gamma^0 \in \Hom_{\M}(\G^0,\Qa^0)$. Then by the above, it is not hard to see that these maps complete the pullback data $$\Ub(\Pa^{1}) \xrightarrow{\Ub(\beta^{1})} \Ub(\Qa^{1}) \xleftarrow{\pi_{\Qa}} \Qa^0,$$ into a commutative square ($\pi_{\Qa} \circ \gamma^0= \Ub(\beta^{1}) \circ \Ub(s^{1}) \circ \pig$). Therefore, by the universal property of the pullback square there is a unique map: $\zeta:\G^0 \to \Ub(\Pa^{1}) \times_{\Ub(\Qa^{1})} \Qa^0$,
making everything compatible. In particular $\gamma^0$ and $\Ub(s^{1}) \circ \pig$  factor through $\zeta$.

Our original lifting problem in $\mua$ defined by $[\sigma]$ and $[\beta]$ is represented by a commutative cube in $\M$. If we unfold it, we find that everything commutes in the diagram hereafter:
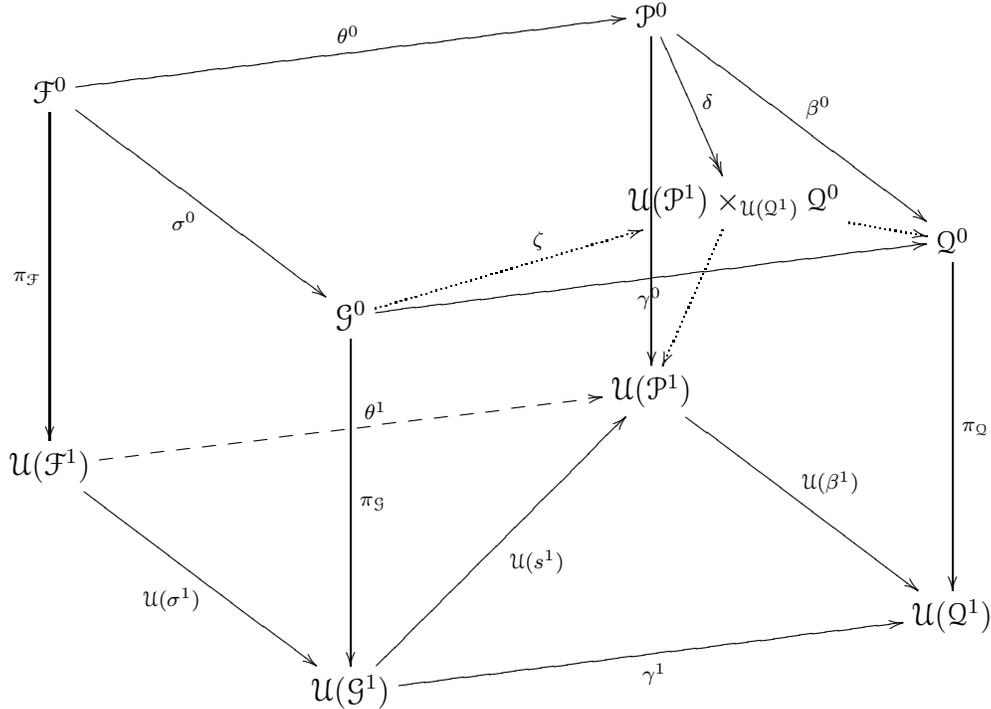
\[
\xy
(-60,20)*+{\F^0}="A";
(20,30)*+{\Pa^0}="B";
(-20,-10)*+{\G^0}="C";
(60,0)*+{\Qa^0}="D";
{\ar@{->}^{\theta^0}"A";"B"};
{\ar@{->}_{\sg^0}"A";"C"};
{\ar@{->}^{\beta^0}"B";"D"};
{\ar@{->}_{\gamma^0~~~~~~~}"C";"D"};
(-60,-30)*+{\Ub(\F^1)}="X";
(20,-20)*+{\Ub(\Pa^{1})}="Y";
(-20,-60)*+{\Ub(\G^1)}="Z";
(60,-50)*+{\Ub(\Qa^{1})}="W";
{\ar@{-->}^{\quad \quad \theta^{1}}"X";"Y"};
{\ar@{->}_{\pif}"A";"X"};
{\ar@{->}^{\Ub(\beta^{1})}"Y";"W"};
{\ar@{->}^{\pig}"C";"Z"};
{\ar@{->}_-{}"B";"Y"};
{\ar@{->}^{\pi_{\Qa}}"D";"W"};
{\ar@{->}_{\Ub(\sg^{1})}"X";"Z"};
{\ar@{->}_{\gamma^{1}}"Z";"W"};
(20,30)+(11,-25)*+{\Ub(\Pa^{1}) \times_{\Ub(\Qa^{1})} \Qa^0 }="E";
{\ar@{->>}^{\delta}"B";"E"};
{\ar@{.>}^{}"E";"D"};
{\ar@{.>}^{}"E";"Y"};
{\ar@{.>}^{\zeta}"C";"E"};
{\ar@{->}_-{\Ub(s^{1})}"Z";"Y"};
\endxy
\]
Thus we get a commutative square that corresponds to a lifting problem defined by the map $\sg^0:\F^0 \to \G^0$ and the map $\delta: \Pa^0 \to  \Ub(\Pa^{1}) \times_{\Ub(\Qa^{1})} \Qa^0 $:
\[
\xy
(0,18)*+{\F^0}="W";
(0,0)*+{\G^0}="X";
(30,0)*+{\Ub(\Pa^{1}) \times_{\Ub(\Qa^{1})} \Qa^0}="Y";
(30,18)*+{\Pa^0}="E";
{\ar@{->}^-{\zeta}"X";"Y"};
{\ar@{->}_-{\sg^0}"W";"X"};
{\ar@{->}^-{\theta^0}"W";"E"};
{\ar@{->}^-{\delta}"E";"Y"};
\endxy
\]
Now it suffices to observe that a solution to this lifting problem gives a solution to the original lifting problem. Indeed, if $s^0 : \G^0 \to \Pa^0$ is a solution to the last lifting problem, then $[s]=[s^0,s^{1}]: \Gc \to \Pc$ is a solution to the original lifting problem. 
Finally, it is clear that the last lifting problem defined by $\sg^0$ and $\delta$ has a  solution $s^0 \in \Hom_{\M}(\G^0,\Pa^0)$ since $\sg^0$ is a cofibration   and $\delta$ is a trivial fibration. This gives Assertion $(1)$.\\
For Assertion $(2)$ we proceed as follows.
Let  $\sigma=[\sg^0,\sg^{1}]: \Fc \to \Gc$ be a map in $\mua$.
Use the axiom of the model category $\ag$ to factor $\sigma^{1}$ as  trivial cofibration followed by a fibration, i.e.,  $\sg^{1}= r(\sg^{1}) \circ l(\sg^{1})$:
$$\F^1 \xrightarrow{\sg^{1}} \G^1 = \F^1\xhookrightarrow[\sim]{l(\sg^{1})}  E^{1} \xtwoheadrightarrow{r(\sg^{1})} \G^1.$$ 
The image under $\Ub$ of this factorization, gives a factorization  $\Ub(\sg^{1})= \Ub(r(\sg^{1})) \circ \Ub(l(\sg^{1}))$. Moreover the map  $\Ub(r(\sg^{1}))$ is a fibration in $\M$ since $\Ub$ preserves the fibrations.
Form the pullback square in $\M$ defined by the pullback data: $$\Ub(E^{1}) \xrightarrow{\Ub(r(\sg^{1}))} \Ub(\G^1) \xleftarrow{\pig} \G^0,$$
and let $p^{1} : \Ub(E^{1}) \times_{\Ub(\G^1)} \G^0 \to \G^0$ and $p_2:   \Ub(E^{1}) \times_{\Ub(\G^1)} \G^0  \to \Ub(\Ea^{1})$ be the canonical maps. Then $p^{1}$ is a fibration in $\M$ since the class of fibrations is closed under pullbacks.  
The universal property of the pullback square gives a unique map $\delta : \F^0 \to \Ub(E^{1}) \times_{\Ub(\G^1)} \G^0,$
such that everything below commutes. 
\[
\xy
(0,30)*+{\F^0}="W";
(0,0)*+{\Ub(\F^1)}="X";
(60,0)*+{\Ub(\G^1)}="Y";
(60,30)*+{\G^0}="E";
{\ar@{->}^-{\pi_{\F}}"W";"X"};
{\ar@{->}^-{\sigma^0}"W";"E"};
{\ar@{->}^-{\pi_{\G}}"E";"Y"};
(30,20)*+{\Ub(E^{1}) \times_{\Ub(\G^1)} \G^0}="U";
(30,0)*+{\Ub(E^{1})}="V";
{\ar@{->}^-{\Ub(l(\sg^{1}))}"X";"V"};
{\ar@{->>}^-{\Ub(r(\sg^{1}))}"V";"Y"};
{\ar@{->}^-{p_2}"U";"V"};
{\ar@{->>}^-{p^{1}}"U";"E"};
{\ar@{->}^-{\delta}"W";"U"};
\endxy
\]
We can factor the map $\delta$ as cofibration followed by a trivial fibration:
$$ \delta : \F^0 \to \Ub(E^{1}) \times_{\Ub(\G^1)} \G^0 = \F^0  \xhookrightarrow{a(\delta)}  m^0 \xtwoheadrightarrow[\sim]{b(\delta)}\Ub(E^{1}) \times_{\Ub(\G^1)} \G^0.$$
Let $\Ec=[\Ea^0,\Ea^{1},\pi_{\Ea}]$ be the object of $\mua$ defined by
$$ \Ea^0= m^0, \quad \Ea^{1}=E^{1}, \quad \pi_{\Ea}=p_2 \circ b(\delta).$$
We  have a map $(i : \Fc \to \Ec) \in  \lcof(\muaij) \cap \lwe(\muaij)$ given by the couple $[a(\delta),  l(\sigma^{1})]$ and a map $(p : \Ec \to \Gc) \in \lfib(\muaij)$ given by the couple $[p^{1} \circ b(\delta),  r(\sigma^{1})]$. Clearly, we have $\sg = p \circ i$, which gives Assertion $(2)$. 

\end{proof}

\subsection{Proof of Lemma \ref{proj-lem-lift}}\label{proof-proj-lem-lift}
To prove the lemma we will need the following result which is well known in the theory of model categories. We shall refer the reader to Hirschhorn \cite[Ch. 7]{Hirsch-model-loc}.
\begin{lem}\label{lem-exist-map}
Let $\D$ be a model category.\\
If $g: X \to Y$ is a weak equivalence between fibrant objects in $\D$ and $C$ is a cofibrant object of $\D$, then $g$ induces an isomorphisms of the sets of homotopy classes of maps: 
$g_\ast: \pi(C,X) \to \pi(C,Y)$. \\
In particular there is a map $C \to X$ in $\D$ if and only if there is a map $C \to Y$ in $\D$.
\end{lem}
\begin{proof}[Proof of Lemma \ref{proj-lem-lift}]
	The proof of this lemma is inspired from the work of Renaudin \cite{Renaudin_mod}. 
	Given a lifting problem as in the lemma, the commutative square defines a morphism $\tau \in \Hom(\Fc, \Qc)$ with $\tau= \beta \circ \theta = \gamma \circ \sigma$. Let $\mua_{/\Qc}$ be the over category whose objects consists of pairs $(\Ec, \alpha)$, where $\alpha: \Ec \to \Qc$ is a morphism in $\mua$. The morphisms are the obvious ones. This category inherits a model structure called the ``over model structure'' from the projective model structure on $\mua$ (see \cite{Hov-model}). The object $(\Qc, \Id_{\Qc})$ is the terminal object. From the commutative square defining the lifting problem we have:
	\begin{itemize}
		\item an object $(\Gc, \gamma) \in\mua_{/\Qc}$,
		\item an object $(\Pc, \beta) \in\mua_{/\Qc}$,
		\item an object $(\Fc, \tau) \in\mua_{/\Qc}$,
		\item a map $\theta: (\Fc, \tau) \to (\Pc, \beta) $ in $\mua_{/\Qc}$,
		\item a map $\sg: (\Fc, \tau) \to (\Gc, \gamma)$ in $\mua_{/\Qc}$.
	\end{itemize}
	With the category $\mua_{/\Qc}$, introduce the under category:
	$$(\Fc, \tau)/\mua_{/\Qc} = (\Fc, \tau) \downarrow \mua_{/\Qc}.$$
	 In simple terms, this is the category of factorizations of the morphism $\tau: \Fc \to \Qc$. By the above observations, we have two objects of this under category $(\Fc, \tau)/\mua_{/\Qc}$:
	\begin{itemize}
		\item $[(\Pc, \beta), \theta]$, where $\theta :(\Fc, \tau) \to (\Pc, \beta)$,
		\item $[(\Gc, \gamma), \sg]$, where $\sg : (\Fc, \tau) \to (\Gc, \gamma)$.
	\end{itemize}
It's important to observe that a solution to the original lifting problem is equivalent to a morphism $[(\Gc, \gamma), \sg] \to [(\Pc, \beta), \theta]$ in $(\Fc, \tau)/\mua_{/\Qc}$. We are therefore reduced to show that some hom-set is nonempty, namely: $\Hom_{}([(\Gc, \gamma), \sg], [(\Pc, \beta), \theta] ) \neq \emptyset$. To prove this, we will work in the model category $(\Fc, \tau)/\mua_{/\Qc}$ and use an abstract argument, which is well known in the theory of model categories as explained below. Recall that
	the under category $(\Fc, \tau)/\mua_{/\Qc}$ inherits also a model structure - the \emph{under model structure} - from the previous model structure on the over category $\mua_{/\Qc}$. For the rest of the argument we need to have a precise description of the cofibrant and the fibrant objects in the model category $(\Fc, \tau)/\mua_{/\Qc}$.
	\begin{enumerate}
		\item A fibrant object in $(\Fc, \tau)/\mua_{/\Qc}$ is an object $[(\Ec,p),r]$ where  $p:\Ec \to \Qc$  is a level-wise fibration (=projective fibration), $r: \Fc \to \Ec$ is any map, such that we have a factorization of $\tau :\Fc \to \Qc$:
		$\tau = \Fc \xrightarrow{r}\Ec \xtwoheadrightarrow{p} \Qc.$
		\item A cofibrant object in $(\Fc, \tau)/\mua_{/\Qc}$ is an object $[(\Bc,q), i]$ where $q:\Bc \to \Qc$ is any map and $i : \Fc \hookrightarrow \Bc$ is a projective cofibration  such that $\tau= q \circ i$.
	\end{enumerate}
	Then by assumption on the lifting problem, the object $[(\Pc, \beta), \theta]$ is fibrant because $\theta$ is in particular a level-wise fibration, while $[(\Gc, \gamma), \sg]$ is cofibrant since $\sg$ is a projective fibration. With these observations, we can build our lift as follows. Like in the injective case, if we project the lifting problem using the functor $\Pi^{1}$ we get a lifting problem defined by $\sg^{1}$ and $\beta^{1}$. The later problem admits a solution $s^{1}: \G^1\to \Pa^{1}$ since $\sg^{1}$ is a trivial cofibration and $\beta^{1}$ is a fibration. With the same reasoning as in proof of Lemma \ref{inj-lem}, keeping the same notation, we find that everything below commutes. 
	
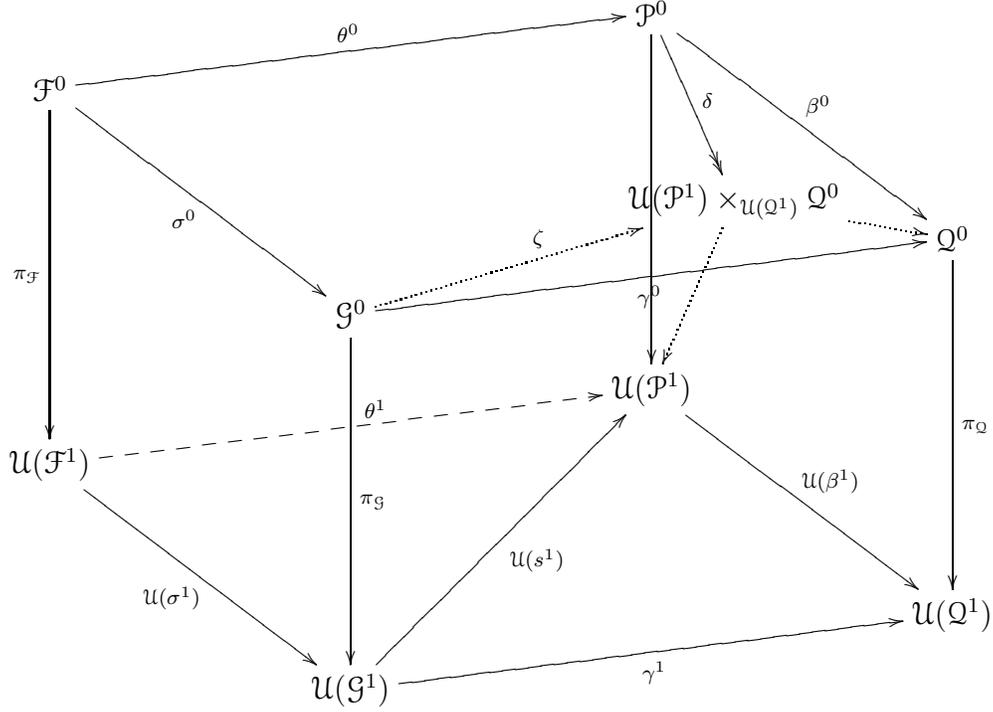
\[
	\xy
	(-60,20)*+{\F^0}="A";
	(20,30)*+{\Pa^0}="B";
	(-20,-10)*+{\G^0}="C";
	(60,0)*+{\Qa^0}="D";
	{\ar@{->}^{\theta^0}"A";"B"};
	{\ar@{->}_{\sg^0}"A";"C"};
	{\ar@{->}^{\beta^0}"B";"D"};
	{\ar@{->}_{\gamma^0~~~~~~~}"C";"D"};
	(-60,-30)*+{\Ub(\F^1)}="X";
	(20,-20)*+{\Ub(\Pa^{1})}="Y";
	(-20,-60)*+{\Ub(\G^1)}="Z";
	(60,-50)*+{\Ub(\Qa^{1})}="W";
	{\ar@{-->}^{\quad \quad \theta^{1}}"X";"Y"};
	{\ar@{->}_{\pif}"A";"X"};
	{\ar@{->}^{\Ub(\beta^{1})}"Y";"W"};
	{\ar@{->}^{\pig}"C";"Z"};
	{\ar@{->}_-{}"B";"Y"};
	{\ar@{->}^{\pi_{\Qa}}"D";"W"};
	{\ar@{->}_{\Ub(\sg^{1})}"X";"Z"};
	{\ar@{->}_{\gamma^{1}}"Z";"W"};
	(20,30)+(11,-25)*+{\Ub(\Pa^{1}) \times_{\Ub(\Qa^{1})} \Qa^0 }="E";
	{\ar@{->>}^{\delta}"B";"E"};
	{\ar@{.>}^{}"E";"D"};
	{\ar@{.>}^{}"E";"Y"};
	{\ar@{.>}^{\zeta}"C";"E"};
	{\ar@{->}_-{\Ub(s^{1})}"Z";"Y"};
	\endxy
	\]
	Let  $\Ec= [\Ea^0,\Ea^{1}, \pi_{\Ea}]$ be  the object defined within the pullback square, where $\Ea^0= \Ub(\Pa^{1}) \times_{\Ub(\Qa^{1})} \Qa^0$ is the pullback object, $\Ea^{1}= \Pa^{1}$ and $\pi_{\Ea}=\Ub(\beta^1)^*(\piq)$ is the base change of $\piq$: $$\Ub(\Pa^{1}) \times_{\Ub(\Qa^{1})} \Qa^0 \to \Ub(\Pa^{1}).$$
	As everything commutes in the above cube  we get the following. 
	\begin{itemize}
		\item The pullback square defines a map $\chi:\Ec \to \Qc$ given by the pullback of $\Ub(\beta^{1})$ and $\beta^{1}$. Since $\beta^{1}$ is a fibration and $\Ub$ is right Quillen, then $\Ub(\beta^{1})$  and its pullback are fibrations in $\M$. Thus $\chi:\Ec \to \Qc$  is a level-wise fibration. Consequently the object $(\Ec,\chi)$ is fibrant in over category $\mua_{/\Qc}$.
		\item There is a map $\xi : \Pc \to \Ec$ given by $\delta$  and $\Id_{\Pa^{1}}$. By assumption $\delta$ is a weak equivalence, therefore $\xi : \Pc \xrightarrow{\sim} \Ec$ is a level-wise weak equivalence in $\muapj$.
		\item There is also a map $\tld{s}: \Gc \to \Ec$ defined by $\zeta: \G^0 \to \Ub(\Pa^{1}) \times_{\Ub(\Qa^{1})} \Qa^0$ and the previous solution $s^{1}: \G^1 \to \Pa^{1}$, i.e, $\tld{s}= [\zeta, s^1]$.
	\end{itemize}
	A key observation is that these various maps fit in two factorizations of $\tau : \Fc \to \Qc$: 
	$$(\tau: \Fc \to \Qc)= \Fc \xrightarrow{\theta} \Pc \xrightarrow[\sim]{\xi} \Ec \xtwoheadrightarrow{\chi} \Qc,$$
	$$(\tau: \Fc \to \Qc)=\Fc \xhookrightarrow{\sigma} \Gc \xrightarrow{\tld{s}} \Ec \xtwoheadrightarrow{\chi} \Qc.$$
	These factorizations determine two maps in the under category $(\Fc, \tau)/\mua_{/\Qc}$:
	\begin{itemize}
		\item $\xi: [(\Pc, \beta), \theta] \xrightarrow{\sim}  [(\Ec, \chi),\xi \circ  \theta]$. This map is a weak equivalence between fibrant objects in the under model category $(\Fc, \tau)/\mua_{/\Qc}$.
		\item $\tld{s} : [(\Gc, \gamma), \sg] \to [(\Ec, \chi),\xi \circ  \theta]$. The source of this map is a cofibrant object in the under model category $(\Fc, \tau)/\mua_{/\Qc}$.
	\end{itemize}
	Given the cofibrant object $[(\Gc, \gamma), \sg]$ and the weak equivalence $\xi$ between  fibrant objects, we know from Lemma \ref{lem-exist-map} that there is a map $[(\Gc, \gamma), \sg] \to [(\Ec, \chi),\xi \circ  \theta]$ if and only if there is a map $[(\Gc, \gamma), \sg] \to [(\Pc, \beta), \theta]$, that is:
	$$\Hom([(\Gc, \gamma), \sg],[(\Ec, \chi),\xi \circ  \theta]) \neq \emptyset \Longleftrightarrow  \Hom([(\Gc, \gamma), \sg],[(\Pc, \beta), \theta]) \neq \emptyset.$$
	The hom-set $\Hom([(\Gc, \gamma), \sg],[(\Ec, \chi),\xi \circ  \theta])$ is nonempty since it contains $\tld{s}$,  therefore the other hom-set is nonempty and there exists  an element $s \in \Hom([(\Gc, \gamma), \sg],[(\Pc, \beta), \theta])$ which is automatically a solution to our lifting problem. 
	
\end{proof}

\subsection{Proof of Lemma \ref{proj-lem-fact}}\label{proof-proj-lem-fact}
\noindent For a matter of clarity we will put ``p.b'' inside a commutative square for a pull-back square and ``p.o'' for a pushout square.
\begin{proof}[Proof of Lemma \ref{proj-lem-fact}]
	Let  $\sigma=[\sg^0,\sg^{1}]: \Fc \to \Gc$ be a map in $\mua$.
	Use the axiom of the model category $\ag$ to factor $\sigma^{1}$ as  trivial cofibration followed by a fibration, i.e.,  $\sg^{1}= r(\sg^{1}) \circ l(\sg^{1})$:
	$$\F^1 \xrightarrow{\sg^{1}} \G^1 = \F^1\xhookrightarrow[\sim]{l(\sg^{1})}  E^{1} \xtwoheadrightarrow{r(\sg^{1})} \G^1.$$ 
	The image under $\Ub$ of this factorization gives a factorization: $$\Ub(\sg^{1})= \Ub(r(\sg^{1})) \circ \Ub(l(\sg^{1})).$$ Moreover, the map  $\Ub(r(\sg^{1}))$ is a fibration in $\M$ since $\Ub$ preserves the fibrations.
	Let $ {Q^0}= \Ub(E^{1}) \times_{\Ub(\G^1)} \G^0 $ be the pullback-object  obtained from  the pullback square defined by the data: $\Ub(E^{1}) \xrightarrow{\Ub(r(\sg^{1}))} \Ub(\G^1) \xleftarrow{\pig} \G^0$.
	Denote by  $p^{1} : {Q^0} \to \G^0$ and $p_2:   {Q^0}  \to \Ub(\Ea^{1})$ the canonical maps. Then $p^{1}$ is a fibration in $\M$ since the class of fibrations is closed under pullbacks.  
	The universal property of the pullback square gives a unique map $\delta : \F^0 \to {Q^0}$
	such that everything below commutes. 
	\[
	\xy
	(0,30)*+{\F^0}="W";
	(0,0)*+{\Ub(\F^1)}="X";
	(60,0)*+{\Ub(\G^1)}="Y";
	(60,30)*+{\G^0}="E";
	(45,15)*+{\tx{p.b}}="PB";
	{\ar@{->}^-{\pi_{\F}}"W";"X"};
	{\ar@{->}^-{\sigma^0}"W";"E"};
	{\ar@{->}^-{\pi_{\G}}"E";"Y"};
	(30,20)*+{{Q^0}}="U";
	(30,0)*+{\Ub(E^{1})}="V";
	{\ar@{->}^-{\Ub(l(\sg^{1}))}"X";"V"};
	{\ar@{->>}^-{\Ub(r(\sg^{1}))}"V";"Y"};
	{\ar@{->}^-{p_2}"U";"V"};
	{\ar@{->>}^-{p^{1}}"U";"E"};
	{\ar@{->}^-{\delta}"W";"U"};
	\endxy
	\]
	Now we can factor the map $\delta$ as cofibration followed by a trivial fibration:
	$$ \delta : \F^0 \to {Q^0} = \F^0  \xhookrightarrow{a_\delta}  m^0 \xtwoheadrightarrow[\sim]{b_\delta} {Q^0}.$$
	The map $\pif \in \Hom_{\M}(\F^0,\Ub(\F^1))$ is equivalent to a morphism $\varphi(\pif) \in \Hom_{\ag}(\Fb m^0,\F^1)$, where $\varphi : \Hom_{\M}(x,\Ub(y))  \xrightarrow{\cong} \Hom_{\ag}(\Fb x, y)$ is the isomorphism of the adjunction $\Fb \dashv \Ub$. The previous commutative diagram in $\M$ corresponds by adjunction to the following commutative diagram in $\ag$:
	\[
	\xy
	(0,30)+(-3,3)*+{\Fb \F^0}="W";
	(0,0)+(-3,-3)*+{\F^1}="X";
	(60,0)+(3,-3)*+{\G^1}="Y";
	(60,30)+(3,3)*+{\Fb \G^0}="E";
	{\ar@{->}_-{\varphi(\pi_{\F})}"W";"X"};
	{\ar@{->}^-{\Fb \sigma^0}"W";"E"};
	{\ar@{->}^-{\varphi(\pi_{\G})}"E";"Y"};
	(30,20)+(3,0)*+{\Fb {Q^0}}="U";
	(15,25)+(-5,2)*+{\Fb m^0}="M";
	(30,0)+(0,-3)+(3,0)*+{E^{1}}="V";
	{\ar@{^(->}^-{l(\sg^{1})}_{\sim}"X";"V"};
	{\ar@{->>}^-{r(\sg^{1})}"V";"Y"};
	{\ar@{->}^-{\varphi(p_2)}"U";"V"};
	{\ar@{->>}^-{\Fb p^{1}}"U";"E"};
	{\ar@{^(->}_-{}"W";"M"};
	{\ar@{->}_-{}"M";"U"};
	\endxy
	\]
	 Since $\Fb$ is a left Quillen functor, it preserves the (trivial) cofibrations, therefore $\Fb(a_\delta)$ is a cofibration in $\ag$. Let $R^{1}=\F^1 \cup^{\Fb \F^0} \Fb m^0$ be the pushout-object obtained from the pushout data: $\F^1 \xleftarrow{\varphi(\pi_{\F})} \Fb \F^0 \xhrw{\Fb(a_\delta)} \Fb m^0 $. The canonical map $\F^1 \hrw R^{1}$ is also a cofibration as the cobase change of the cofibration $\Fb(a_\delta)$. Moreover, the universal property of the pushout square gives a unique map $R^{1} \to E^{1}$ that we can factor as a cofibration followed by a trivial fibration: $R\xhrw{}{E^1}' \xtwoheadrightarrow{\sim} E^{1}$. Putting all together we get a diagram in $\ag$ in which everything commutes.
	\[
	\xy
	(0,30)+(-3,3)*+{\Fb \F^0}="W";
	(0,0)+(-3,-3)*+{\F^1}="X";
	(60,0)+(3,-3)*+{\G^1}="Y";
	(60,30)+(3,3)*+{\Fb \G^0}="E";
	{\ar@{->}_-{\varphi(\pi_{\F})}"W";"X"};
	{\ar@{->}^-{\Fb \sigma^0}"W";"E"};
	{\ar@{->}^-{\varphi(\pi_{\G})}"E";"Y"};
	(30,20)+(3,0)*+{\Fb {Q^0}}="U";
	(15,25)+(-5,2)*+{\Fb m^0}="M";
	(15,10)+(0,-2)+(-5,0)*+{R^{1}}="H";
	(20,5)+(0,-2)*+{{E^1}'}="G";
	(30,0)+(0,-3)+(3,0)*+{E^{1}}="V";
	{\ar@{^(->}^-{l(\sg^{1})}_{\sim}"X";"V"};
	{\ar@{->>}^-{r(\sg^{1})}"V";"Y"};
	{\ar@{->}^-{\varphi(p_2)}"U";"V"};
	{\ar@{->>}^-{\Fb p^{1}}"U";"E"};
	{\ar@{^(->}_-{}"W";"M"};
	{\ar@{->}_-{}"M";"U"};
	(3.5,15)*+{\tx{p.o}}="PO";
	{\ar@{^(->}_-{}"X";"H"};
	{\ar@{->}_-{}"M";"H"};
	{\ar@{^(->}_-{}"H";"G"};
	{\ar@{->>}^-{\sim}"G";"V"};
	\endxy
	\]
	A key observation is that the composite of cofibrations $(\F^1 \hrw R^{1} \hrw {E^1}')$ is a weak equivalence by $3$-for-$2$ since both $l(\sg^{1}):\F^1 \xrightarrow{\sim} E^{1}$ and ${E^1}' \xtwoheadrightarrow{\sim} E^{1}$ are weak equivalences. It follows that the composite $\F^1 \hrw {E^1}'$ is in fact a trivial cofibration that will be denoted henceforth as $i^{1}$. 
	Moreover, since $\Ub$ is a right Quillen functor, the image under $\Ub$ of the trivial fibration $({E^1}' \xtwoheadrightarrow{\sim} E^{1})$  is a trivial fibration $\Ub({E^1}') \xtwoheadrightarrow{\sim} \Ub(E^{1})$ in $\M$. Let ${Q^0}' = \Ub({E^1}') \times_{\Ub(E^{1})} {Q^0}$ be the pullback-object obtained from the pullback data: $ \Ub({E^1}') \xtwoheadrightarrow{\sim} \Ub(E^{1}) \xleftarrow{p_2} {Q^0}$. The canonical map ${Q^0}'  \xtwoheadrightarrow{\sim} {Q^0}$ is a trivial fibration as the base change of the trivial fibration $\Ub({E^1}') \xtwoheadrightarrow{\sim} \Ub(E^{1})$. From the last diagram displayed in the category $\ag$, if we go back to the category $\M$ by adjunction, we find a commutative diagram where everything commutes and where we have omitted the object $\Ub(R^{1})$ for simplicity:
	\[
	\xy
	(0,30)+(-3,3)+(-5,0)*+{\F^0}="W";
	(0,0)+(-3,-3)+(-5,0)*+{\Ub(\F^1)}="X";
	(60,0)+(3,-3)*+{\Ub(\G^1)}="Y"; 
	(60,30)+(3,3)*+{\G^0}="E";
	{\ar@{->}_-{\pi_{\F}}"W";"X"};
	{\ar@{->}^-{\sigma^0}"W";"E"};
	{\ar@{->}^-{\pi_{\G}}"E";"Y"};
	(30,20)+(3,0)*+{{Q^0}}="U";
	(15,25)+(-5,2)*+{m^0}="M";
	(15,10)+(0,-2)+(-5,0)*+{\Ub ({E^1}')}="H";
	(20,17)*+{{Q^0}'}="G";
	(45,15)*+{\tx{p.b}}="PB";
	(25,10)*+{\tx{p.b}}="PB2";
	(30,0)+(0,-3)+(3,0)*+{\Ub(E^{1})}="V";
	{\ar@{->}^-{\Ub(l(\sg^{1}))}"X";"V"};
	{\ar@{->>}^-{\Ub(r(\sg^{1}))}"V";"Y"};
	{\ar@{->}^-{p_2}"U";"V"};
	{\ar@{->>}^-{p^{1}}"U";"E"};
	{\ar@{^(->}_-{}"W";"M"};
	{\ar@{->>}^-{\sim}"M";"U"};
	{\ar@{->}^-{\Ub(i^{1})}"X";"H"};
	{\ar@{->}_-{}"M";"H"};
	{\ar@{->>}^-{\sim}"H";"V"};
	{\ar@{->>}^-{\sim}"G";"U"};
	{\ar@{.>}^-{}"M";"G"};
	{\ar@{->}^-{}"G";"H"};
	\endxy
	\]
	In this last diagram, the map $m^0 \to \Ub({E^1}')$ is adjoint to the composite $(\Fb m^0 \to R^{1} \to {E^1}')$. The universal property of the pullback gives a unique (dotted) map  $\tau: m^0 \to {Q^0}'$. It's important to notice that $\tau$ is a weak equivalence by $3$-for-$2$ with respect to the commutative triangle above involving $m^0, {Q^0}$ and ${Q^0}'$. Another important aspect is that ``the pullback of the pullback is a pullback'', therefore the map ${Q^0}' \to \Ub({E^1}')$ is a base change of $\pig$ along the composite $\Ub({E^1}') \to \Ub(E^{1}) \to \Ub(\G^1)$. Put differently, the commutative square bounded by the object ${Q^0}', \G^0, \Ub(\G^1)$ and $\Ub({E^1}')$ is a pullback square. Let $j^{1}: {E^1}' \to \G^1$ be the composite fibration $({E^1}' \xtwoheadrightarrow{\sim} E^{1} \xtwoheadrightarrow{r(\sg^{1})} \G^1)$ and let $\Ec=[\Ea^0,\Ea^{1},\pi_{\Ea}]$ be the object of $\mua$ defined by:
	$$ \Ea^0= m^0, \quad \Ea^{1}={E^1}', \quad \pi_{\Ea}=(m^0 \to \Ub({E^1}')).$$
	\begin{itemize}
		\item We  have a map $i : \Fc \to \Ec$ given by the couple $[a_\delta,  i^{1}]$. By construction $i$ is a projective cofibration in $\mua$. Moreover, $i^{1}$ is a weak equivalence in $\ag$ which means that $i$ is a left weak equivalence. Thus $i  \in  \lcof(\muapj) \cap \lwe(\muapj)$.
		\item We also have a map $j : \Ec \to \Gc$  given by the couple $[p^{1} \circ b_\delta,  j^{1}]$. By construction $j$ is level-wise fibration such that the universal map $(\Ea^0 \to \Ub(\Ea^{1}) \times_{\Ub(\G^1)} \G^0 ) = (m^0 \to {Q^0}')$ is a weak equivalence. Thus $j \in \lfib(\muapj)$. 
	\end{itemize}
	Clearly, one has $\sg= j \circ i$ and the lemma follows.
\end{proof}
\addcontentsline{toc}{section}{References}

\bibliographystyle{plain}
\bibliography{Bibliography}

\begin{thebibliography}{10}

\bibitem{Bacard_QS}
H.~{Bacard}.
\newblock {Quillen-Segal Algebras and Stable homotopy theory.}
\newblock {\em {Higher Structures}}, 4(1):57--114, 2020.

\bibitem{Barton_phd}
R.~W. {Barton}.
\newblock {A model 2-category of enriched combinatorial premodel categories}.
\newblock {\em arXiv e-prints}, April 2020.
\newblock \url{https://arxiv.org/abs/2004.12937}.

\bibitem{Bergner_holim}
J.~{Bergner}.
\newblock {Homotopy limits of model categories and more general homotopy
  theories.}
\newblock {\em {Bull. Lond. Math. Soc.}}, 44(2):311--322, 2012.

\bibitem{Bergner_hocolim_mod}
J.~{Bergner}.
\newblock {Homotopy colimits of model categories.}
\newblock In {\em {An Alpine expedition through algebraic topology. Proceedings
  of the fourth Arolla conference on algebraic topology, Arolla, Switzerland,
  August 20-25, 2012}}, pages 31--37. Providence, RI: American Mathematical
  Society (AMS), 2014.

\bibitem{Harpaz_lax}
Y.~Harpaz.
\newblock Lax limits of model categories, 2019.
\newblock \url{https://arxiv.org/abs/1902.04867}.

\bibitem{Hirsch-model-loc}
P.~{Hirschhorn}.
\newblock {\em {Model categories and their localizations.}}, volume~99.
\newblock Providence, RI: American Mathematical Society (AMS), 2003.

\bibitem{Hovey_Arr}
M.~{Hovey}.
\newblock {Smith ideals of structured ring spectra}.
\newblock \url{https://arxiv.org/pdf/1401.2850v1}.

\bibitem{Hov-model}
M.~Hovey.
\newblock {\em Model categories}, volume~63 of {\em Mathematical Surveys and
  Monographs}.
\newblock American Mathematical Society, Providence, RI, 1999.

\bibitem{Joy-Tier}
A.~Joyal and M.~Tierney.
\newblock Strong stacks and classifying spaces.
\newblock In {\em Category theory ({C}omo, 1990)}, volume 1488 of {\em Lecture
  Notes in Math.}, pages 213--236. Springer, Berlin, 1991.

\bibitem{Kelly-doctrine}
G.~M. {Kelly}.
\newblock {Doctrinal adjunction.}
\newblock {Category Sem., Proc., Sydney 1972/1973, Lect. Notes Math. 420,
  257-280 (1974).}, 1974.

\bibitem{Lack_model_2_cat}
S.~{Lack}.
\newblock {A Quillen model structure for 2-categories.}
\newblock {\em {\(K\)-Theory}}, 26(2):171--205, 2002.

\bibitem{Leinster_higher_op}
T.~{Leinster}.
\newblock {\em {Higher operads, higher categories.}}, volume 298.
\newblock Cambridge: Cambridge University Press, 2004.

\bibitem{Maclane_Moer_sh}
S.~{Mac Lane} and I.~{Moerdijk}.
\newblock {\em {Sheaves in geometry and logic: a first introduction to topos
  theory.}}
\newblock New York etc.: Springer-Verlag, 1992.

\bibitem{Quillen_HA}
D.~{Quillen}.
\newblock {\em {Homotopical algebra.}}, volume~43.
\newblock Springer, Cham, 1967.

\bibitem{Renaudin_mod}
O.~{Renaudin}.
\newblock {Theories homotopiques de Quillen combinatoires et derivateurs de
  Grothendieck}.
\newblock 2006.
\newblock \url{https://arxiv.org/pdf/math/0603339.pdf}.

\bibitem{Rezk_mod_top}
C.~Rezk.
\newblock Toposes and homotopy toposes.
\newblock Available on the author's webpage.

\bibitem{Sch-Sh-equiv}
S.~{Schwede} and B.~{Shipley}.
\newblock {Equivalences of monoidal model categories.}
\newblock {\em {Algebr. Geom. Topol.}}, 3:287--334, 2003.

\bibitem{Shulman_compose_der_func}
M.~{Shulman}.
\newblock {Comparing composites of left and right derived functors.}
\newblock {\em {New York J. Math.}}, 17:75--125, 2011.

\bibitem{To_hall}
B.~{To\"en}.
\newblock {Derived Hall algebras.}
\newblock {\em {Duke Math. J.}}, 135(3):587--615, 2006.

\bibitem{Vistoli_fib}
A.~{Vistoli}.
\newblock {Notes on Grothendieck topologies, fibered categories and descent
  theory}.
\newblock {\em arXiv Mathematics e-prints}.
\newblock \url{https://arxiv.org/abs/math/0412512}.

\bibitem{Weibel_HA}
C.~{Weibel}.
\newblock {\em {An introduction to homological algebra.}}, volume~38.
\newblock Cambridge Univ. Press, 1995.

\end{thebibliography}

\end{document}